\documentclass[final]{alt2026} 

\usepackage[final,nopatch=footnote]{microtype}
\usepackage{graphicx}
\usepackage{caption}
\usepackage{subcaption}
\usepackage{booktabs} 
\usepackage{hyperref}
\usepackage{pifont}
\newcommand{\cmark}{\ding{51}}%
\newcommand{\xmark}{\ding{55}}%

\usepackage{array}
\newcolumntype{M}[1]{>{\centering\arraybackslash}m{#1}}

\usepackage{cellspace} 
\usepackage{color, colortbl}	
\definecolor{Gray}{gray}{0.7}

\usepackage{wrapfig}

\usepackage{enumitem}

\usepackage{dsfont}
\usepackage[utf8]{inputenc} 
\usepackage[T1]{fontenc}    
\usepackage{hyperref}       
\usepackage{url}            
\usepackage{booktabs}       
\usepackage{amsfonts}       
\usepackage{nicefrac}       
\usepackage{microtype}      
\usepackage{xcolor}         
\usepackage{mathtools}

\usepackage{multicol}

\usepackage{algorithm} 
\usepackage{algorithmic}
\usepackage{natbib}
\usepackage{graphicx}
\usepackage{xcolor}
\usepackage{float}
\usepackage{nicematrix}
\usepackage{dsfont}
\usepackage{makecell}
\usepackage{pifont}
\usepackage{multirow}

\usepackage{pgfplots}
\pgfplotsset{compat=1.18}

\newcommand{\suchthat}{\;\ifnum\currentgrouptype=16 \middle\fi|\;}

\DeclareMathOperator*{\argmin}{arg\,min}
\usepackage{tikz}

\usepackage[capitalize,noabbrev]{cleveref}

\newtheorem{assumption}{Assumption}

\newtheorem*{theorem*}{Theorem}
\newtheorem*{proposition*}{Proposition}
\newtheorem*{lemma*}{Lemma}

\usepackage[textsize=tiny]{todonotes}

\title[Online Markov Decision Processes with Terminal Law Constraints]{Online Markov Decision Processes with Terminal Law Constraints}
\usepackage{times}

\altauthor{%
 \Name{Bianca {Marin Moreno}} \Email{bianca.marin-moreno@inria.fr}\\
 \addr Univ. Grenoble Alpes, Inria, CNRS, Grenoble INP, LJK, 38000 Grenoble, France \\
 \addr EDF Lab, 7 bd Gaspard Monge, 91120 Palaiseau, France \\
 \addr FiME Laboratoire de Finance des Marches de l'Energie - Dauphine, CREST, EDF R$\&$D
 \AND
 \Name{Margaux {Brégère}} \Email{margaux.bregere@edf.fr}\\
  \addr EDF Lab, 7 bd Gaspard Monge, 91120 Palaiseau, France \\
 \addr Sorbonne Université LPSM, Paris, France 
 \AND
 \Name{Pierre {Gaillard}} \Email{pierre.gaillard@inria.fr}\\
 \addr Univ. Grenoble Alpes, Inria, CNRS, Grenoble INP, LJK, 38000 Grenoble, France 
 \AND
 \Name{Nadia {Oudjane}} \Email{nadia.oudjane@edf.fr}\\
 \addr EDF Lab, 7 bd Gaspard Monge, 91120 Palaiseau, France \\
 \addr FiME Laboratoire de Finance des Marches de l’Energie - Dauphine, CREST, EDF R$\&$D%
}

\begin{document}

\maketitle

\begin{abstract}%
Traditional reinforcement learning usually assumes either episodic interactions with resets or continuous operation to minimize average or cumulative loss. While episodic settings have many theoretical results, resets are often unrealistic in practice. The infinite-horizon setting avoids this issue but lacks non-asymptotic guarantees in online scenarios with unknown dynamics. In this work, we move towards closing this gap by introducing a reset-free framework called the \emph{periodic} framework, where the goal is to find \emph{periodic policies}: policies that not only minimize cumulative loss but also return the agents to their initial state distribution after a fixed number of steps. We formalize the problem of finding optimal periodic policies and identify sufficient conditions under which it is well-defined for tabular Markov decision processes. To evaluate algorithms in this framework, we introduce the \emph{periodic regret}, a measure that balances cumulative loss with the terminal law constraint. We then propose the first algorithms for computing periodic policies in two multi-agent settings and show they achieve sublinear periodic regret of order $\tilde{O}(T^{3/4})$. This provides the first non-asymptotic guarantees for reset-free learning in the setting of $M$ homogeneous agents, for $M > 1$.
\end{abstract}


\section{Introduction}\label{sec:introduction}

Markov Decision Processes (MDPs) have long been used to model Reinforcement Learning (RL) problems, where one or more agents interact with an environment by taking sequential actions, leading to state transitions and observed losses. Typically, the agents interact with the environment over episodes of fixed length or until a termination condition is met, after which they are reset to an initial state drawn from a fixed initial distribution, which is known as the episodic RL setting \citep{Rosenberg2019, jin_bandit, luo_dilated_bonus}. Alternatively, the agents may operate continuously without resets, and the learner aims to minimize either the long-term average loss (infinite-horizon average-reward setting \citep{UCRL-2}) or the total loss over a specified number of steps (infinite-horizon setting \citep{neu_2010}). The episodic framework for online MDPs is well studied, with many algorithms giving theoretical guarantees for adversarial losses and unknown dynamics, but with the disadvantage of requiring agents to be restarted. On the other hand, the infinite-horizon framework is more realistic but more difficult to solve, as finite-time guarantees for online tasks have only been shown when the dynamics are known \citep{dai_ftrl}.

In an effort to bridge the gap between the well-studied episodic setting and the more realistic but less theoretically developed infinite-horizon setting, we introduce the \emph{periodic} setting, a reset-free framework where the goal is to learn a policy that minimizes the total loss while ensuring that, after a predefined period of time steps, the agents return to a state distribution equal to their initial distribution. We refer to a policy that satisfies this condition (returning the agents to their initial distribution after a fixed period) as a \emph{periodic policy} (formally defined in Sec.~\ref{sec:learning_problem}). We call each such period an episode, though unlike in episodic RL, there are no resets: learning is continuous, as in infinite-horizon RL, but with the requirement of returning to the initial distribution. This setting allows us to establish the first non-asymptotic guarantees for learning without resets in the case of $M$ homogeneous agents with $M > 1$ (see Table~\ref{table:comparisons} for a comparison with existing approaches).

\textbf{Motivations.} The reset-free periodic framework provides a practical approximation to infinite-horizon problems by considering a restricted, more tractable class of policies, which in turn allows finite-time guarantees in online settings with unknown dynamics. An example of a real-world problem that fits this model is that of energy demand control in smart grids. As storing energy is challenging, thermostatically controlled electrical devices, such as water-heaters, can adjust their consumption to help balance supply and demand \citep{marinmoreno_energy}. For that, an electricity utility sends signals through smart meters to control devices’ On-Off states so that their average consumption follows a target profile consumption, with each device’s state defined by its temperature and On-Off status. Daily control is needed, but episodic algorithms fail because devices cannot be reset to the same temperature distribution each day, while infinite-horizon methods fail to handle adversarial losses or unknown transitions caused by supply fluctuations and unknown user behavior. A periodic framework is better suited: devices naturally follow daily cycles, and optimizing over periodic policies enables infinite-horizon control without resets. The periodic framework also provides a natural model for tasks that inherently exhibit periodicity. Consider, for instance, autonomous robots operating in a warehouse, tasked with transporting and sorting goods. Given a list of assignments, their goal is to maximize task completion within a limited time while ensuring they autonomously return to their charging station at the end. The robot's state refers to their position, and all other  factors are captured within the objective function. 

\textbf{Challenges.} The main challenge of finding the optimal periodic policy arises when the environment dynamics are unknown, forcing the learner to rely on estimates to compute a policy. As a result, the learner cannot guarantee that the computed policy is truly periodic, \emph{i.e.}, if it returns the agents to their initial distribution in the true environment. This can cause the agents to drift gradually over successive episodes, potentially preventing a return to the initial distribution. 
\begin{table*}
  \begin{center}
    \resizebox{1\textwidth}{!}{
  \begin{tabular}{|Sc|Sc|Sc|Sc|Sc|}
    \cline{1-5}
   Framework &Resets  &Transition & Algorithm & Regret \\ \cline{1-5}
     Episodic  &  \multirow{2}{1em}{\cmark} & \multirow{1}{2em}{known} & \cite{neu_2013}  & $\widetilde{O}\big(N \sqrt{T}\big)$  \\\cline{3-5} 
                                        & & \multirow{2}{4em}{unknown} & \cite{neu12}  & $\widetilde{O}\big(N^2 |\mathcal{X}| |\mathcal{A}| \sqrt{T}\big)$  \\ 
                                                                        & & & \cite{Rosenberg2019}&  $\widetilde{O}\big(N^2 |\mathcal{X}| \sqrt{|\mathcal{A}| T} \big)$ \\
                                                                        \cline{1-5} 
                                                                       
     Infinite-horizon & \multirow{2}{1em}{\xmark} & \multirow{4}{2em}{known}  & \cite{even-dar}  & (ergodic) $\widetilde{O}\big(\tau^2 \sqrt{T}\big)$ \\ 
                                                                     & & & \cite{communicating} & (commu.+ ineff.) $\widetilde{O}\big(D^2 \sqrt{|\mathcal{X}| T}\big)$ \\
                                                                      & & & \cite{dai_ftrl} & (commu. + eff.) $\widetilde{O}\big(|\mathcal{A}|^{1/2} (|\mathcal{X}| D)^{2/3} T^{5/6}\big)$ \\
                                                                       & & & \cite{zhao} & (ergodic) $\widetilde{O}\big(\sqrt{\tau T}\big)$ \\  \cline{3-5} 
                                                    & & unknown & None  &  \\ 
                                                     \cline{1-5} 
       \textbf{Periodic (this work)}  & \xmark & \multirow{2}{4em}{unknown} & \textbf{MDPP-K} (\textbf{Alg.}~\ref{alg:main1}) & (Ass.~\ref{ass:contraction}) $\widetilde{O}\big((1-\alpha)^{-1/2}N^3 |\mathcal{A}|^{1/4} |\mathcal{X}|^{5/4} T^{3/4} \big)  $ \\
                                              & & & \textbf{MDPP-U} (\textbf{Alg.}~\ref{alg:main2}) & (Ass.~\ref{ass:ergodicity})  $\widetilde{O}\big((1-\alpha)^{-1}N^3 |\mathcal{A}|^{1/4} |\mathcal{X}|^{5/4} T^{3/4} \big)  $ \\
    \cline{1-5}
  \end{tabular}
  }
      \caption{Online MDP frameworks with theoretical guarantees under adversarial losses and stochastic transitions; all notations are in App.~\ref{notations}, $\tau = -(\log(\alpha))^{-1}$, $D$ is the MDP diameter, `commu.' denotes a communicating MDP, `ineff.' an inefficient algorithm, and `eff.' an efficient one.}
          \label{table:comparisons}
  \end{center}
\end{table*}

\textbf{Outline and contributions.}  In this paper, we introduce and study the problem of computing periodic policies in online MDPs with unknown transitions and adversarial loss functions, \emph{i.e.}, that can vary arbitrarily across episodes and are unknown to the learner in advance \citep{even-dar}. In Sec.~\ref{sec:learning_problem} we formally define the problem of finding an optimal periodic policy and identify sufficient conditions on the MDP for this problem to be well-posed. To evaluate the performance of algorithms designed for this setting, we introduce a novel metric, the \emph{periodic regret}. In Sec.~\ref{sec:algorithm} and~\ref{sec:algorithm_unknown_rhot} we design algorithms for computing periodic policies in environments with unknown dynamics and adversarial losses, considering the case of $M > 1$ agents under two distinct settings, each based on different underlying assumptions. We show both cases achieve a periodic regret of order $\smash{\widetilde{O}(T^{3/4})}$ while having low computational complexity. We present empirical results in Sec.~\ref{sec:experiments}. 

\section{The learning problem}\label{sec:learning_problem}

Our first contribution is to formalize the task of finding an optimal periodic policy and to establish sufficient conditions on the MDP under which this problem is well-defined. To address a broader range of tasks, such as those explored in our experiments (Sec.~\ref{sec:experiments}), we formulate the problem within the convex RL framework \citep{hazan2019, geist2022}, meaning that we consider any convex loss over the state-action distribution induced by the agent’s policy, rather than linear losses (classic RL). \textbf{Notations:} For any finite set $\mathcal{S}$, we denote by $\Delta_\mathcal{S}$ the simplex induced by this set, and by $|\mathcal{S}|$ its cardinality. For all $d \in \mathbb{N}$, let $[d] := \{1, \ldots, d\}$. Let $\|\cdot\|_1$ be the $L_1$ norm, and for all $v := (v_n)_{n \in [N]}$, such that $\smash{v_n \in \mathbb{R}^{|\mathcal{X}| |\mathcal{A}|}}$, we set $\|v\|_{\infty,1} := \sup_{n \in [N]} \|v_n\|_1.$ We denote the index of an episode by $t$, and a time step within an episode by $n$. A list of notations is in App.~\ref{notations}.

\subsection{Preliminaries: finite-horizon (convex) MDP} 

For clarity, we define here the structure of a Markov decision process (MDP) within an episode without the notation $t$. We define a MDP with finite state and action spaces $\mathcal{X}$ and $\mathcal{A}$, a horizon of length $N$, an initial state-action distribution $\rho \in \Delta_{\mathcal{X} \times \mathcal{A}}$, and a time-dependent dynamics $p := (p_n)_{n \in [N]}$, with $p_n : \mathcal{X} \times \mathcal{A} \times \mathcal{X} \rightarrow [0,1]$ the probability transition function at time step $n$ where $p_n(x'|x,a)$ denotes the probability that an agent moves to state $x'$ when performing action $a$ at state $x$ at time step $n$. The agent chooses an action at time step $n$ by sampling from a non-stationary random Markov policy $\pi := (\pi_n)_{n \in [N]}$, with $\pi_n : \mathcal{X} \rightarrow \Delta_\mathcal{A}$, and $a_n \sim \pi_n(\cdot|x_n) $. We let $\smash{\Pi := (\Delta_\mathcal{A})^{\mathcal{X} \times N}}$ denote the space of policies. We define the state-action distribution sequence $\mu^{\pi, \rho} := (\mu_n^{\pi, \rho})_{n \in [N]}$ induced by playing the policy $\pi$ and starting at the initial distribution $\rho$ recursively through the following forward equation for all $(x,a) \in \mathcal{X} \times \mathcal{A}$, with $\mu_0^{\pi, \rho}(x,a) := \rho(x,a)$, and for all $n \in [N]$,
\begin{equation}\label{mu_induced_pi}
\textstyle{\mu_{n+1}^{\pi, \rho}(x,a) := \sum_{x',a'} \mu_n^{\pi, \rho}(x',a') p_{n+1}(x|x', a') \pi_{n+1}(a|x).}
\end{equation}

At each time step, the learner incurs a loss $f_n(\mu_n^{\pi, \rho})$, where $f_n : \mathbb{R}^{\mathcal{X} \times \mathcal{A}} \rightarrow \mathbb{R}$ is any convex and $\ell$-Lipschitz function with respect to the norm $\|\cdot\|_1$ in the state-action distribution. The learner's objective is then to find a policy $\pi$ that minimizes the loss $\smash{F(\mu^{\pi, \rho}) := \sum_{n=1}^N f_n(\mu_n^{\pi, \rho})}$. This formulation is known as the convex RL problem \citep{pmlr-v238-moreno24a}. In the case where $f_n(\mu) = \langle \ell_n, \mu_n \rangle$ for $\ell_n: \mathcal{X} \times \mathcal{A} \rightarrow \mathbb{R}$ a loss function, we retrieve the classic RL case \citep{neu_2013}. Considering a general convex function over the state-action distribution instead of a linear one increases the range of applicable scenarios, for example the energy optimization task discussed in the previous section \citep{busic_energy_2023,marinmoreno_energy} where the loss is typically quadratic on the distribution, as well as the experiments in Sec.~\ref{sec:experiments}, where we compare our methods with episodic RL approaches by incorporating a regularizer into the loss function of the episodic algorithm, thereby breaking the linearity of the RL problem.

Note that $F$ is not convex on the policy $\pi$. To convexify the problem we introduce the set of sequences of state-action distributions initialized with $\rho \in \Delta_{\mathcal{X} \times \mathcal{A}}$ and satisfying the dynamics $p$:
\begin{equation}\label{eq:bellman_flow}
\textstyle{\mathcal{M}^p_\rho := \Big\{\mu := (\mu_n)_{n \in [N]} \; \Big| \; \mu_0 = \rho \text{ and } \sum_{a} \mu_{n+1}(x,a) := \sum_{x',a'} \mu_n(x',a') p_{n+1}(x|x', a') \;   \Big\}.}
\end{equation}
Since $\mathcal{M}_\rho^p$ is a set of affine constraints, it is convex \citep{puterman2014markov}. For any $\mu \in \mathcal{M}_\rho^p$, there is a policy $\pi$ such that $\mu^{\pi, \rho} = \mu$. It suffices to take $\pi_n(a|x) =\mu_n(x,a)/\sum_{a'} \mu_n(x,a')$, if the denominator is non-zero, and arbitrarily defined if not. Thus, finding an optimal policy minimizing $F(\mu^{\pi, \rho})$ (a non-convex problem) is equivalent to minimizing $F(\mu)$ over $\mathcal{M}_\rho^p$ (a convex problem).

\subsection{Periodic policies}
For $n \in [N]$, let $y_n := (x_n,a_n)$, and for all $\pi \in \Pi$, let $P^\pi_n(y_{n-1}, y_n) := p_n(x_n|y_{n-1}) \pi_n(a_n|x_n)$ be the probability to move from $y_{n-1}$ to $y_n$ when following $\pi$. We introduce for all $y_0, y_N\in \mathcal{X} \times \mathcal{A}$,
\begin{equation}\label{eq:markov_chain_proba}
    \begin{split}
\textstyle{        P_\pi(y_0, y_N) := \sum_{y_1}P^\pi_1(y_0, y_1) \ldots \sum_{y_{N-1}} P^\pi_{N-1}(y_{N-2}, y_{N-1}) P^\pi_N(y_{N-1}, y_N),}
    \end{split}
\end{equation}
as the probability to move from the initial state $y_0$ to the final state $y_N$. Interpreting \( P_\pi \) as a \( |\mathcal{X}| |\mathcal{A}| \times |\mathcal{X}| |\mathcal{A}| \) matrix, for any \(\nu \in \mathbb{R}^{|\mathcal{X}|  |\mathcal{A}|}\), we define \(\nu P_\pi\) as the vector satisfying for each \((x,a)\),  $\nu P_\pi(x,a) = \sum_{x',a'} \nu(x',a') P_\pi(x',a'; x,a).$

We consider agents that interact repeatedly with the finite-horizon MDP over a sequence of $T$ episodes by playing a sequence of policies $(\pi_t)_{t \in [T]}$, with each $\pi_t := (\pi_{t,n})_{n \in [N]}$ a policy over the horizon $N$. At the first episode, the agents are initialized following the initial state-action distribution $\rho$. In episodic settings, the system restarts the agents at the distribution $\rho$ after the end of each episode. However, in real-world, resets are often prohibitive. Instead, we consider that the system never resets to the distribution $\rho$. Thus, at episode $t+1$, the initial distribution, that we denote by $\rho_{t+1}$, is equal to the final distribution on the previous episode, \emph{i.e.}, $\rho_{t+1} := \rho_{t} P_{\pi_t} = \mu^{\pi_t, \rho_t}_N$. 
\begin{definition}[Periodic policy]\label{def:periodic_policy}
    We call a policy $\pi := (\pi_n)_{n \in [N]}$ periodic in the MDP with initial state-action distribution $\rho$, if $\rho P_\pi = \rho$. 
\end{definition}

The set of periodic policies can be seen as the set of policies whose stationary distribution under the MDP transition is \(\rho\). Note that for a sequence of periodic policies, \(\rho_t = \rho\) for all \(t \in [T]\). In this paper, we study the online learning problem of computing a sequence of approximately \emph{periodic} policies \smash{\((\pi_t)_{t \in [T]}\)} over \(T\) episodes, aiming to minimize the total loss \smash{\(\sum_{t=1}^T F_t(\mu^{\pi_t, \rho})\)}. The loss functions \smash{\((F_t)_{t \in [T]}\)} are an arbitrary sequence of convex losses over the state-action distribution and are unknown to the learner in advance. At the end of each episode \(t\), the loss function \(F_t\) is fully revealed to the learner (full-information setting). We also assume the probability kernel $p$ is unknown and must be learned. This problem is well-defined only under the following assumption:  
\begin{assumption}\label{ass:feasibility}
    The set of periodic policies of the MDP with initial distribution $\rho$ is non-empty. 
\end{assumption}

Although Ass.~\ref{ass:feasibility} cannot be verified for an arbitrary $\rho$, one can construct a suitable $\rho$ that satisfies it and use this distribution to run the proposed algorithms. For example, under an ergodicity assumption, running any policy for sufficiently many episodes yields an approximate stationary distribution, for which the set of periodic policies is non-empty. Moreover, some processes are inherently periodic, with the nominal policy naturally returning the state to a stationary distribution. For instance, in the energy demand control example in Sec.~\ref{sec:introduction}, the uncontrolled water temperature distribution in electric water-heaters follows a periodic daily cycle due to the seasonality of user behavior.

Due to uncertainties in the unknown dynamics, the learner cannot guarantee that the computed policy sequence is strictly periodic. As a result, in practice, \(\rho_t\) may deviate from \(\rho\). To justify that we can compute a policy that induces a state-action distribution at the end of an episode close enough to \(\rho\) when starting from \(\rho_t \neq \rho\) (\emph{i.e.}, \(\rho_t P_\pi \approx \rho\)), we impose the following contraction assumption:
\begin{assumption}\label{ass:contraction}
    There exists some $0 \leq \alpha < 1$ such that for all periodic policies $\pi$ and distributions $\nu, \nu' \in \Delta_{\mathcal{X} \times \mathcal{A}}$, $\|\nu P_\pi - \nu' P_\pi \|_1 \leq \alpha \|\nu - \nu'\|_1.$
\end{assumption}

Ass.~\ref{ass:contraction} can be understood as requiring that the Markov chain induced by any periodic policy be ergodic, \emph{i.e.}, irreducible and aperiodic \citep{mixing_times}. Consequently, the Markov chain under a periodic policy converges to a unique stationary distribution (Ass.~\ref{ass:contraction}), which in this case is \(\rho\) (Ass.~\ref{ass:feasibility}), independent of the initial state. This type of assumption is common in infinite-horizon MDPs \citep{even-dar,neu_2010,chen2022learninginfinitehorizonaveragerewardmarkov}.

We evaluate the learner's performance through a new regret that we call the \emph{periodic regret}, which compares the learner's total loss to that of any periodic policy \(\pi\) in the MDP with the initial state-action distribution \(\rho\), and also penalizes the distance between the final distribution in episode $t-1$, $\rho_{t}$, for all $t \in \{2, \ldots, T+1\}$, and $\rho$, scaled by a factor of $\gamma > 0$, as follows:
 \begin{equation}\label{eq:periodic_regret}
     \textstyle{R_T(\pi) :=  \sum_{t=1}^T \big[ F_t(\mu^{\pi_t, \rho_t})  - F_t(\mu^{\pi, \rho}) \big] + \gamma \sum_{t=1}^T \|\rho_t - \rho \|_1.}
 \end{equation}

Our aim is to find a near-optimal periodic policy, one that returns the agents to the prescribed initial distribution at the end of each episode while minimizing the cumulative loss. Accordingly, it is natural to define a static regret that compares the learner’s performance to that of the best periodic policy in hindsight, augmented with a penalty term that captures deviations from the desired initial distribution. However, since a non-periodic policy may achieve a lower total loss than the optimal periodic policy, the penalty term must be scaled by a constant $\gamma$.

\subsection{Learning frameworks}

We consider $M > 1$ homogeneous agents and observe an independent trajectory from one of them per episode to estimate the dynamics. This independence is essential for applying standard martingale concentration bounds. Unlike the episodic setting, where restarts ensure independence, here it is obtained through trajectories observed from different agents every episode. A sufficiently large $M$ is needed to guarantee independence across sampled trajectories. When $M=1$, the conditional expectation of the initial distribution inducing the occupancy measure in the current episode given the past trajectories collapses to a Dirac mass at the last observed state at time step $N$, instead of the true initial distribution $\rho_t$ (see the proof of Prop.~\ref{prop:mdp_martingale} for more details). This prevents the use of martingale difference sequence arguments typically employed to control the regret term associated with learning the dynamics. Introducing a second independent agent resolves this issue starting from the second episode, and the same reasoning then applies inductively across episodes.

In practice, under standard ergodicity assumptions, the system can be expected to eventually ``forget'' its past, making it reasonable to reuse agents in the estimation procedure after a finite number of episodes. The multi-agent assumption is therefore justified both as a natural first step toward estimating dynamics in the online infinite-horizon setting, and by the fact that it already encompasses many practical applications, including those discussed in the introduction. We leave the single-agent case for future work.

We examine two frameworks in this paper. In the first, we assume that the learner has access to the initial state-action distribution \(\rho_t\) of the agents at the beginning of each episode \(t\). We introduce an algorithm for this setting in Sec.~\ref{sec:algorithm} and establish its periodic regret guarantee. 

While observing \(\rho_t\) can be justified for large agent populations, it may not always be feasible in practice. To address this limitation, our second framework eliminates the assumption of direct access to \(\rho_t\). Instead, we introduce an estimated distribution \(\tilde{\rho}_t\) and apply the same algorithm from Sec.~\ref{sec:algorithm}, replacing \(\rho_t\) with its estimate. However, this introduces additional challenges, as errors now accumulate across episodes due to both unknown transition probabilities and unknown initial distributions, motivating the need for additional assumptions. First, that the MDP must satisfy the contraction property of Ass.~\ref{ass:contraction} for all policies \(\pi\), not just periodic ones. Additionally, we assume that one randomly selected agent can be reset at the start of each episode, which we use \emph{exclusively} to estimate the initial distribution $\rho_t$ for all $t \in [T]$. Note that this is still a weaker requirement than resetting all agents to the initial distribution as is done in the episodic case. The assumption is also realistic in certain applications. For example, in the energy demand control scenario from Sec.~\ref{sec:introduction}, some consumers may have special electricity contracts allowing the utility to reset their thermic devices to a desired temperature distribution at the end of each day, whereas resetting all devices for all consumers would be impractical. In Sec.~\ref{sec:algorithm_unknown_rhot}, we provide a detailed discussion of this framework. 


\textbf{Online protocol (Alg.~\ref{alg:online_protocol}).} The learning process proceeds in episodes, but unlike episodic RL, the agents are not reset between episodes. We consider \( M > 1 \) independent agents that do no interact with each other. At the start of each episode \( t \), the learner selects a policy \( \pi_t \), sends it to all agents, and then observes the trajectory of an agent uniformly sampled from the $M$ available, independently of previous observations. Using the independent trajectories observed up to episode \( t \), the learner computes \( \hat{p}_{t+1} \), an estimate of the dynamics \( p \). For the first episode, the initial state-action pair for each agent $j \in [M]$, \((x_{1,0}^j, a_{1,0}^j)\), is drawn from the initial distribution \(\rho \sim \Delta_{\mathcal{X} \times \mathcal{A}}\). In subsequent episodes, the initial state-action pair is carried over from the final step of the previous episode: $ (x_{t+1,0}^j, a_{t+1,0}^j) = (x_{t,N}^j, a_{t,N}^j).$ The initial distribution at episode \( t+1 \) evolves according to \( \rho_{t+1} =  \rho_t P_{\pi_t} \). At the end of each episode, the learner observes the full loss function \( F_t \), and computes the next policy $\pi_{t+1}$ using different information depending on the framework, as illustrated in the learner protocol in Alg.~\ref{alg:online_protocol}.
\begin{algorithm}[ht]
\caption{Multi-agent Online Protocol}\label{alg:online_protocol}
\begin{algorithmic}[1]
\STATE {\bfseries Setting:} initial policy $\pi_1$, $M$ agents
   \FOR{$t= 1,\ldots,T$}
     \STATE All agents follow policy $\pi_t$ for an episode
     \STATE Observe the trajectory of an agent independently sampled from the population of $M$ agents
     \STATE Use the trajectory to update the estimate $\hat{p}_{t+1}$ of $p$
   \vspace{-0.4cm}
   \begin{multicols}{2}
   
    \textbf{Framework 1: known $\rho_{t+1}$}
   \STATE Observe $\rho_{t+1}$ and $F_t$
   \STATE Compute $\pi_{t+1}$ using $\rho_{t+1}, F_t$ and $\hat{p}_{t+1}$
\columnbreak

 \textbf{Framework 2: unknown $\rho_{t+1}$}
   \STATE Compute $\tilde{\rho}_{t+1}$ estimate of $\rho_{t+1}$, observes $F_t$
   \STATE Compute $\pi_{t+1}$ using $\tilde{\rho}_{t+1}, F_t$ and $\hat{p}_{t+1}$
   \end{multicols}
      \vspace{-0.4cm}
      \ENDFOR
   \end{algorithmic}
 \end{algorithm}
\section{Framework 1: known initial distributions}\label{sec:algorithm}

We first study the case where the learner observes the initial distribution at the start of each episode.
\subsection{Preliminaries}

We aim to formulate the learning problem as the one of finding the optimal policy per episode within a constrained MDP (C-MDP), where the final state-action distribution must satisfy an equality constraint. Leveraging the equivalence between the policy space and the space of state-action distributions satisfying the MDP dynamics as in Eq.~\ref{eq:bellman_flow}, we would like to solve at episode $t$
\begin{equation}\label{offline_problem}
\textstyle{\min_{\mu \in \mathcal{M}_{\rho_t}^p} \quad F_t(\mu), \quad \text{s.t.} \quad \mu_N := \rho_t P_\pi = \rho,}
\end{equation}  
where \( \pi \) is the policy corresponding to the distribution \( \mu \) in the sense of Eq.~\eqref{mu_induced_pi}.

However, in practice, we encounter several challenges when considering Eq.~\eqref{offline_problem}: $(1)$ The sequence of objective functions is adversarial. $(2)$ The probability kernel is unknown.  $(3)$ The problem may be ill-posed because it is not necessarily feasible. Ass.~\ref{ass:feasibility} ensures the existence of a policy satisfying \(\rho P_\pi = \rho\), but due to the uncertainties on the probability kernel, the  distribution \(\rho_t\) at the end of episode \(t-1\) may not exactly match \(\rho\). Moreover, there is no guarantee that there is a policy $\pi$ such that \(\rho_t P_\pi = \rho\).

While numerous works in RL have addressed challenges \(1\) and \(2\) (see Table~\ref{table:comparisons}), they do not consider constrained problems or periodic regret. Although we model our algorithm as a C-MDP, previous works \citep{efroni2020explorationexploitation, qiu2020, castiglioni, ding2020provably, stradi24, muller2024truly} differ on that they focus on either episodic or average-reward infinite-horizon tasks rather than on bounding periodic regret. Some works on the literature of stochastic control investigate constraints on the probability distribution of the terminal state \citep{daudin,bouchard_elie,quantile,pfeiffer_2020,pfeiffer_2021}. However, these studies primarily focus on deriving necessary optimality conditions under known dynamics. As such, they do not address challenge \(1\), nor \(2\). We present the first approach to address the three challenges with provable non-asymptotic guarantees. 

\subsection{Algorithm}\label{subsec:algo}
At each episode, the algorithm uniformly and independently selects one of the $M$ agents, and observe its trajectory that we denote by  $(x_{t,n}, a_{t,n})_{n \in [N]}$. Since agents are not reset between episodes, it is essential to ensure that observations remain independent across episodes. We define $\smash{N_{t,n}(x,a) = \sum_{s=1}^{t-1} \mathds{1}_{\{x_{s,n} = x, a_{s,n} = a\}}}$, $\smash{M_{t,n}(x'|x,a) = \sum_{s=1}^{t-1} \mathds{1}_{\{x_{s,n+1} = x' ,x_{s,n} = x, a_{s,n} = a\}}}$, as the counters of the number of visits of the pair $(x,a)$, and the triple $(x,a,x')$. The learner's estimate for the transition kernel at the end of episode $t-1$, to be used in episode $t$, is 
\begin{equation}\label{eq:proba_est}
\textstyle{    \hat{p}_{t,n+1}(x'|x,a) := \frac{M_{t,n}(x'|x,a)}{\max\{1, N_{t,n}(x,a)\}}.}
\end{equation}
From Eq.~\eqref{eq:bellman_flow}, we let $\smash{\mathcal{M}_t := \mathcal{M}^{\hat{p}_t}_{\rho_t}}$. For any initial distribution $\nu \in \Delta_{\mathcal{X} \times \mathcal{A}}$, we define  \(\hat{\mu}^{\pi, \nu}_t := (\hat{\mu}^{\pi,\nu}_{t,n})_{n \in [N]}\) the sequence of state-action distributions induced by the policy \(\pi\) in the estimated MDP $\hat{p}_t$ at episode $t$, where for all $(x,a) \in \mathcal{X} \times \mathcal{A}$,  $\hat{\mu}_{t,0}^{\pi, \nu}(x,a) := \nu(x,a)$, and for all $n \in [N]$, 
\[
\textstyle{\hat{\mu}_{t,n+1}^{\pi, \nu}(x,a) := \sum_{x',a'} \hat{\mu}_{t,n}^{\pi, \nu}(x',a') \hat{p}_{t,n+1}(x|x', a') \pi_{n+1}(a|x).}
\]

We present a bonus-based exploration iterative algorithm for adversarial convex constrained MDPs in Eq.~\eqref{iteration_md_solver}. To address the challenge of handling convex adversarial objective functions, the learner performs a Mirror Descent-like (MD) \citep{MD} iteration in each episode \( t \). Given that the true transition kernel is unknown, the learner solves the optimization problem over the set of state-action distributions induced by the current estimated MDP \(\hat{p}_t\). 

To ensure accurate estimation of \( p \), we need to properly explore the environment. To achieve this, we adjust the gradient of the objective function in MD by subtracting a sequence of vectors \( \bar{b}_t := (\bar{b}_{t,n})_{n \in [N]} \), which we denote by bonus vectors, where \( \bar{b}_{t,n} \in \mathbb{R}^{\mathcal{X} \times \mathcal{A}} \) is defined in Eq.~\eqref{eq:bonus}. To ensure the feasibility of the problem, we must carefully construct the constraint set to consider at each episode. Ass.~\ref{ass:feasibility} guarantees the existence of a periodic policy in the true MDP, starting from the distribution $\rho$. However, this guarantee does not extend to the estimated MDP, as we cannot ensure that the same policy remains periodic there. To address this, we transform the initial equality constraint in the terminal law in an inequality constraint using a second bonus vector $b_t := (b_{t,n})_{n \in [N]}$ also defined in Eq.~\eqref{eq:bonus}. Recall that $\ell$ is the Lipschitz constant of $f_{t,n}$ with respect to $\|\cdot\|_1$ for all $t \in [T]$ and $n \in [N]$. We set
\begin{equation}\label{eq:bonus}
    \textstyle{\bar{b}_{t,n}(x,a) =  \frac{\ell (N-n) C_\delta}{\sqrt{\max{\{1,N_{t,n}(x,a)}\}} }, \quad \text{and} \quad b_{t,n}(x,a) = \frac{ C_\delta}{\sqrt{\max{\{1,N_{t,n}(x,a)}\}} },}
\end{equation}
where $\smash{C_\delta = (2 |\mathcal{X}| \log\big(|\mathcal{X}| |\mathcal{A}| N T/\delta\big)})^{1/2}$, for $\delta \in (0,1)$. 

The structure of the bonus terms can be interpreted as follows. The term $\bar{b}$ is subtracted from the MD subgradient in Eq.~\eqref{iteration_md_solver} to encourage exploration: intuitively, a state-action pair $(x,a)$ that has been visited only a few times (and thus has a large $\bar{b}_n(x,a)$) is treated as having a lower loss, which promotes further sampling. The specific form of $\bar{b}$ follows from the concentration bound on the $L_1$ distance between the true and estimated transition dynamics established in Lemma~\ref{lemma:proba_difference}, and the underlying intuition is analogous to the role of UCB bonuses in multi-armed bandits \citep{lattimore}. The bonus term $b$ is added to the constraints of the MD scheme to ensure feasibility of the optimization problem with respect to the estimated MDP. It explicitly accounts for the uncertainty in the transition model $p$. More details on the structure of both bonus terms becomes clear through the regret analysis.

Adding a bonus to the initial constraint raises another feasibility issue, as we cannot guarantee that the policy computed satisfies the equality constraint on the terminal distribution when applied in practice. This implies that the initial distribution at episode $t$, $\rho_t$, may differ from $\rho$, and we cannot ensure the existence of a policy $\pi$ that brings $\rho_t$ back to $\rho$, \emph{i.e.}, $\rho_t P_\pi = \rho$. To tackle this issue, we make use of Ass.~\ref{ass:contraction} to build a feasible constraint set as we later show in Lemma~\ref{lemma:feasibility}. 

Let $\pi_t$ be the policy played at episode $t$. For short let $\mu_t := \hat{\mu}^{\pi_t, \rho_t}_t$, and $\ell_t := \nabla F_t(\mu_t)$. Let $D_\psi$ be any Bregman divergence over a sequence of state-action distributions, with $\psi$ the function inducing it. Let $0 \leq \bar{\alpha} < 1$. At episode $t+1$, with some $\eta> 0$ to be tuned later, the learner computes
\begin{equation}\label{iteration_md_solver}
\begin{aligned}
 \mu_{t+1} \in \argmin_{\mu \in \mathcal{M}_{t+1}} \quad &\Big\{ \eta \langle \ell_{t} - \bar{b}_t, \mu \rangle +  D_\psi(\mu, \mu_{t}) \Big\} \\
\textrm{s.t.} \quad  &\| \mu_N - \rho \|_1 \leq \langle \mu, b_t \rangle +  \bar{\alpha} \|\rho_t - \rho\|_1,
\end{aligned}
\end{equation}
and sends the policy $\pi_{t+1}$ associated with $\mu_{t+1}$ to all agents. We show that the constraint set is sufficiently large to ensure that the problem solved at each iteration is feasible (see Lemma~\ref{lemma:feasibility}), yet sufficiently small to ensure that the cumulative distance between the initial distribution at each episode, $\rho_t$, and the desired distribution, $\rho$, remains small (see Lemma~\ref{lemma:almost_equal_dist}). This balance is crucial for analyzing the periodic regret in Sec.~\ref{sec:regret_analysis}. The proof of both Lemmas is in App.~\ref{app:feasibility}. We define MDPP-K (Mirror Descent for Periodic Policies - Known initial distributions) in Alg.~\ref{alg:main1} at App.~\ref{app:algo_scheme} as a method that solves one iteration of Eq.~\eqref{iteration_md_solver} at each episode $t$.
\begin{lemma}\label{lemma:feasibility}
    For any $\bar{\alpha} \geq \alpha$, for $\alpha$ as in Ass.~\ref{ass:contraction}, the problem in Eq.~\eqref{iteration_md_solver} is feasible with high-probability.
\end{lemma}

\begin{lemma}\label{lemma:almost_equal_dist}
Let $\pi_t$ be the policy obtained from $\mu_t$ solution of Eq.~\eqref{iteration_md_solver} at episode $t$. Recall that $\rho_1 = \rho$, and for all $t \in [T]$, $\rho_{t+1} = \rho_t P_{\pi_t}$. Then, with high probability, 
\[
  \textstyle{\sum_{t=1}^T \| \rho_{t+1} - \rho \|_1 \leq   \tilde{O}\Big( \frac{N^2}{1- \alpha} |\mathcal{X}|^{3/2} \sqrt{|\mathcal{A}| T} \Big).}
\]
\end{lemma}

\textbf{Practical solution of Problem~\eqref{iteration_md_solver}.}
Our theoretical analysis assumes that Problem~\eqref{iteration_md_solver} can be solved optimally, which we can do in certain special cases of the Bregman divergence. We outline the key idea behind the practical solution, with full details provided in App.~\ref{app:practical_sol}. This approach is used in the experiments presented in Sec.~\ref{sec:experiments}. We adopt a Lagrangian approach that dualizes the constraint on the final distribution, leading to the following Lagrangian formulation:  
\[
\textstyle{\mathcal{L}_t(\mu, \lambda) := \langle \ell_t - \bar{b}_t, \mu \rangle + \frac{1}{\eta} D_\psi(\mu, \mu_t) + \lambda\big[ \|\mu_N - \rho\|_1 - \langle \mu, b_t \rangle - \bar{\alpha} \|\rho_t - \rho \|_1 \big],}
\]
where \( \lambda \in \mathbb{R}_+ \) is the Lagrange multiplier. From the feasibility result of Lemma~\ref{lemma:feasibility}, we know that the optimal solutions to  $\max_{\lambda > 0} \min_{\mu \in \mathcal{M}_{t+1}} \mathcal{L}_t(\mu,\lambda)$ are a pair of optimal primal and dual variables. For a fixed \( \lambda > 0 \) and for a specific Bregman divergence introduced in Eq.~\eqref{gamma:non_standard}, it can be shown that the policy inducing the optimal distribution in the unconstrained problem $\min_{\mu \in \mathcal{M}_{t+1}} \mathcal{L}_t(\mu,\lambda)$ has a closed-form solution stated in Eq.~\eqref{eq:closed_form} \citep{neu2017unifiedviewentropyregularizedmarkov}. Thus, we adopt a min-max approach: we initialize the Lagrange multiplier, compute the corresponding distribution by solving the associated unconstrained minimization problem, and then update the Lagrange multiplier using a gradient ascent step. This process continues iteratively until the resulting distribution satisfies the constraints.

Lemma~\ref{lemma:feasibility} established that the problem in Eq.~\eqref{iteration_md_solver} is feasible for any $\bar{\alpha} \geq \alpha$. In practice, the contraction parameter \( \alpha \) may be unknown. Nonetheless, since the algorithm only requires a value \( \bar{\alpha} \) that guarantees feasibility, a practical solution is to construct a grid over the interval \( (0,1) \) and choose the smallest \( \bar{\alpha} \) for which the problem is feasible. To do this, we can search over an exponentially spaced grid $\{\frac{1}{2}, \frac{3}{4},\frac{7}{8},\ldots \}$, and stop once the problem becomes feasible for the current value. This procedure yields a computational cost of $\smash{\log\big(\frac{1}{1-\alpha}\big)}$ per episode. To test feasibility, the $L_1$ constraint can be reformulated as a set of linear inequalities, allowing a solver to efficiently check feasibility. This introduces only a factor of $2$ in the regret, and in practice may perform even better as we may find a value $\bar{\alpha} < \alpha$ for which the problem is already feasible.

\subsection{Periodic regret analysis}\label{sec:regret_analysis}
The main result is in Thm.~\ref{thm:main_periodic_regret_known_rhot}, with its proof in App.~\ref{app:main_result}. We highlight here the novelties of our analysis.
\begin{theorem}\label{thm:main_periodic_regret_known_rhot}
    Running MDPP-K (Alg.~\ref{alg:main1}) for $T$ episodes in an MDP with unknown transition kernel $p$ and initial state-action distribution $\rho$, against adversarial objectives $\smash{F_t := \sum_{n=1}^N f_{t,n}}$, where each $f_{t,n}$ is convex and $\ell$-Lipschitz with respect to the norm $\|\cdot\|_1$, with parameter $\alpha$, the optimal choice of $\eta$, \( \|\nabla \psi(\mu_t)\|_{1, \infty} \leq \Psi \), guarantees that, for any periodic policy $\pi$, with high probability, 
    \[
\textstyle{R_T(\pi) \leq \tilde{O} \big(  \ell N^3|\mathcal{X}|^{5/4} |\mathcal{A}|^{1/4} \sqrt{\Psi(1 - \alpha)^{-1}} T^{3/4} +  (1-\alpha)^{-1} \ell N^4 |\mathcal{X}|^{2} \sqrt{|\mathcal{A}| T}   \big).}
\]
\end{theorem}
To obtain this result, we begin by decomposing the periodic regret into two components: 
\begin{equation*}
    \begin{split}
      \displaystyle{   R_T(\pi) = \underbrace{\sum_{t=1}^T F_t(\mu^{\pi_t, \rho_t}) - F_t(\hat{\mu}^{\pi_t, \rho_t}_t)}_{R_T^{\text{MDP}}} + \underbrace{\sum_{t=1}^T F_t(\hat{\mu}^{\pi_t, \rho_t}_t) - F_t(\mu^{\pi, \rho})}_{R_T^{\text{policy}}}} + \gamma \sum_{t=1}^T \|\rho - \rho_t \|_1.
    \end{split}
\end{equation*}
We show that with high probability $R_T^{\text{MDP}} = \tilde{O}(\sqrt{T})$. As this analysis generalizes conventional techniques \citep{Rosenberg2019}, we leave the details in the appendix. Recall that $\ell_t := \nabla F_t(\mu_t)$ and $\mu_t := \hat{\mu}^{\pi_t, \rho_t}_t$. Using the convexity of $F_t$, we further upper bound $\smash{R_T^{\text{policy}}}$ with
\begin{equation*}
    \begin{split}
        \displaystyle{ \underbrace{\sum_{t=1}^T \langle \ell_t - \bar{b}_t, \mu_t - \hat{\mu}^{\pi, \rho_t}_t \rangle}_{R_T^{\text{MD}}} + \underbrace{\sum_{t=1}^T \langle \ell_t - \bar{b}_t, \hat{\mu}^{\pi, \rho_t}_t - \hat{\mu}^{\pi, \rho}_t \rangle}_{R_T^{\text{diff. }\rho_t}} + \underbrace{\sum_{t=1}^T \langle \bar{b}_t, \mu_t - \hat{\mu}^{\pi, \rho}_t \rangle + \langle \ell_t, \hat{\mu}^{\pi, \rho}_t - \mu^{\pi, \rho} \rangle}_{R_T^{\text{bonus}}}.}
    \end{split}
\end{equation*}
The bonus vector added to the gradient term in Eq.~\eqref{iteration_md_solver} ensures that the term  $R_T^{\text{bonus}} = \tilde{O}(\sqrt{T})$, and its analysis is generalized from the one in \citep{moreno_icml}, hence the details are left in App.~\ref{app:main_result}. The novelties of our analysis lie in the way we treat the term $\smash{R_T^{\text{MD}}}$ and the $\rho_t$ variation in $\smash{R_T^{\text{diff. }\rho_t}}$.

\textbf{$R_T^{\text{diff. }\rho_t}$ term.} This term accounts for the deviation of the initial distribution in each episode from the original target. In the episodic setting, where all agents reset at the end of each episode, this term does not appear. We upper bound it in Prop.~\ref{prop:diff_rho_rhot}, with the proof in App.~\ref{app:proof_rho_rhot}. 
\begin{proposition}\label{prop:diff_rho_rhot}
With high probability we have that
\[
\textstyle{R_T^{\text{diff. } \rho_t} \leq \tilde{O}\big( (1-\alpha)^{-1} \ell N^4 |\mathcal{X}|^2 \sqrt{|\mathcal{A}| T} \big).}
\]
\end{proposition}
\textbf{$R_T^{\text{MD}}$ term.}
This term measures the quality of our computations using an MD scheme, with an upper bound in Prop.~\ref{prop:bound_md_term} and proof in App.~\ref{app:md_proof}. Unlike in standard MD, this term is more challenging because each iteration uses an MDP induced by a different transition estimate $\hat{p}_t$, starting from a different initial distribution $\rho_t$. We assume $\|\nabla \psi(\mu_t)\|_{1, \infty} \leq \Psi$. This can be guaranteed in practice by mixing the output policy with a uniform policy, adding only logarithmic terms on $T$ to the regret. We do not elaborate on it further in this paper, as the technique is already described in works such as \cite{moreno_icml} and requires only straightforward additional calculations. 

\begin{proposition}\label{prop:bound_md_term}
    For each episode \( t \in [T] \), let \( \mu_t \) denote the solution of the optimization problem in Eq.~\eqref{iteration_md_solver}. Suppose further that \( \|\nabla \psi(\mu_t)\|_{1, \infty} \leq \Psi \), and let \( \ell \) be the Lipschitz constant of \( f_{t,n} \) with respect to $\|\cdot\|_1$ for all \( t \in [T] \) and $n \in [N]$. Then, with high probability,
    \[
\textstyle{R_T^{\text{MD}} \leq \tilde{O} \big( \ell N^3 |\mathcal{X}|^{5/4} |\mathcal{A}|^{1/4} \sqrt{\Psi (1 - \alpha)^{-1}} T^{3/4} \big).}
    \]
\end{proposition}
\textbf{Discussion.} 
We remark that the factor $(1-\alpha)^{-1}$ scales polynomially with the mixing time $\tau$, commonly used to bound regret in infinite-horizon settings (see Table~\ref{table:comparisons}). The $\smash{T^{3/4}}$ dependency of MDPP-K comes from applying OMD over state-action distribution sequences initialized with varying distributions $\rho_t$ per episode. Lower bounds for the adversarial infinite-horizon problem are known only in the case of known dynamics, yielding $O(\sqrt{T})$ in the full information setting \citep{communicating}. For unknown dynamics, establishing a lower bound remains an open problem.

Regarding the dependence of the regret on the MDP parameters, it is known that exploration bonuses lead to an upper bound with an $N^3$ dependence \citep{luo_dilated_bonus}, introducing an extra factor of $N$ compared to the episodic RL lower bound. In our setting, we show that MDPP-K exhibits the same dependence. Concerning the dependence on the number of states, prior work \citep{moreno_icml} shows that, in episodic RL, combining exploration bonuses with OMD in the occupancy-measure space yields a $|\mathcal{X}|^{3/2}$ dependence, an additional $\sqrt{|\mathcal{X}|}$ over the lower bound arising from the constants in the bonus terms (see Eq.~\eqref{eq:bonus}). In our setting, OMD must additionally handle changing initial state distributions, and the exploration bonuses appearing in the constraints influence these variations. As a result, the same constant is incurred once more, leading to an overall $|\mathcal{X}|^{5/4}$ dependence.

The computational complexity of MDPP-K increases linearly with $|\mathcal{X}|$, $|\mathcal{A}|$, and $N$, due to the structure of the closed-form output policy in Eq.~\eqref{eq:closed_form}. When we refer to the method as having low computational complexity, we mean that it admits a closed-form solution, as detailed in App.~\ref{app:practical_sol}, so no additional optimization step or external solver is required.

\section{Framework 2: unknown initial distributions}\label{sec:algorithm_unknown_rhot}

In this setting, the initial distribution $\rho_t$ is unknown. Thus, we approximate it by $\tilde{\rho}_t$. We then apply the iterative scheme from Eq.~\eqref{iteration_md_solver} as if $\tilde{\rho}_t$ were the true initial distribution. As $\rho_t = \rho_{t-1} P_{\pi_{t-1}}$, an accurate estimation relies on a good transition kernel estimate. In Framework 1, we needed a transition kernel estimate $\hat{p}_t$ that is accurate on the state–action pairs visited by agents initialized from the \emph{true} initial distribution $\rho_t$ under the computed policy $\pi_t$ (see the analysis of $R_T^{\text{MDP}}$ in App.~\ref{prop:R_T_mdp}). Since agents without resets are initialized from $\rho_t$, observing independent trajectories in each episode was sufficient in Sec.~\ref{sec:algorithm}.

However, using $\hat{p}_t$ to estimate the initial distribution $\rho_t$ does not yield a good approximation. Let $\tilde{\rho}_t$ denote its estimate at episode $t$. To reliably estimate the initial distribution for episode $t+1$, we need a transition kernel estimate accurate over the state–action pairs visited by an agent initialized from the previously \emph{estimated} initial distribution $\tilde{\rho}_t$ (see the proof of Prop.~\ref{prop:diff_rho_rhot} in App.~\ref{app:setting2}). The difficulty is that, without resets, agents actually start from the true distribution $\rho_t$, not the estimated one $\tilde{\rho}_t$.

To address this challenge, we introduce two additional assumptions. First, we strengthen Ass.~\ref{ass:contraction} so that it applies to all policies, not just periodic ones, which corresponds to the standard ergodicity assumption often used in infinite-horizon MDPs \citep{chen2022learninginfinitehorizonaveragerewardmarkov}. Second, at the start of each episode $t$, we allow a randomly selected agent to be restarted from a state–action pair sampled according to the \emph{estimated} initial distribution $\tilde{\rho}_t$. This agent can be the same each episode or vary across episodes. 

\textbf{Algorithm overview.}
Let \( \tilde{\rho}_1 = \rho \). For all \( t \geq 1 \), in the start of episode $t$, we reset a randomly selected agent to a initial state-action pair sampled from $\tilde{\rho}_t$. We observe the trajectory followed by the restarted agent at episode $t$, and define $\tilde{p}_{t+1}$, a second empirical estimation of $p$, such as in Eq.~\eqref{eq:proba_est}, but using the trajectories of the restarted agent instead. We denote by \(\smash{ \widetilde{P}^t_{\pi_t}} \) the stochastic matrix induced by the transition kernel \( (\tilde{p}_{t,n})_{n \in [N]} \) and the policy \( \pi_t \), as defined in Eq.~\eqref{eq:markov_chain_proba}. We take $\smash{\tilde{\rho}_{t+1} = \tilde{\rho}_t \widetilde{P}^t_{\pi_{t}} }$. Prop.~\ref{prop:estimate_rhot} provides a key result on the accuracy of this estimate and is the main result for extending the methods of the previous section. The proof and additional details are in App.~\ref{app:setting2}. The final method for this setting is MDPP-U (Mirror Descent for Periodic Policies - Unknown initial distributions) in Alg.~\ref{alg:main2} in App.~\ref{app:algo_scheme}. We state the main result of MDPP-U in Thm.~\ref{thm:main_periodic_regret_unknown_rhot}.
\begin{proposition}\label{prop:estimate_rhot}
For any $\delta \in (0,1)$, with probability at least $1-2\delta$,
       \[
  \textstyle{  \sum_{t=1}^T \|\rho_t - \tilde{\rho}_t \|_1 \leq (1-\alpha)^{-1} \big[3 N |\mathcal{X}| \sqrt{2 |\mathcal{A}| T \log\big(|\mathcal{X}| |\mathcal{A}| N T/\delta \big)} + 2 N |\mathcal{X}| \sqrt{2 T \log\big(N/\delta \big)}  \big].}
    \]
\end{proposition}

\textbf{Discussion.}
Deriving theoretical results in complex settings usually starts with simplifying assumptions. In the online infinite-horizon setting, non-asymptotic guarantees have so far been obtained only under known dynamics, since estimating them without resets can lead to errors accumulating linearly as agents drift from their initial distribution. In our analysis in Sec.~\ref{sec:algorithm}, we show that if the initial distribution is known at each episode, with $M>1$ homogeneous agents and a well-posed constrained problem, the dynamics can be reliably estimated from independent trajectories. When the initial distribution is unknown, additional exploration strategies are necessary. We demonstrate that allowing a single random agent to be reset each episode from a given distribution is sufficient to recover asymptotic guarantees even with unknown dynamics. This requirement is far weaker than resetting all agents simultaneously: for example, in the energy demand control scenario of Sec.~\ref{sec:introduction}, some consumers may have contracts permitting utilities to reset their water-heater temperature daily, whereas resetting all devices is impractical. Moreover, we believe that our analysis can provide insights for designing less constrained exploration strategies in future works.

Regarding the regret dependencies in Thm.~\ref{thm:main_periodic_regret_unknown_rhot}, since the MDPP-U algorithm does not assume knowledge of the initial distribution at each episode, an additional quantity must be learned. This results in an increased dependence on the mixing time in the regret bound. Nevertheless, this dependence remains polynomial. 

\begin{theorem}\label{thm:main_periodic_regret_unknown_rhot}
    Running MDPP-U (Alg.~\ref{alg:main2}) for $T$ episodes in an MDP with unknown transition kernel $p$ and initial state-action distribution $\rho$, against adversarial objectives $\smash{F_t := \sum_{n=1}^N f_{t,n}}$, where each $f_{t,n}$ is convex and $\ell$-Lipschitz with respect to the norm $\|\cdot\|_1$, with parameter $\alpha$, the optimal choice of $\eta$, \( \|\nabla \psi(\mu_t)\|_{1, \infty} \leq \Psi \), guarantees that, for any periodic policy $\pi$, with high probability, 
    \[
\textstyle{R_T(\pi) \leq \tilde{O} \big(  \ell N^3 (1-\alpha)^{-1} |\mathcal{X}|^{5/4} |\mathcal{A}|^{1/4} \Psi^{1/2} T^{3/4} +  \ell N^4 (1-\alpha)^{-2} |\mathcal{X}|^{2} \sqrt{|\mathcal{A}| T}   \big).}
\]
\end{theorem}

\section{Experiments}\label{sec:experiments}
We evaluate MDPP-K and MDPP-U on the \emph{max-entropy} and \emph{obstacle} tasks from \cite{geist2022}, without resets and with the added goal of returning agents to the initial state at the end of the episode. These tasks use a fixed objective. Adversarial MDPs are harder to implement due to the challenge of finding optimal stationary policies and limited benchmarks. The tasks under consideration could represent, for example, autonomous robots in a warehouse aiming to maximize task completion while ensuring they return to their charging stations at the end of each episode. For instance, in the \emph{obstacle} experiment, the robot’s objective is to reach a yellow target (representing a task) and then autonomously return to its initial state (the charging station) by the episode’s end.\footnote{All the code to reproduce the empirical results
is available at: \url{https://github.com/biancammoreno/OnlineMDPConstraintsLaw/}.}

We compare our periodic framework with the episodic framework, using the Bonus O-MD-CURL baseline from \cite{moreno_icml}, adding a regularization term $\gamma \| \rho_t - \rho\|_1$ to the objective of the episodic algorithm to allow fair comparison. Note that incorporating a $L_1$ regularizer with respect to the final distribution is only possible for an episodic algorithm in the convex RL case, as this regularizer is not a linear function of the state-action distribution. 
\begin{figure}
    \centering
    \begin{minipage}{.5\textwidth}
    \centering
    \subfigure[]{ \includegraphics[scale=0.2]{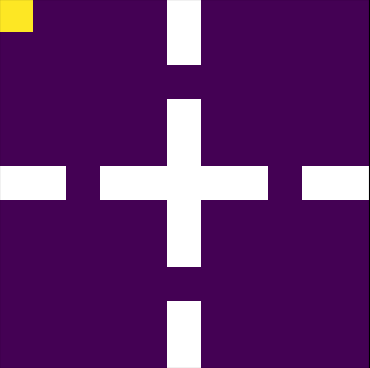}}
     \subfigure[]{ \includegraphics[scale=0.2]{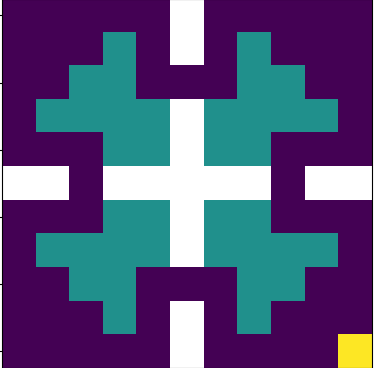}}
     \caption{[left] Initial agent dist.; [right] \textit{Obstacles} (reward in yellow, constraints in blue).}
        \label{fig:auxiliary_images}
    \end{minipage}%
    \begin{minipage}{0.5\textwidth}
        \centering
     \subfigure[]{ \input{img/max_entropy/periodic_regret_entropy_max.txt}}
     \subfigure[]{ \input{img/constrained/periodic_regret_constrained.txt}}
    \caption{Periodic regret for each environment.}
    \label{fig:regrets}
    \end{minipage}
\end{figure}

The environment is an $11 \times 11$ four-room grid world, where rooms are connected by single doors. The agent moves in cardinal directions or stays in place: $x_{n+1} = x_n + a_n + \varepsilon_n$, where $\varepsilon_n$ adds random displacement to a neighboring state. The initial state distribution is a Dirac located in the upper-left corner of the grid, as in Fig.~\ref{fig:auxiliary_images} [left]. We set the learning rate $\eta = 0.01$, the dual learning rate $\eta_\lambda = 0.01$, the contraction parameter $\alpha = 0.1$, and $\gamma = 1000$. \textit{Max-entropy:} The objective is to maximize entropy, aiming to reach a uniform distribution over the state space. At each step, $f_{t,n}(\mu_n^{\pi,\rho_t}) := \langle \mu_n^{\pi,\rho_t}, \log(\mu_n^{\pi,\rho_t}) \rangle$, and we take $N=40$. \textit{Obstacles:} The goal is to concentrate the state distribution on the yellow target in Fig.~\ref{fig:auxiliary_images} [right] while avoiding the constraint states in blue. The objective function is defined as \(\textstyle{ f_{t,n}(\mu_n^{\pi,\rho_t}) := -\langle r, \mu_n^{\pi,\rho_t} \rangle + (\langle \mu_n^{\pi,\rho_t}, c \rangle)^2 }\), where \(\smash{ r, c \in \mathbb{R}^{|\mathcal{X}| \times |\mathcal{A}|}_+ }\). Here, \( r \) and \(c\) are zero everywhere except at the target and constraint states respectively, and $N = 80$.
 
\begin{figure}[t!]
\centering

\subfigure[$n = 1$]{ \includegraphics[scale=0.18]{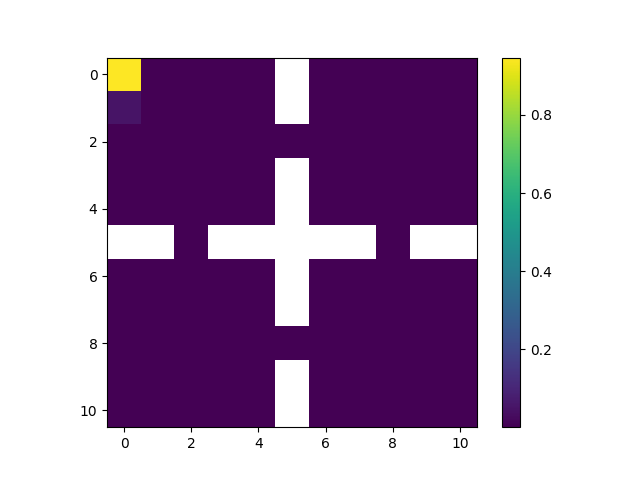}}
\subfigure[$n = 14$]{\includegraphics[scale=0.18]{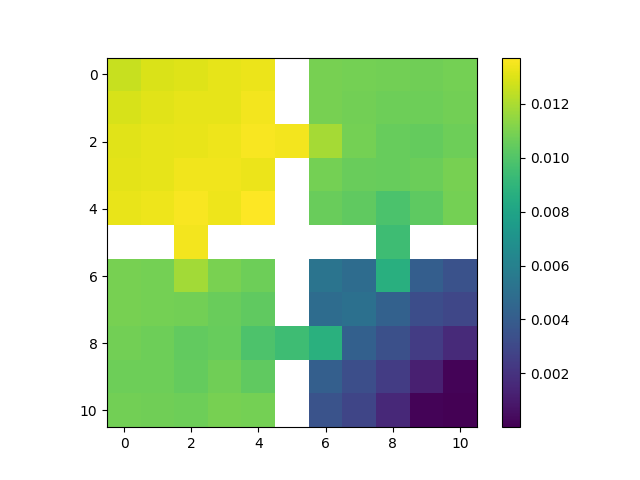}}
\subfigure[$n = 24$]{\includegraphics[scale=0.18]{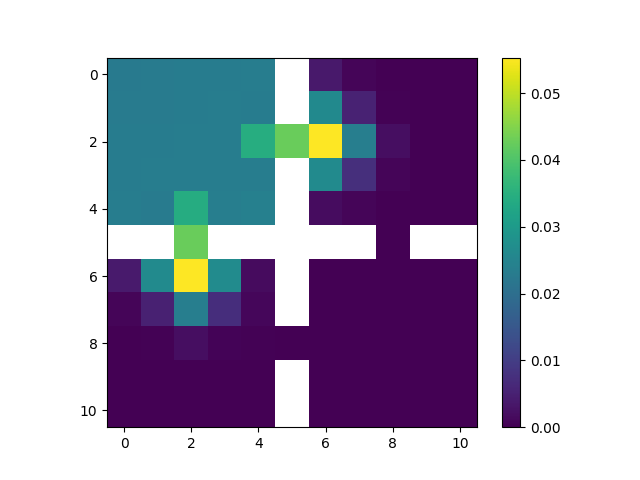}}
\subfigure[$n = 30$]{\includegraphics[scale=0.18]{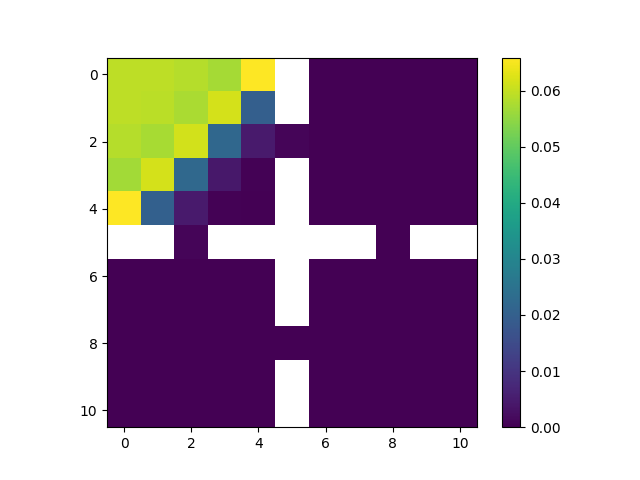}}
\subfigure[$n = 40$]{\includegraphics[scale=0.18]{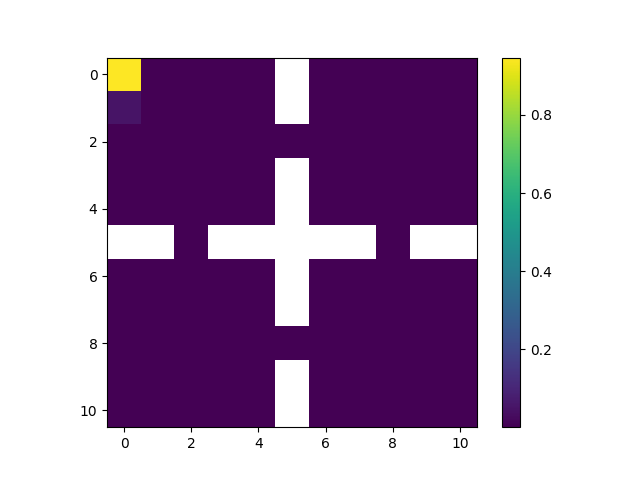}}

\subfigure[$n = 1$]{\includegraphics[scale=0.18]{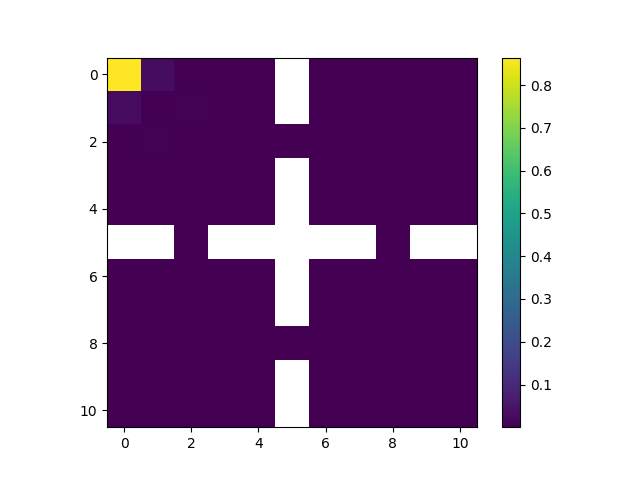}}
\subfigure[$n = 14$]{\includegraphics[scale=0.18]{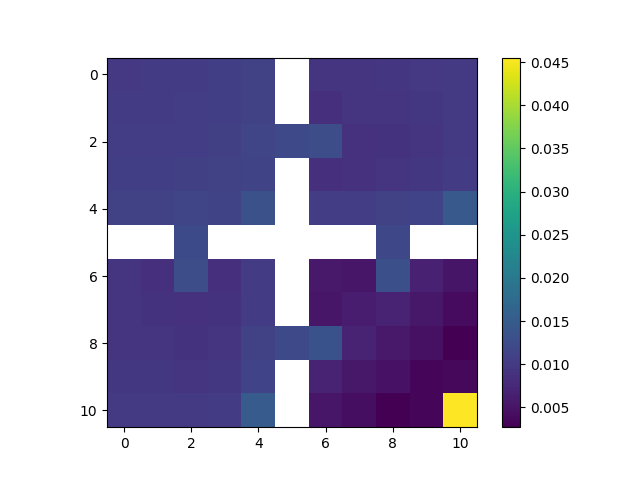}}
\subfigure[$n = 24$]{\includegraphics[scale=0.18]{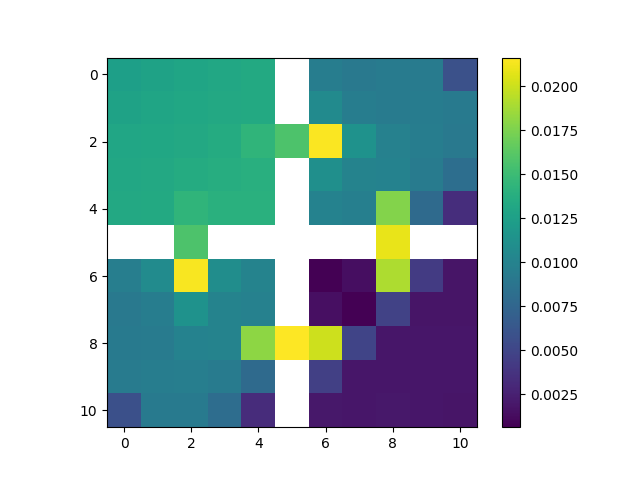}}
\subfigure[$n = 30$]{\includegraphics[scale=0.18]{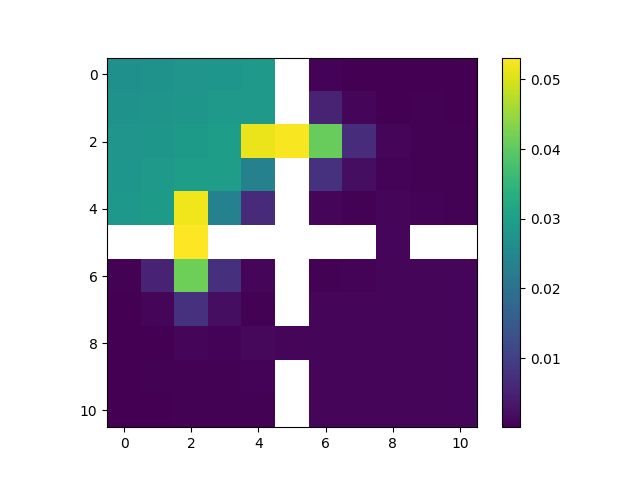}}
\subfigure[$n = 40$]{\includegraphics[scale=0.18]{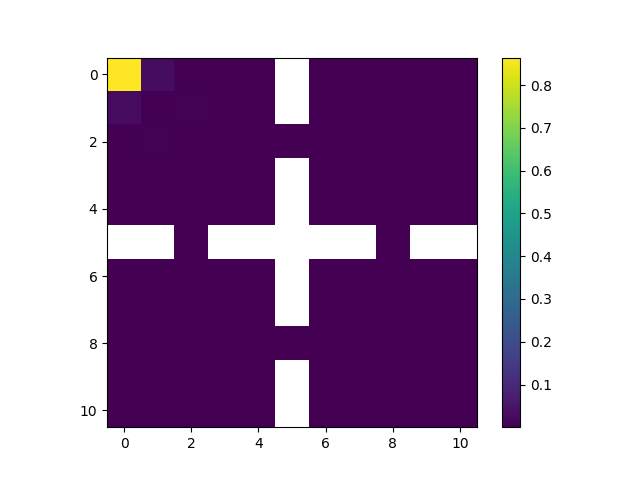}}

    \caption{\emph{Max-entropy}: state dist. after $5000$ eps. of MDPP-K [up] and Bonus MD-CURL [down].}
    \label{fig:max_entropy}
\end{figure}


\begin{figure}[t!]
\centering
\subfigure[$n = 1$]{ \includegraphics[scale=0.17]{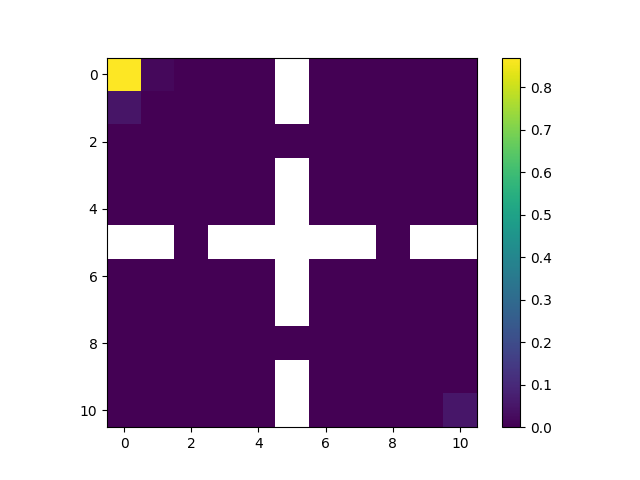}}
\subfigure[$n = 20$]{ \includegraphics[scale=0.17]{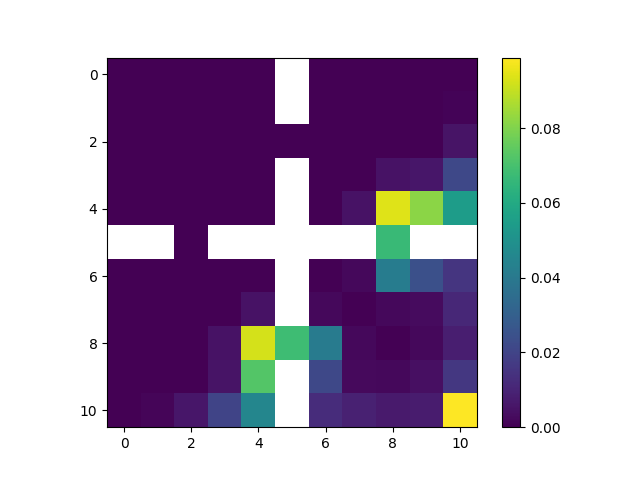}}
\subfigure[$n = 30$]{ \includegraphics[scale=0.17]{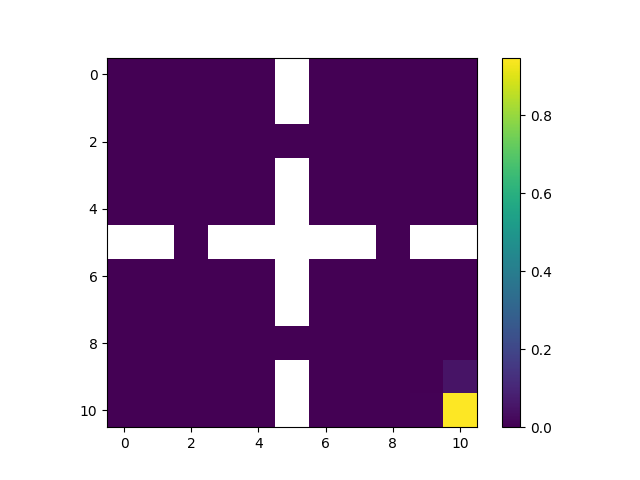}}
\subfigure[$n = 70$]{ \includegraphics[scale=0.17]{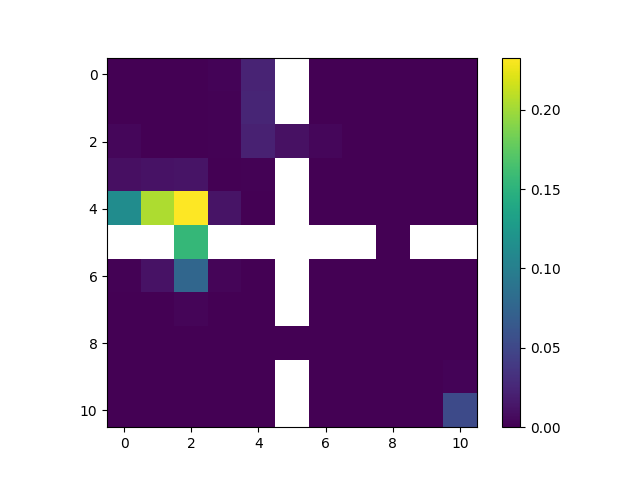}}
\subfigure[$n = 80$]{ \includegraphics[scale=0.17]{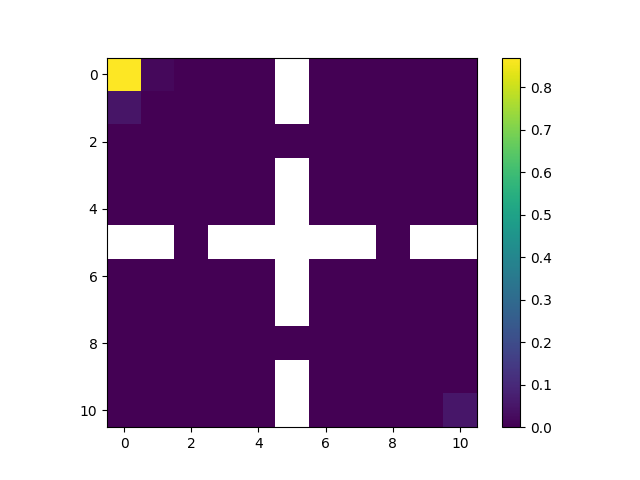}}

\subfigure[$n = 1$]{ \includegraphics[scale=0.17]{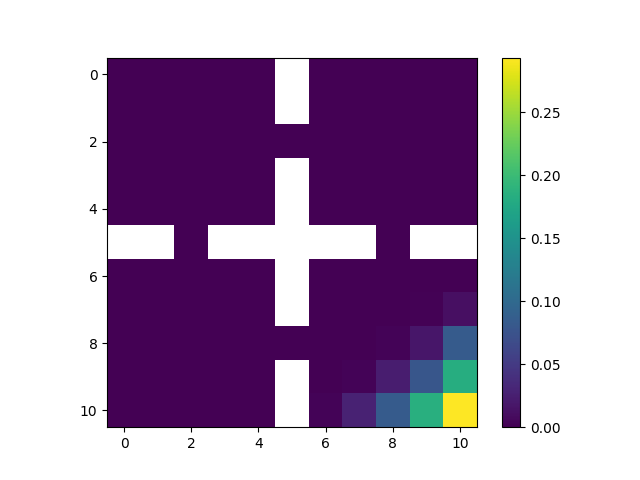}}
\subfigure[$n = 20$]{ \includegraphics[scale=0.17]{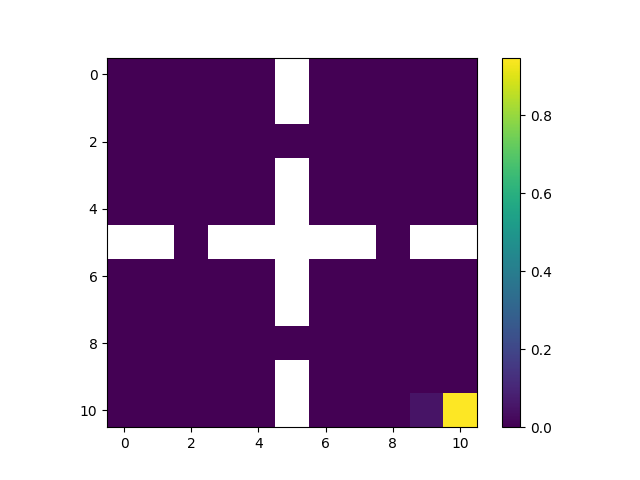}}
\subfigure[$n = 40$]{ \includegraphics[scale=0.17]{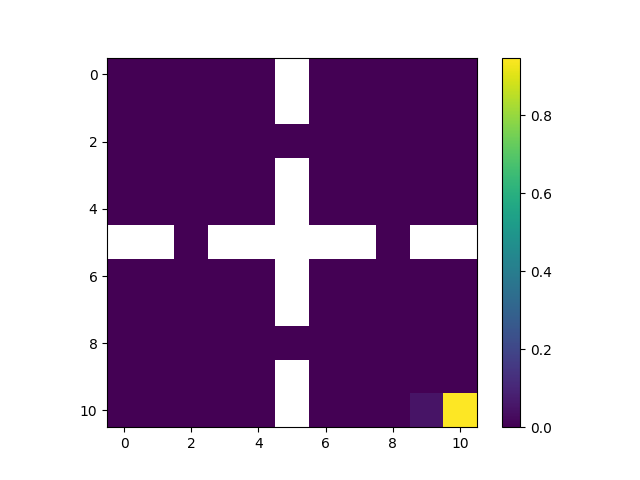}}
\subfigure[$n = 70$]{ \includegraphics[scale=0.17]{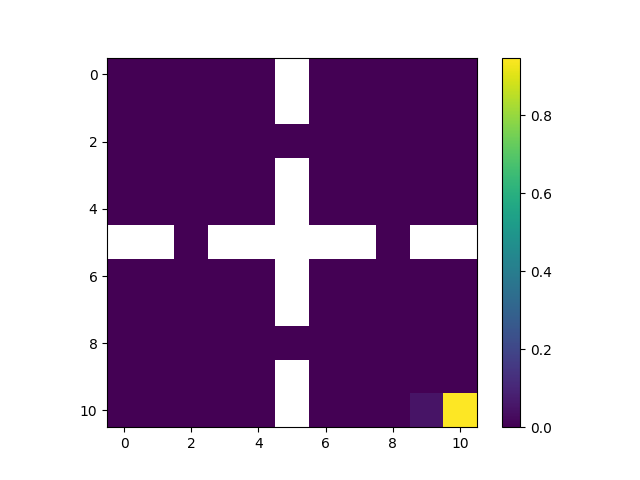}}
\subfigure[$n = 80$]{ \includegraphics[scale=0.17]{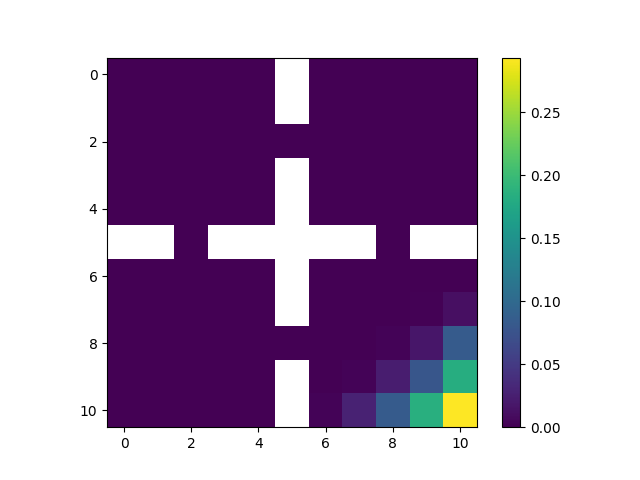}}
    
    \caption{\emph{Obstacles}: state dist. after $1000$ eps. for MDPP-K [up] and Bonus MD-CURL [down].}
    \label{fig:constrained}
\end{figure}

We plot the state distributions at different time steps in Fig.~\ref{fig:max_entropy}, and Fig.~\ref{fig:constrained}, comparing MDPP-K [up] with Bonus O-MD-CURL [down]. For \emph{max-entropy}, both methods lead agents to reach their goals midway through the episode and return to the initial state by the end. For the \emph{obstacle} task, MDPP-K reaches the target and then returns to the initial state, whereas the episodic approach becomes trapped in the target square. To compute the optimal periodic policy in the periodic regret, we solve the constrained problem in Eq.~\eqref{offline_problem} assuming known dynamics. In Fig.~\ref{fig:regrets} we see that MDPP-K and MDPP-U both achieve sub-linear regret, with MDPP-U suffering higher regret as it also has to estimate $\rho_t$. In contrast, Bonus O-MD-CURL incurs linear regret on both tasks.

Regularizers perform poorly in an episodic algorithm because such algorithms assume agents are reset at the end of each episode, an assumption that does not hold in our experiments. This leads to practical errors and makes tuning the regularization parameter difficult. For \emph{max-entropy}, we observe that the regularizer $\gamma \|\rho_t - \rho\|_1$ is too strong, preventing correct optimization, while that in the \emph{obstacle} task, it is to weak, failing to ensure a return to the initial state. Since agents actually start from a different distribution $\rho_t$ rather than the assumed $\rho$, a fixed regularizer is ineffective; an adaptive $\gamma_t$ would be needed, but computing it is infeasible without known dynamics. Our method, in contrast, effectively implements an implicit adaptive regularization, handling imperfect resets by automatically adjusting the Lagrange multipliers at each episode.


\acks{This work was supported by funding from the French government, managed by the National Research Agency (ANR), under the France 2030 program, reference ANR-23-IACL-0006.}

\newpage
\bibliography{bib}

\begin{thebibliography}{44}
\providecommand{\natexlab}[1]{#1}
\providecommand{\url}[1]{\texttt{#1}}
\expandafter\ifx\csname urlstyle\endcsname\relax
  \providecommand{\doi}[1]{doi: #1}\else
  \providecommand{\doi}{doi: \begingroup \urlstyle{rm}\Url}\fi

\bibitem[Bach(2024)]{francis}
Francis Bach.
\newblock \emph{Learning Theory from First Principles}.
\newblock The MIT Press, 2024.
\newblock ISBN 9780262049443.

\bibitem[Bartlett and Tewari(2009)]{regal}
Peter~L. Bartlett and Ambuj Tewari.
\newblock Regal: a regularization based algorithm for reinforcement learning in weakly communicating mdps.
\newblock In \emph{Proceedings of the Twenty-Fifth Conference on Uncertainty in Artificial Intelligence}, page 35–42, 2009.

\bibitem[Beck and Teboulle(2003)]{MD}
Amir Beck and Marc Teboulle.
\newblock Mirror descent and nonlinear projected subgradient methods for convex optimization.
\newblock \emph{Oper. Res. Lett.}, 31\penalty0 (3):\penalty0 167–175, may 2003.
\newblock ISSN 0167-6377.
\newblock \doi{10.1016/S0167-6377(02)00231-6}.

\bibitem[Bhavya et~al.(2024)Bhavya, Treven, Dorfler, Coros, and Krause]{neorl}
Bhavya, Lenart Treven, Florian Dorfler, Stelian Coros, and Andreas Krause.
\newblock Neorl: Efficient exploration for nonepisodic rl.
\newblock In A.~Globerson, L.~Mackey, D.~Belgrave, A.~Fan, U.~Paquet, J.~Tomczak, and C.~Zhang, editors, \emph{Advances in Neural Information Processing Systems}, volume~37, pages 74966--74998, 2024.

\bibitem[Bouchard et~al.(2010)Bouchard, Elie, and Imbert]{bouchard_elie}
Bruno Bouchard, Romuald Elie, and Cyril Imbert.
\newblock Optimal control under stochastic target constraints.
\newblock \emph{SIAM J. Control Optim.}, 48\penalty0 (5):\penalty0 3501–3531, February 2010.
\newblock ISSN 0363-0129.

\bibitem[Bourdais et~al.(2025)Bourdais, Oudjane, and Russo]{bourdais}
Thibaut Bourdais, Nadia Oudjane, and Francesco Russo.
\newblock An entropy penalized approach for stochastic optimization with marginal law constraints. complete version, 2025.

\bibitem[Brafman and Tennenholtz(2003)]{rmax}
Ronen~I. Brafman and Moshe Tennenholtz.
\newblock R-max - a general polynomial time algorithm for near-optimal reinforcement learning.
\newblock \emph{J. Mach. Learn. Res.}, 3:\penalty0 213–231, March 2003.

\bibitem[Castiglioni et~al.(2024)Castiglioni, Celli, Marchesi, Romano, and Gatti]{castiglioni}
Matteo Castiglioni, Andrea Celli, Alberto Marchesi, Giulia Romano, and Nicola Gatti.
\newblock A unifying framework for online optimization with long-term constraints.
\newblock In \emph{Proceedings of the 36th International Conference on Neural Information Processing Systems}, 2024.
\newblock ISBN 9781713871088.

\bibitem[Chandrasekaran and Tewari(2021)]{communicating}
Gautam Chandrasekaran and Ambuj Tewari.
\newblock Learning in online mdps: is there a price for handling the communicating case?
\newblock In \emph{Conference on Uncertainty in Artificial Intelligence}, 2021.

\bibitem[Chen et~al.(2022)Chen, Jain, and Luo]{chen2022learninginfinitehorizonaveragerewardmarkov}
Liyu Chen, Rahul Jain, and Haipeng Luo.
\newblock Learning infinite-horizon average-reward markov decision processes with constraints.
\newblock In \emph{International Conference on Machine Learning}, 2022.

\bibitem[Coffman et~al.(2023)Coffman, Bušić, and Barooah]{busic_energy_2023}
Austin Coffman, Ana Bušić, and Prabir Barooah.
\newblock A unified framework for coordination of thermostatically controlled loads.
\newblock \emph{Automatica}, 152:\penalty0 111002, 2023.

\bibitem[Dai et~al.(2022)Dai, Luo, and Chen]{dai_ftrl}
Yan Dai, Haipeng Luo, and Liyu Chen.
\newblock Follow-the-perturbed-leader for adversarial markov decision processes with bandit feedback.
\newblock In \emph{Advances in Neural Information Processing Systems}, volume~35, pages 11437--11449, 2022.

\bibitem[Daudin(2020)]{daudin}
Samuel Daudin.
\newblock Optimal control of diffusion processes with terminal constraint in law.
\newblock \emph{Journal of Optimization Theory and Applications}, 195:\penalty0 1 -- 41, 2020.

\bibitem[Ding and Lavaei(2023)]{ding2020provably}
Yuhao Ding and Javad Lavaei.
\newblock Provably efficient primal-dual reinforcement learning for cmdps with non-stationary objectives and constraints.
\newblock In \emph{AAAI Conference on Artificial Intelligence and Conference on Innovative Applications of Artificial Intelligence}, 2023.

\bibitem[Efroni et~al.(2020)Efroni, Mannor, and Pirotta]{efroni2020explorationexploitation}
Yonathan Efroni, Shie Mannor, and Matteo Pirotta.
\newblock Exploration-exploitation in constrained mdps, 2020.

\bibitem[Even-Dar et~al.(2009)Even-Dar, Kakade, and Mansour]{even-dar}
Eyal Even-Dar, Sham.~M. Kakade, and Yishay Mansour.
\newblock Online markov decision processes.
\newblock \emph{Mathematics of Operations Research}, 34\penalty0 (3):\penalty0 726--736, 2009.
\newblock ISSN 0364765X, 15265471.

\bibitem[Eysenbach et~al.(2018)Eysenbach, Gu, Ibarz, and Levine]{eysenbach2017leavetracelearningreset}
Benjamin Eysenbach, Shixiang Gu, Julian Ibarz, and Sergey Levine.
\newblock Leave no trace: Learning to reset for safe and autonomous reinforcement learning.
\newblock In \emph{International Conference on Learning Representations}, 2018.

\bibitem[F{\"o}llmer and Leukert(1999)]{quantile}
Hans F{\"o}llmer and Peter Leukert.
\newblock Quantile hedging.
\newblock \emph{Finance and Stochastics}, 3\penalty0 (3):\penalty0 251--273, 1999.

\bibitem[Geist et~al.(2022)Geist, P\'{e}rolat, Lauri\`{e}re, Elie, Perrin, Bachem, Munos, and Pietquin]{geist2022}
Matthieu Geist, Julien P\'{e}rolat, Mathieu Lauri\`{e}re, Romuald Elie, Sarah Perrin, Oliver Bachem, R\'{e}mi Munos, and Olivier Pietquin.
\newblock Concave utility reinforcement learning: The mean-field game viewpoint.
\newblock In \emph{Proceedings of the 21st International Conference on Autonomous Agents and Multiagent Systems}, page 489–497, 2022.

\bibitem[Guo et~al.(2022)Guo, Langrené, Loeper, and and]{Guo}
Ivan Guo, Nicolas Langrené, Grégoire Loeper, and Wei~Ning and.
\newblock Portfolio optimization with a prescribed terminal wealth distribution.
\newblock \emph{Quantitative Finance}, 22\penalty0 (2):\penalty0 333--347, 2022.
\newblock \doi{10.1080/14697688.2021.1967432}.

\bibitem[Hazan et~al.(2019)Hazan, Kakade, Singh, and Van~Soest]{hazan2019}
Elad Hazan, Sham Kakade, Karan Singh, and Abby Van~Soest.
\newblock Provably efficient maximum entropy exploration.
\newblock In \emph{International Conference on Machine Learning}, volume~97, pages 2681--2691, 2019.

\bibitem[Jaksch et~al.(2008)Jaksch, Ortner, and Auer]{UCRL-2}
Thomas Jaksch, Ronald Ortner, and Peter Auer.
\newblock Near-optimal regret bounds for reinforcement learning.
\newblock \emph{J. Mach. Learn. Res.}, 11:\penalty0 1563--1600, 2008.

\bibitem[Jin et~al.(2020)Jin, Jin, Luo, Sra, and Yu]{jin_bandit}
Chi Jin, Tiancheng Jin, Haipeng Luo, Suvrit Sra, and Tiancheng Yu.
\newblock Learning adversarial {M}arkov decision processes with bandit feedback and unknown transition.
\newblock In \emph{International Conference on Machine Learning}, volume 119, pages 4860--4869, 2020.

\bibitem[Lattimore and Szepesvári(2020)]{lattimore}
Tor Lattimore and Csaba Szepesvári.
\newblock \emph{Bandit Algorithms}.
\newblock Cambridge University Press, 2020.

\bibitem[Levin et~al.(2006)Levin, Peres, and Wilmer]{mixing_times}
David~A. Levin, Yuval Peres, and Elizabeth~L. Wilmer.
\newblock \emph{{Markov chains and mixing times}}.
\newblock American Mathematical Society, 2006.

\bibitem[Luo et~al.(2021)Luo, Wei, and Lee]{luo_dilated_bonus}
Haipeng Luo, Chen-Yu Wei, and Chung-Wei Lee.
\newblock Policy optimization in adversarial mdps: Improved exploration via dilated bonuses.
\newblock In \emph{Advances in Neural Information Processing Systems}, volume~34, pages 22931--22942, 2021.

\bibitem[Moreno et~al.(2024)Moreno, Bregere, Gaillard, and Oudjane]{pmlr-v238-moreno24a}
Bianca Moreno, Margaux Bregere, Pierre Gaillard, and Nadia Oudjane.
\newblock Efficient model-based concave utility reinforcement learning through greedy mirror descent.
\newblock In \emph{International Conference on Artificial Intelligence and Statistics}, volume 238, pages 2206--2214, 02--04 May 2024.

\bibitem[Moreno et~al.(2025{\natexlab{a}})Moreno, Br\'{e}g\`{e}re, Gaillard, and Oudjane]{marinmoreno_energy}
Bianca Moreno, Margaux Br\'{e}g\`{e}re, Pierre Gaillard, and Nadia Oudjane.
\newblock ({O}nline) {C}onvex {O}ptimization for {D}emand-{S}ide {M}anagement: {A}pplication to {T}hermostatically {C}ontrolled {L}oads.
\newblock \emph{J. Optim. Theory Appl.}, 205\penalty0 (3), 2025{\natexlab{a}}.
\newblock ISSN 0022-3239.
\newblock \doi{10.1007/s10957-025-02658-9}.

\bibitem[Moreno et~al.(2025{\natexlab{b}})Moreno, Eldowa, Gaillard, Br{\'e}g{\`e}re, and Oudjane]{moreno_icml}
Bianca Moreno, Khaled Eldowa, Pierre Gaillard, Margaux Br{\'e}g{\`e}re, and Nadia Oudjane.
\newblock Online episodic convex reinforcement learning.
\newblock In \emph{Forty-second International Conference on Machine Learning}, 2025{\natexlab{b}}.

\bibitem[M\"{u}ller et~al.(2024)M\"{u}ller, Alatur, Cevher, Ramponi, and He]{muller2024truly}
Adrian M\"{u}ller, Pragnya Alatur, Volkan Cevher, Giorgia Ramponi, and Niao He.
\newblock Truly no-regret learning in constrained mdps.
\newblock In \emph{International Conference on Machine Learning}, 2024.

\bibitem[Neu et~al.(2010)Neu, Antos, Gy\"{o}rgy, and Szepesv\'{a}ri]{neu_2010}
Gergely Neu, Andras Antos, Andr\'{a}s Gy\"{o}rgy, and Csaba Szepesv\'{a}ri.
\newblock Online markov decision processes under bandit feedback.
\newblock In \emph{Advances in Neural Information Processing Systems}, volume~23, 2010.

\bibitem[Neu et~al.(2012)Neu, Gyorgy, and Szepesvari]{neu12}
Gergely Neu, Andras Gyorgy, and Csaba Szepesvari.
\newblock The adversarial stochastic shortest path problem with unknown transition probabilities.
\newblock In \emph{Proceedings of the Fifteenth International Conference on Artificial Intelligence and Statistics}, volume~22, pages 805--813, 2012.

\bibitem[Neu et~al.(2017)Neu, Jonsson, and Gómez]{neu2017unifiedviewentropyregularizedmarkov}
Gergely Neu, Anders Jonsson, and Vicenç Gómez.
\newblock A unified view of entropy-regularized markov decision processes.
\newblock In \emph{Conference on Neural Information Processing Systems}, 2017.

\bibitem[Pfeiffer(2020)]{pfeiffer_2020}
Laurent Pfeiffer.
\newblock Optimality conditions in variational form for non-linear constrained stochastic control problems.
\newblock \emph{Mathematical Control and Related Fields}, 10\penalty0 (3):\penalty0 493--526, 2020.
\newblock ISSN 2156-8472.
\newblock \doi{10.3934/mcrf.2020008}.

\bibitem[Pfeiffer et~al.(2021)Pfeiffer, Tan, and Zhou]{pfeiffer_2021}
Laurent Pfeiffer, Xiaolu Tan, and Yu-Long Zhou.
\newblock Duality and approximation of stochastic optimal control problems under expectation constraints.
\newblock \emph{SIAM Journal on Control and Optimization}, 59\penalty0 (5):\penalty0 3231--3260, 2021.
\newblock \doi{10.1137/20M1349886}.

\bibitem[Puterman(2014)]{puterman2014markov}
Martin~L Puterman.
\newblock \emph{Markov decision processes: discrete stochastic dynamic programming}.
\newblock John Wiley \& Sons, 2014.

\bibitem[Qiu et~al.(2020)Qiu, Wei, Yang, Ye, and Wang]{qiu2020}
Shuang Qiu, Xiaohan Wei, Zhuoran Yang, Jieping Ye, and Zhaoran Wang.
\newblock Upper confidence primal-dual reinforcement learning for cmdp with adversarial loss.
\newblock In \emph{Advances in Neural Information Processing Systems}, volume~33, pages 15277--15287, 2020.

\bibitem[Rosenberg and Mansour(2019)]{Rosenberg2019}
Aviv~A. Rosenberg and Y.~Mansour.
\newblock Online convex optimization in adversarial markov decision processes.
\newblock In \emph{International Conference on Machine Learning}, 2019.

\bibitem[Shalev-Shwartz(2012)]{shalev_oco}
Shai Shalev-Shwartz.
\newblock \emph{Online Learning and Online Convex Optimization}.
\newblock Now Publishers, 2012.
\newblock \doi{10.1561/2200000018}.

\bibitem[Sharma et~al.(2022{\natexlab{a}})Sharma, Ahmad, and Finn]{sharma2022statedistributionmatchingapproachnonepisodic}
Archit Sharma, Rehaan Ahmad, and Chelsea Finn.
\newblock A state-distribution matching approach to non-episodic reinforcement learning.
\newblock In \emph{International Conference on Machine Learning}, 2022{\natexlab{a}}.

\bibitem[Sharma et~al.(2022{\natexlab{b}})Sharma, Xu, Sardana, Gupta, Hausman, Levine, and Finn]{autonomous_rl}
Archit Sharma, Kelvin Xu, Nikhil Sardana, Abhishek Gupta, Karol Hausman, Sergey Levine, and Chelsea Finn.
\newblock Autonomous reinforcement learning: Formalism and benchmarking.
\newblock In \emph{International Conference on Learning Representations}, 2022{\natexlab{b}}.

\bibitem[Stradi et~al.(2024)Stradi, Germano, Genalti, Castiglioni, Marchesi, and Gatti]{stradi24}
Francesco~Emanuele Stradi, Jacopo Germano, Gianmarco Genalti, Matteo Castiglioni, Alberto Marchesi, and Nicola Gatti.
\newblock Online learning in {CMDP}s: Handling stochastic and adversarial constraints.
\newblock In \emph{International Conference on Machine Learning}, volume 235, pages 46692--46721, 2024.

\bibitem[Zhao et~al.(2022)Zhao, Li, and Zhou]{zhao}
Peng Zhao, Long-Fei Li, and Zhi-Hua Zhou.
\newblock Dynamic regret of online {M}arkov decision processes.
\newblock In \emph{International Conference on Machine Learning}, volume 162, pages 26865--26894, 2022.

\bibitem[Zimin and Neu(2013)]{neu_2013}
Alexander Zimin and Gergely Neu.
\newblock Online learning in episodic markovian decision processes by relative entropy policy search.
\newblock In \emph{Advances in Neural Information Processing Systems}, volume~26, 2013.

\end{thebibliography}

\newpage
\appendix

\hrule
\begin{center}
    {\huge \bfseries Appendix}
\end{center}
\hrule

\vspace*{0.5cm}

This appendix is organized as follows: 
\begin{itemize}[nosep, label={--}]
    \item Appendix~\ref{notations}: summary of the notations used in the main paper and in the analysis
    \item Appendix~\ref{app:related_work}: additional discussion on related work
    \item Appendix~\ref{app:algo_scheme}: algorithms schemes for MDPP-K and MDPP-U
    \item Appendix~\ref{app:auxiliary}: auxiliary lemmas
    \item Appendix~\ref{app:feasibility}: proofs of the feasibility results in the main paper
    \item Appendix~\ref{app:md_proof}: details on the upper bound of the Mirror descent regret term
    \item Appendix~\ref{app:main_result}: proof of the main result in Sec.~\ref{sec:algorithm} - MDPP-K
    \item Appendix~\ref{app:practical_sol}: details on the practical implementation of the algorithms
    \item Appendix~\ref{app:setting2}: proofs of Sec.~\ref{sec:algorithm_unknown_rhot} - MDPP-U
\end{itemize}

The code for reproducing the experiments in Sec.~\ref{sec:experiments} can be found at: \url{https://github.com/biancammoreno/OnlineMDPConstraintsLaw/}

\section{List of Notations}\label{notations}
In this appendix, we recall all notations that are used throughout the paper.

Below are generic notations for any finite sets $\mathcal{S}, \mathcal{D}$:

\begin{itemize}
    \item $[s] := \{1, \ldots, s\}$
    \item $|\mathcal{S}| = $ cardinality of a set $\mathcal{S}$
    \item $\Delta_\mathcal{S} := \left\{ v \in \mathbb{R}^{|\mathcal{S}|} \mid \sum_{i=1}^{|\mathcal{S}|} v(i) = 1 \right\}$
    \item $(\Delta_{\mathcal{S}})^\mathcal{D} := \left\{ v = (v_d)_{d \in [D]} \mid  v_d \in \mathbb{R}^{|\mathcal{S}|}, \; \sum_{i=1}^{|\mathcal{S}|} v_d(i) = 1, \forall d \in [D] \right\}$
    \item Given functions $f, g : \mathbb{N} \rightarrow [0, \infty)$ define $f(n) = O\big(g(n) \big) \Leftrightarrow \limsup_{n \to \infty} \frac{f(n)}{g(n)} < \infty$ 
    \item $f(n) = \tilde{O}(g(n))  \Leftrightarrow \limsup_{n \to \infty} \frac{f(n) }{g(n) \log^p(n)} < \infty$, for some power $p \in \mathbb{N}$
    \item $\|\cdot\|_1$ is the $L_1$ norm
    \item For all $v := (v_n)_{n \in [N]}$, such that $\smash{v_n \in \mathbb{R}^{\mathcal{X} \times \mathcal{A}}}$, we define $\|v\|_{\infty,1} := \sup_{n \in [N]} \|v_n\|_1.$
\end{itemize}

Below are notations related to the Markov decision process:

\begin{itemize}
    \item $\mathcal{X} $ finite state space
    \item $\mathcal{A}$ finite action space
    \item $N := $ episode length
    \item $n := $ time step index within an episode, $n \in [N]$
    \item $\pi := (\pi_n)_{n \in [N]}$ with, for all $n \in [N], x \in \mathcal{X}$, $\pi_n(\cdot|x) \in \Delta_\mathcal{A} =$ policy at time step $n$ conditioned on state $x$
    \item $\Pi := (\Delta_{\mathcal{A}})^{N \times \mathcal{X}} = $ space of policies (defined using the generic notations from above)
    \item $p := (p_n)_{n \in [N]} = $ true probability transition kernel with, for all $n \in [N], (x,a) \in \mathcal{X} \times \mathcal{A}$, $p_n(\cdot|x,a) \in \Delta_\mathcal{X}$
    \item $\rho \in \Delta_{\mathcal{X} \times \mathcal{A}}=$ target initial state-action distribution
    \item For any $\pi \in \Pi$, and $\nu \in \Delta_{\mathcal{X} \times \mathcal{A}}$, $\mu^{\pi,\nu} := (\mu_n^{\pi,\nu})_{n \in [N]}$, where for all $(x,a)$, $\mu_0^{\pi,\nu}(x,a) = \nu(x,a)$, and for all $n$,
    \[
\mu_{n+1}^{\pi, \nu}(x,a) := \sum_{x',a'} \mu_n^{\pi, \nu}(x',a') p_{n+1}(x|x', a') \pi_{n+1}(a|x)
    \]
    \item For any $\nu \in \Delta_{\mathcal{X} \times \mathcal{A}}$, and for any probability transition $p$, 
    \[
    \mathcal{M}^p_{\nu} :=  \bigg\{ \mu \big| \; \mu_0 = \nu \;, \sum_{a' \in \mathcal{A}} \mu_{n}(x',a') = \sum_{x \in \mathcal{X} , a \in \mathcal{A}} p_{n}(x'|x,a) \mu_{n-1}(x,a)\;, \forall x' \in\mathcal{X}, 
         \forall n \in [N] \bigg\} 
         \]
         \item  For all $n \in [N]$, for $y_n := (x_n,a_n)$, and for all policy $\pi \in \Pi$, $P^\pi_n(y_{n-1}, y_n) := p_n(x_n|y_{n-1}) \pi_n(a_n|x_n)$ 
         \item $P_\pi(y_0, y_N) := \sum_{y_1}P^\pi_1(y_0, y_1) \ldots \sum_{y_{N-1}} P^\pi_{N-1}(y_{N-2}, y_{N-1}) P^\pi_N(y_{N-1}, y_N)$
         \item For all $\nu \in \Delta_{\mathcal{X} \times \mathcal{A}}$, $\nu P_\pi(x,a) = \sum_{x',a'} \nu(x',a') P_\pi(x',a'; x,a)$
         \item $M$ = number of agents
         \item $\alpha$ = contraction parameter defined in Ass.~\ref{ass:contraction}
\end{itemize}

Below are notations related to the online settings:

\begin{itemize}
	\item $T :=$ number of episodes
	\item $t:= $ episode index, $t \in [T]$
		    \item $\pi_t :=$ policy played at episode $t$ associated to the state-action distribution sequence solution of Eq.~\eqref{iteration_md_solver} at episode $t$
		\item $\rho_t \in \Delta_{\mathcal{X} \times \mathcal{A}}=$ initial state-action distribution at episode $t$
	\item $\hat{p}_t$ empirical probability transition kernel estimated at episode $t$ following Eq.~\ref{eq:proba_est}
	\item For any $\pi \in \Pi$, and $\nu \in \Delta_{\mathcal{X} \times \mathcal{A}}$, $\hat{\mu}^{\pi,\nu}_t := (\hat{\mu}_{t,n}^{\pi,\nu})_{n \in [N]}$, where for all $(x,a)$, $\hat{\mu}_{t,0}^{\pi,\nu}(x,a) = \nu(x,a)$, and for all $n$,
    \[
\hat{\mu}_{t,n+1}^{\pi, \nu}(x,a) := \sum_{x',a'} \hat{\mu}_{t,n}^{\pi, \nu}(x',a') \hat{p}_{t,n+1}(x|x', a') \pi_{n+1}(a|x)
    \]
    \item $\mathcal{M}_t := \mathcal{M}^{\hat{p}_t}_{\rho_t}$
    \item $\mu_t := \hat{\mu}^{\pi_t, \rho_t}_t$, the state-action distribution sequence solution of Eq.~\eqref{iteration_md_solver} at episode $t$
    \item $\ell_t := \nabla F_t(\mu_t)$
	\item For all $t \in [T], n \in [N]$, $f_{t,n} : \mathbb{R}^{\mathcal{X} \times \mathcal{A}} \rightarrow \mathbb{R}$ any convex and $\ell$-Lipschitz function with respect to $\|\cdot\|_1$
	\item $F_t := \sum_{n=1}^N f_{t,n}$ 
	\item $\gamma :=$ penalization parameter of the periodic regret defined in Eq.~\eqref{eq:periodic_regret}
    \item $\eta = $ online mirror descent learning rate
    \item $\eta_\lambda = $ dual learning rate
    \item $\lambda = $ Lagrange multiplier
    \item $\tilde{p}_t = $ empirical probability transition kernel estimated using the trajectories of the restarted agent in framework 2
    \item $\tilde{\rho}_t= $ estimate of the initial state-action distribution at episode $t$ used in framework 2
\end{itemize}

\section{Related Work}\label{app:related_work}

Table~\ref{table:comparisons} provides an overview of existing approaches to online MDPs with adversarial losses in both episodic and infinite-horizon settings, highlighting the distinctions with our periodic framework. Below, we discuss additional RL works that share certain similarities with our approach.

\paragraph{Autonomous RL:} Many lines of work in RL address continual learning tasks from a practical point of view, without offering theoretical guarantees; see \cite{autonomous_rl} for an overview. Most of the deep RL approaches to continual learning concentrate on learning to reset. For instance, \cite{eysenbach2017leavetracelearningreset, sharma2022statedistributionmatchingapproachnonepisodic} propose ideas of state distribution matching similar to ours. These methods involve learning both a ``forward policy'' to optimize the task and a ``backward policy'' to drive the agent toward a predefined state distribution. However, they treat task optimization and distribution matching as distinct objectives, and do not assume a fixed episode length in which the agent must jointly minimize cumulative loss and return to the initial distribution by the episode's end, as we do. Furthermore, these works do not have the theoretical guarantees provided by our approach.

As for theoretical works addressing non-episodic settings, they typically focus on optimizing the infinite-horizon average loss \cite{rmax,UCRL-2,regal,neorl}, which differs from minimizing periodic regret. Moreover, these works assume a fixed loss across episodes, in contrast to our setting, which considers an adversarially varying sequence of losses. 

\paragraph{Stochastic control with constraints in law:} Some works on the literature of stochastic control investigate constraints on the probability distribution of the terminal state \cite{daudin,bouchard_elie,quantile,pfeiffer_2020,pfeiffer_2021}. However, these studies primarily focus on deriving necessary optimality conditions under known dynamics. As such, they do not address learning problems, nor do they consider online settings with adversarially changing losses as we do in this paper. Such control problems naturally arise in economics and finance, where an agent seeks to minimize a cost subject to constraints on the terminal distribution (e.g., achieving a specified distribution of terminal wealth \cite{Guo}), as well as in energy management applications, as detailed in the main paper \cite{bourdais}.

\section{Algorithm schemes for both settings}\label{app:algo_scheme}

We outline below the algorithms for computing a sequence of policies that achieve low periodic regret. The procedure is detailed for the two settings considered in the main paper: \textcolor{red}{Framework 1}, where $\rho_t$ is observed at the beginning of each episode; and \textcolor{blue}{Framework 2}, where $\rho_t$ is unobserved and must be estimated by the agent. Steps specific to \textcolor{red}{Framework 1 are highlighted in red}, while those specific to \textcolor{blue}{Framework 2 are highlighted in blue}. All remaining steps are similar for both settings.

\begin{algorithm}[H]
    \caption{\textcolor{red}{MDPP-K: Framework $1$: known initial distributions}}\label{alg:main1}
    \begin{algorithmic}
        \STATE {\bfseries Input:} number of episodes $T$; length of each episode $N$; number of agents $M$; initial state-action distribution $\rho$; initial policy sequence $\pi_1$
        \STATE {\bfseries Initialize:} $\rho_1 = \rho$; $N_{1,n}(x,a) = M_{1,n}(x'|x,a) = 0$; $\hat{p}_{1,n}(x'|x,a) = 1/|\mathcal{X}|$ for all $(n, x,a,x') \in [N] \times \mathcal{X} \times \mathcal{A} \times \mathcal{X}$
        \STATE Each agent $j \in [M]$ starts at $(x_{1,0}^j, a_{1,0}^j) \sim \rho(\cdot)$
        \FOR{$t = 1, \ldots, T$}
            \FOR{$n = 1, \ldots, N$}
                \STATE Environment samples new state $x_{t,n}^j \sim p_n(\cdot|x_{t,n-1}^j, a_{t,n-1}^j)$ for all agents $j \in [M]$
                \STATE Each agent $j \in [M]$ takes an action $a_{t,n}^j \sim \pi_{t,n}(\cdot|x_{t,n}^j)$
            \ENDFOR
            \STATE Uniformly sample an agent $j \in [M]$ to ensure independence of the observed trajectories, and observe its state-action trajectory $(x_{t,n}^j,a_{t,n}^j)_{n \in [N]}$
            \STATE Update the counts using the entire trajectory
            \begin{align*}
        &N_{t+1,n-1}(x_{t,n-1}^j, a_{t,n-1}^j) \gets N_{t,n-1}(x_{t,n-1}^j, a_{t,n-1}) + 1 \\
        &M_{t+1,n-1}(x_{t,n}^j| x_{t,n-1}^j, a_{t,n-1}^j) \gets M_{t,n-1}(x_{t,n}^j| x_{t,n-1}^j, a_{t,n-1}) + 1
    \end{align*}
    \STATE Update the estimate $\hat{p}_{t+1}$ based on the updated counts following Eq.~\eqref{eq:proba_est}
    \STATE Observe the objective function $F_t$ (full-information)
    \STATE Recall that for each agent $j \in [M]$, $(x_{t+1,0}^j, a_{t+1,0}^j) = (x_{t,N}^j, a_{t,N}^j)$
    \STATE \textcolor{red}{Observe $\rho_{t+1} = \rho_t P_{\pi_t}$ (the initial state-action distribution at episode $t+1$)}
    \STATE \textcolor{red}{Compute $\pi_{t+1}$ solving Eq.~\eqref{iteration_md_solver} using $\rho_{t+1}$, $\hat{p}_{t+1}$ and $F_t$}
    \ENDFOR
    \end{algorithmic}
\end{algorithm}

\newpage

\begin{algorithm}[H]
    \caption{\textcolor{blue}{MDPP-U: Framework $2$: unknown initial distributions}}\label{alg:main2}
    \begin{algorithmic}[1]
        \STATE {\bfseries Input:} number of episodes $T$; length of each episode $N$; number of agents $M$; initial state-action distribution $\rho$; initial policy sequence $\pi_1$; \textcolor{blue}{restarted agent, assumed here to be the agent $M$, but in practice it can be a different agent per episode}
        \STATE {\bfseries Initialize:} $\rho_1 = \rho$; \textcolor{blue}{$\tilde{\rho}_1 = \rho$}; $N_{1,n}(x,a) = M_{1,n}(x'|x,a) = 0$; \textcolor{blue}{$\widetilde{N}_{1,n}(x,a) = \widetilde{M}_{1,n}(x'|x,a) = 0$}; $\hat{p}_{1,n}(x'|x,a) = 1/|\mathcal{X}|$; \textcolor{blue}{$\tilde{p}_{1,n}(x'|x,a) = 1/|\mathcal{X}|$} for all $(n, x,a,x') \in [N] \times \mathcal{X} \times \mathcal{A} \times \mathcal{X}$
        \STATE Each agent $j \in [M]$ starts at $(x_{1,0}^j, a_{1,0}^j) \sim \rho(\cdot)$
        \FOR{$t = 1, \ldots, T$}
            \FOR{$n = 1, \ldots, N$}
                \STATE Environment samples new state $x_{t,n}^j \sim p_n(\cdot|x_{t,n-1}^j, a_{t,n-1}^j)$ for all agents $j \in [M]$
                \STATE Each agent $j \in [M]$ takes an action $a_{t,n}^j \sim \pi_{t,n}(\cdot|x_{t,n}^j)$
            \ENDFOR
            \STATE Uniformly sample an agent $j \in [M]$ to ensure independence of the observed trajectories, and observe its state-action trajectory $(x_{t,n}^j,a_{t,n}^j)_{n \in [N]}$
            \STATE Update the counts using the entire trajectory
            \begin{align*}
        &N_{t+1,n-1}(x_{t,n-1}^j, a_{t,n-1}^j) \gets N_{t,n-1}(x_{t,n-1}^j, a_{t,n-1}^j) + 1 \\
        &M_{t+1,n-1}(x_{t,n}^j| x_{t,n-1}^j, a_{t,n-1}^j) \gets M_{t,n-1}(x_{t,n}^j| x_{t,n-1}^j, a_{t,n-1}^j) + 1
    \end{align*}
    \STATE Update the estimate $\hat{p}_{t+1}$ based on the updated counts following Eq.~\eqref{eq:proba_est}
    \STATE Observe the objective function $F_t$ (full-information)
    \STATE Recall that for each agent \textcolor{blue}{$j \in [M-1]$}, $(x_{t+1,0}^j, a_{t+1,0}^j) = (x_{t,N}^j, a_{t,N}^j)$
    \STATE \textcolor{blue}{Update the second set of counts using the trajectory of agent $M$}
     \begin{align*}
        &\widetilde{N}_{t+1,n-1}(x_{t,n-1}^M, a_{t,n-1}^M) \gets \widetilde{N}_{t,n-1}(x_{t,n-1}^M, a_{t,n-1}^M) + 1 \\
        &\widetilde{M}_{t+1,n-1}(x_{t,n}^M| x_{t,n-1}^M, a_{t,n-1}^M) \gets \widetilde{M}_{t,n-1}(x_{t,n}^M| x_{t,n-1}^M, a_{t,n-1}^M) + 1
    \end{align*}
    \STATE \textcolor{blue}{Update the estimate $\tilde{p}_{t+1}$ based on the second set of updated counts following Eq.~\eqref{eq:proba_est}}
    \STATE \textcolor{blue}{Compute the initial state-action distribution estimate $\tilde{\rho}_{t+1} = \tilde{\rho}_t \widetilde{P}_{\pi_t}$}
    \STATE \textcolor{blue}{Reset agent $M$ at $(x_{t+1,0}^M, a_{t+1,0}^M) \sim \tilde{\rho}_{t+1}(\cdot)$}
    \STATE Compute $\pi_{t+1}$ solving Eq.~\eqref{iteration_md_solver} \textcolor{blue}{but with $\tilde{\rho}_{t+1}$ instead of $\rho_{t+1}$}, $\hat{p}_{t+1}$ and $F_t$
    \ENDFOR
    \end{algorithmic}
\end{algorithm}

\newpage

\section{Auxiliary Lemmas}\label{app:auxiliary}

\begin{lemma}[Lem. 17 of \citealp{UCRL-2}]\label{lemma:proba_difference}
    For any $0 < \delta < 1$,  with a probability of at least $1-\delta$,
        \[ 
        {\|p_n(\cdot|x,a) - \hat{p}_n^t(\cdot|x,a)\|_1 \leq 
        \sqrt{\frac{2 |\mathcal{X}|\log\left(\frac{|\mathcal{X}| |\mathcal{A}| N T}{\delta}\right)}{\max{\{1,N_{n-1}^t(x,a)}\}} } }
        \]
    holds simultaneously for all $(t,n,x,a) \in [T] \times [N] \times \mathcal{X} \times \mathcal{A}$.
\end{lemma}

\begin{lemma}\label{lemma:bound_norm_mu_diff_mu0}
        For any strategy $\pi \in (\Delta_\mathcal{A})^{\mathcal{X} \times N}$, for any two probability kernels $p = (p_n)_{n \in [N]}$ and $q = (q_n)_{n \in [N]}$ such that $p_n, q_n: \mathcal{X} \times \mathcal{A} \times \mathcal{X} \rightarrow [0,1]$, for any two initial state-action distributions $\nu, \nu' \in \Delta_{\mathcal{X} \times \mathcal{A}}$, and for all $n \in [N]$,
    \begin{equation*}
    \begin{split}
        &\|\mu_n^{\pi,p, \nu} - \mu_n^{\pi, q, \nu'}\|_1 \leq \\
        &\quad \quad \sum_{i=0}^{n-1} \sum_{x,a} \mu_i^{\pi, p, \nu}(x,a) \| p_{i+1}(\cdot| x,a) - q_{i+1} (\cdot| x,a) \|_1 + \|\nu - \nu'\|_1,
    \end{split}
\end{equation*}
where $\mu^{\pi,p,\nu}$ is the state-action distribution sequence induced by policy $\pi$, in the MDP with probability transition $p$, with initial state-action distribution $\nu$.
\end{lemma}
\begin{proof}
    Using that for $n \geq 1$, 
    \[
    \mu_n^{\pi, p, \nu}(x,a) = \sum_{x',a'} \mu_{n-1}^{\pi, p, \nu'}(x',a')p_n(x|x',a') \pi_n(a|x),
    \]
    we can show that (see for example Lemma D.1 from \cite{pmlr-v238-moreno24a})
    \[
    \|\mu_n^{\pi,p, \nu} - \mu_n^{\pi, q, \nu'}\|_1 \leq \sum_{x,a} \mu_{n-1}^{\pi, p, \nu}(x,a) \| p_{n}(\cdot| x,a) - q_{n} (\cdot| x,a) \|_1 + \|\mu_{n-1}^{\pi, p, \nu'} - \mu_{n-1}^{\pi, q, \nu} \|_1.
    \]
    Hence, by induction and using that $\mu_0^{\pi, p, \nu'} = \nu'$ and $\mu_0^{\pi, q, \nu} = \nu$, we conclude the proof.
\end{proof}

\begin{lemma}\label{lemma:mu_bonus}
    For any policy $\pi \in \Pi$, for any initial state-action distribution $\nu \in \Delta_{\mathcal{X} \times \mathcal{A}}$, for any $\delta \in (0,1)$, for any episode $t \in [T]$, we have that with probability at least $1-\delta$,
    \[
    \|\nu (P_{\pi} - \widehat{P}_{\pi}^t) \|_1 \leq \langle \nu  \widehat{P}_{\pi}^t, b_t \rangle = \langle \hat{\mu}_{t,N}^{\pi, \nu}, b_t \rangle,
    \]
    where $b_t$ is the bonus vector defined for each $n \in \{0,1,\ldots,N-1\}$, $(x,a) \in \mathcal{X} \times \mathcal{A}$, as
    \[
    b_{t,n}(x,a) := \frac{C_\delta}{\sqrt{\max\{1, N_{t,n}(x,a) \}}},
    \]
    with 
    \begin{equation}\label{eq:c_delta}
        C_\delta = \sqrt{2 |\mathcal{X}| \log\bigg(\frac{|\mathcal{X}| |\mathcal{A}| N T }{\delta} \bigg)}.
    \end{equation}
\end{lemma}
\begin{proof}
    Rewriting the initial $L_1$ norm in terms of the state-action distributions we have that
    \begin{equation*}
        \begin{split}
            \|\nu (P_{\pi} - \widehat{P}_{\pi}^t) \|_1  &= \|\mu_N^{\pi, \nu} - \hat{\mu}_{t,N}^{\pi, \nu} \|_1 \\
            &\underbrace{\leq}_{\text{Lemma~\ref{lemma:bound_norm_mu_diff_mu0}}} \sum_{i=0}^{N-1} \sum_{x,a} \hat{\mu}^{ \pi, \nu}_{t,i}(x,a) \|p_{i+1}(\cdot|x,a) - \hat{p}^t_{i+1}(\cdot|x,a) \|_1 \\
            &\underbrace{\leq}_{\text{Lemma~\ref{lemma:proba_difference}}} \sum_{i=0}^{N-1} \sum_{x,a} \hat{\mu}_{t,i}^{\pi,\nu}(x,a) \frac{C_\delta}{\sqrt{\max\{1, N_{t,i}(x,a) \}}} \\
            & = \langle \hat{\mu}^{\pi,\nu}_t, b_t \rangle.
        \end{split}
    \end{equation*}
\end{proof}

\begin{proposition}\label{prop:mdp_martingale}
   For \( t \in [T] \), let \( (x_{t,n}, a_{t,n})_{n \in [N]} \) represent the trajectory of an agent following policy \( \pi_t \), with the initial state-action pair sampled independently from previous trajectories according to a distribution \( \nu_t \in \Delta_{\mathcal{X} \times \mathcal{A}} \). Define the filtration \( \mathcal{F}_t := \sigma\big( (x_{s,n}, a_{s,n})_{n \in [N]}, s < t \big) \), capturing the sigma-algebra generated by all observations prior to the start of episode \( t \). Assume that both \( \pi_t \) and \( \nu_t \) are \( \mathcal{F}_t \)-measurable. Let \( N_{t,n}(x, a) := \sum_{s=1}^{t-1} \mathds{1}_{\{x_{s,n} = x, a_{s,n} = a\}} \) denote the number of times the state-action pair \( (x, a) \) has been visited at time step \( n \) up to the start of episode \( t \), and define \( \xi_{t,n}(x, a) \leq \frac{C}{\sqrt{\max\{1,N_{t,n}(x, a)}\}} \) for all \( (t, n, x, a) \), where \( C \) is a fixed constant, and $\xi_{t,n}(x,a)$ is $\mathcal{F}_t$-measurable. We also assume $\xi_{t,n}(x,a) \leq c$, for a fixed constant $c$ that may differ from $C$. Hence, for all $n \in [N]$, for all $\delta \in (0,1)$, with probability at least $1-\delta$,
    \[
    \sum_{t=1}^T \langle \mu^{\pi_t, \nu_t}_n, \xi_{t,n} \rangle := \sum_{t=1}^T \sum_{x,a} \mu^{\pi_t,\nu_t}_n(x,a) \xi_{t,n}(x,a) \leq   3 C \sqrt{|\mathcal{X}| |\mathcal{A}| T} +  c |\mathcal{X}| \sqrt{2 T \log\bigg(\frac{N}{\delta} \bigg)}.
    \]
\end{proposition}
\begin{proof}
    For a fixed $(x,a) \in \mathcal{X} \times \mathcal{A}$, we have from Lemma 19 of \cite{UCRL-2} that, 
    \[
    \sum_{t=1}^T \mathds{1}_{\{x_{t,n}=x,a_{t,n}=a\}} \xi_{t,n}(x,a) \leq C \sum_{t=1}^T  \frac{\mathds{1}_{\{x_{t,n}=x,a_{t,n}=a\}}}{\sqrt{\max\{N_{t,n}(x,a),1\}}} \leq 3 C \sqrt{N_{T,n}(x,a)}.
    \]

    Summing the inequality above over all possible $(x,a)$, using Jensen's inequality and that $\sum_{x,a} N_{T,n}(x,a) = T$, we obtain that
    \begin{equation}\label{eq:mdp_term_1}
        \sum_{t=1}^T \sum_{x,a} \mathds{1}_{\{x_{t,n}=x,a_{t,n}=a\}} \xi_{t,n}(x,a) \leq 3 C \sqrt{|\mathcal{X}| |\mathcal{A}| T}.
    \end{equation}

    Therefore, to conclude the proof we need to analyze the term
    \[
    \sum_{t=1}^T \sum_{x,a} \big( \mu_n^{\pi_t, \nu_t}(x,a) -  \mathds{1}_{\{x_{t,n}=x,a_{t,n}=a\}} \big) \xi_{t,n}(x,a).
    \]
   Note that, for all \( (x, a) \), \( \xi_{t,n}(x, a) \) is \( \mathcal{F}_t \)-measurable by construction, and, by assumption, so are \( \pi_t \) and \( \nu_t \). Consequently, we have \( \mathbb{E}[\mu_n^{\pi_t, \nu_t}(x, a) \xi_{t,n}(x, a) \mid \mathcal{F}_t] = \mu_n^{\pi_t, \nu_t}(x, a) \xi_{t,n}(x, a) \). Furthermore, since the trajectories in each episode are observed independently from agents initialized according to the distribution \( \nu_t \), it follows that  
\[
\mathbb{E}[\mathds{1}_{\{x_{t,n} = x, a_{t,n} = a\}} \xi_{t,n}(x, a) \mid \mathcal{F}_t] = \xi_{t,n}(x, a) \mathbb{E}[\mathds{1}_{\{x_{t,n} = x, a_{t,n} = a\}} \mid \mathcal{F}_t] = \xi_{t,n}(x, a) \mu_n^{\pi_t, \nu_t}(x, a).
\]

Define the sequence $M_{0,n} = 0$, and for all $t \in [T]$
\[
M_{t,n} := \sum_{s=1}^t \sum_{x,a} \big( \mu_n^{\pi_s, \nu_s}(x,a) -  \mathds{1}_{\{x_{s,n}=x,a_{s,n} = a\}} \big) \xi_{s,n}(x,a).
\]
From the arguments above, $(M_{t,n})_{t \in [T]}$ defines a martingale difference sequence. Recall that $|\xi_{t,n}(x,a)| \leq c$. Hence, from Azuma-Hoeffding inequality we obtain that, for any $\delta \in (0,1)$, with probability at least $1-\delta$,
\begin{equation}\label{eq:mdp_term_2}
    M_{T,n} \leq c |\mathcal{X}| \sqrt{2 T \log\bigg(\frac{N}{\delta} \bigg)}.
\end{equation}

Joining the bounds on Eq.~\eqref{eq:mdp_term_1} and~\eqref{eq:mdp_term_2}, we obtain that, with probability at least $1-\delta$,
\[
 \sum_{t=1}^T \langle \mu^{\pi_t, \nu_t}_n, \xi_{t,n} \rangle \leq  3 C \sqrt{|\mathcal{X}| |\mathcal{A}| T} +  c |\mathcal{X}| \sqrt{2 T \log\bigg(\frac{N}{\delta} \bigg)},
\]
concluding the proof.
\end{proof}

\begin{corollary}\label{cor:bonus_analysis}
    Let $(b_t)_{t \in [T]}$ and $(\bar{b_t})_{t \in [T]}$ be the two sequences of bonus vectors defined in Eq.~\eqref{eq:bonus}. Hence, for any $\delta \in (0,1)$, with probability at least $1-\delta$, 
    \begin{equation*}
    \textstyle{
            \sum_{t=1}^T \langle b_t, \mu^{\pi_t, \rho_t}  \rangle = \tilde{O}(N |\mathcal{X}|^{3/2} \sqrt{|\mathcal{A}|T}), \quad \text{ and } \quad  \sum_{t=1}^T \langle \bar{b}_t, \mu^{\pi_t, \rho_t}  \rangle = \tilde{O}(\ell N^2 |\mathcal{X}|^{3/2} \sqrt{|\mathcal{A}|T})
            }
    \end{equation*}
    and with probability at least $1-2\delta$, 
        \begin{equation*}
    \textstyle{
 \sum_{t=1}^T \langle b_t, \hat{\mu}^{\pi_t, \rho_t}_t  \rangle = \tilde{O} \big(N^2 |\mathcal{X}|^{3/2} \sqrt{|\mathcal{A}| T}\big) , \quad \text{ and } \quad \sum_{t=1}^T \langle \bar{b}_t, \hat{\mu}^{\pi_t, \rho_t}_t  \rangle = \tilde{O} \big(\ell N^3 |\mathcal{X}|^{3/2} \sqrt{|\mathcal{A}| T}\big).
            }
    \end{equation*}
\end{corollary}
\begin{proof}
    From Prop.~\ref{prop:mdp_martingale} with $\nu_t = \rho_t$, the policy $\pi_t$ solution of Eq.~\eqref{iteration_md_solver}, the trajectories observed in MDPP-K (Alg.~\ref{alg:main1}) that are independent across episodes, and $\xi_t = b_t$, such that $C  = c = C_\delta$ we obtain that, for any $\delta \in (0,1)$, with probability at least $1-\delta$,
\begin{equation*}
    \begin{split}
        \sum_{t=1}^T \langle b_t, \mu^{\pi_t, \rho_t} \rangle &\leq 3 N C_\delta \sqrt{|\mathcal{X}| |\mathcal{A}| T} + N C_\delta |\mathcal{X}| \sqrt{2 T \log\bigg(\frac{N}{\delta} \bigg)} \\
        &\leq N (3 |\mathcal{X}|+ 2|\mathcal{X}|^{3/2}) \sqrt{2 |\mathcal{A}| T \log\bigg(\frac{|\mathcal{X} |\mathcal{A}| N T }{\delta} \bigg) \log \bigg(\frac{N}{\delta} \bigg)} .
    \end{split}
\end{equation*}

For the bonus sequence \((\bar{b}_t)_{t \in [T]}\), we apply Prop.~\ref{prop:mdp_martingale} in a similar manner, this time with \(C = c = \ell N C_\delta\), hence we obtain that
\begin{equation*}
    \begin{split}
        \sum_{t=1}^T \langle \bar{b}_t, \mu^{\pi_t, \rho_t} \rangle \leq  \ell N^2 (3 |\mathcal{X}| + 2|\mathcal{X}|^{3/2} )\sqrt{2 |\mathcal{A}| T \log\bigg(\frac{|\mathcal{X} |\mathcal{A}| N T }{\delta} \bigg) \log \bigg(\frac{N}{\delta} \bigg)}.
    \end{split}
\end{equation*}

As for the terms involving the sum of the product between a bonus vector and $\hat{\mu}^{\pi_t, \rho_t}_t$, we have that
\[
\sum_{t=1}^T \langle b_t, \hat{\mu}^{\pi_t, \rho_t}_t  \rangle = \underbrace{ \sum_{t=1}^T \langle b_t, \hat{\mu}^{\pi_t, \rho_t}_t - \mu^{\pi_t, \rho_t}  \rangle}_{(i)} + \underbrace{\sum_{t=1}^T \langle b_t, \mu^{\pi_t, \rho_t}  \rangle}_{(ii)},
\]
where we have just shown a high-probability upper bound for term $(ii)$. 

We now analyze term $(i)$: using Holder's inequality, and that from the definition of the bonus vector, $\|b_{t,n}\|_\infty \leq C_\delta$, we have that
\begin{equation*}
    \begin{split}
        (i) &\leq \sum_{t=1}^T \sum_{n=1}^N \|b_{t,n} \|_\infty \|\hat{\mu}_{t,n}^{\pi_t, \rho_t} - \mu_n^{\pi_t, \rho_t} \|_1 \\
        &\underbrace{\leq}_{\text{Lemma~\ref{lemma:bound_norm_mu_diff_mu0}}} C_\delta  \sum_{t=1}^T \sum_{n=1}^N \sum_{i=0}^{n-1} \sum_{x,a} \mu_i^{\pi_t, \rho_t}(x,a) \|p_{i+1}(\cdot|x,a) - \hat{p}^t_{i+1}(\cdot|x,a) \|_1.
    \end{split}
\end{equation*}
We then apply Prop.~\ref{prop:mdp_martingale} again with the trajectories observed in MDPP-K (Alg.~\ref{alg:main1}), policy $\pi_t$, $\nu_t = \rho_t$, but now with $\xi_{t,n}(x,a) := \|p_{n+1}(\cdot|x,a) - \hat{p}_{t,n+1}(\cdot|x,a) \|_1 \leq \frac{C_\delta}{\max{\{1,\sqrt{N_{t,n}(x,a)}\}}}$, where $\xi_{t,n}(x,a) \leq 2$ by definition. Hence, we can take $C = C_\delta$ and $c = 2$ to obtain that, for any $\delta \in (0,1)$, with probability at least $1-\delta$,
\begin{equation*}
    \begin{split}
        (i) &\leq C_\delta N^2 \bigg[3 C_\delta \sqrt{|\mathcal{X}||\mathcal{A}|T} + 2|\mathcal{X}| \sqrt{2 T \log\bigg(\frac{N}{\delta}\bigg)} \bigg] \\
        &\leq 10 N^2 |\mathcal{X}|^{3/2} \sqrt{|\mathcal{A}| T} \log\bigg(\frac{|\mathcal{X} |\mathcal{A}| N T }{\delta} \bigg) \sqrt{\log \bigg(\frac{N}{\delta} \bigg)}.
    \end{split}
\end{equation*}
Joining both terms, we get that for any $\delta \in (0,1)$, with probability at least $1-2\delta$,
\begin{equation*}
    \begin{split}
        \sum_{t=1}^T \langle b_t, \hat{\mu}^{\pi_t, \rho_t}_t  \rangle &\leq 15  N^2  |\mathcal{X}|^{3/2} \sqrt{ |\mathcal{A}| T } \log\bigg(\frac{|\mathcal{X} |\mathcal{A}| N T }{\delta} \bigg) \sqrt{\log \bigg(\frac{N}{\delta} \bigg)} \\
        &= O\bigg(N^2 |\mathcal{X}|^{3/2}  \sqrt{ |\mathcal{A}| T } \log\bigg(\frac{|\mathcal{X}| |\mathcal{A}| N T}{\delta} \bigg) \bigg),
    \end{split}
\end{equation*}
where we use the definition of $C_\delta$ in Eq.~\eqref{eq:c_delta}.

Similarly, for the bonus sequence \(\bar{b}_t\), we have \(\|\bar{b}_{t,n} \|_\infty \leq \ell N C_\delta\). Therefore, by applying a similar reasoning with adjusted constants, we obtain that
\begin{equation*}
    \begin{split}
        \sum_{t=1}^T \langle \bar{b}_t, \hat{\mu}^{\pi_t, \rho_t}_t  \rangle &\leq  \ell C_\delta N^3  \bigg[3 C_\delta \sqrt{|\mathcal{X}||\mathcal{A}|T} + 2|\mathcal{X}| \sqrt{2 T \log\bigg(\frac{N}{\delta}\bigg)} \bigg]+ \sum_{t=1}^T \langle \bar{b}_t, \mu^{\pi_t, \rho_t} \rangle \\
        &\leq  10 \ell N^3 |\mathcal{X}|^{3/2} \sqrt{|\mathcal{A}| T} \log\bigg(\frac{|\mathcal{X} |\mathcal{A}| N T }{\delta} \bigg) \sqrt{\log \bigg(\frac{N}{\delta} \bigg)} + \sum_{t=1}^T \langle \bar{b}_t, \mu^{\pi_t, \rho_t} \rangle \\
        &\leq 15 \ell N^3 |\mathcal{X}|^{3/2} \sqrt{|\mathcal{A}| T} \log\bigg(\frac{|\mathcal{X} |\mathcal{A}| N T }{\delta} \bigg) \sqrt{\log \bigg(\frac{N}{\delta} \bigg)} \\
        &= O \bigg( \ell N^3 |\mathcal{X}|^{3/2} \sqrt{|\mathcal{A}| T} \log\bigg(\frac{|\mathcal{X}| |\mathcal{A}| N T}{\delta} \bigg)   \bigg).
\end{split}
\end{equation*}

\end{proof}


\begin{lemma}\label{lemma:difference_consecutive_p}[Lemma A.3 from \cite{moreno_icml}]
    For all $n \in [N]$, $(x,a, x') \in \mathcal{X} \times \mathcal{A} \times \mathcal{X}$, and $t \in [T]$, let $\hat{p}_{t,n+1}(x'|x,a)$ be defined as in Eq.~\eqref{eq:proba_est}. Hence,
    \[
    \|\hat{p}_{t+1,n+1}(\cdot|x,a) - \hat{p}_{t,n+1}(\cdot|x,a) \|_1 \leq \frac{\mathds{1}_{\{x_{t,n} = x, a_{t,n} = a\}}}{\max\{1, N_{t+1,n}(x,a)\}}.
    \]
\end{lemma}

\begin{lemma}\label{lemma:kernel_diff_two_episodes}[Lemma A.4 from \cite{moreno_icml}]
    For $(n,x,a) \in [N] \times \mathcal{X} \times \mathcal{A}$, 
    let $(q_t)_{t \in [T]}$ be a sequence of probability transition kernels with $q_t := (q_{t,n})_{n \in [N]}$ such that
    \[
    \|q_{t+1,n}(\cdot|x,a) - q_{t,n}(\cdot|x,a) \|_1 \leq \frac{ c \mathds{1}_{\{x_{t,n-1}=x, a_{t,n-1}=a\}}}{\max{\{1, N_{t+1,n-1}(x,a)\}}}
    \]
    for some constant $c>0$. Then,
    \[
    \sum_{t=1}^T \|q_{t+1,n}(\cdot|x,a) - q_{t,n}(\cdot|x,a) \|_1 \leq e c \log(T) \,.
    \]
\end{lemma}

\section{Feasibility}\label{app:feasibility}

\subsection{Proof of Lemma~\ref{lemma:feasibility}}
\begin{proof}
    We show that for every episode $t$, there is a state-action distribution sequence $\mu \in \mathcal{M}_t$ that satisfies the constraint
    \begin{equation}\label{eq:aditional_constraints}
        \begin{split}
              \| \mu_N - \rho \|_1 \leq \langle \mu, b_{t} \rangle + \alpha \|\rho_t - \rho\|_1.
        \end{split}
    \end{equation}

    Recall that we denote by $\rho_t$ the initial state-action distribution at episode $t$. From the equivalence between solving the problem over the set of state-action distributions sequences satisfying the Bellman flow and the set of policies, we have that $\mu_N = \rho_{t} \widehat{P}_\pi^{t}$ for some policy $\pi$. 
    
    From Ass.~\ref{ass:feasibility}, there exists a periodic policy $\pi$ in the true MDP with stationary distribution $\rho$, meaning that $\rho P_\pi = \rho$. Thus, this periodic policy $\pi$ satisfies that
    \begin{equation*}
            \begin{split}
                \|\rho_{t} \widehat{P}^t_\pi - \rho \|_1 &= \|\rho_{t} \widehat{P}^t_\pi - \rho_t P_\pi + \rho_t P_\pi - \rho \|_1 \\
                &\leq  \|\rho_{t} \widehat{P}^t_\pi - \rho_t P_\pi \|_1  + \|  \rho_t P_\pi - \rho \|_1 \\
                &\leq \langle \hat{\mu}^{\pi, \rho_t}_t, b_{t} \rangle + \alpha \|\rho_t - \rho \|_1 \\
            \end{split}
    \end{equation*}
   where the last inequality follows from Lemma~\ref{lemma:mu_bonus} stating that for any $\delta \in (0,1)$, with probability $1-\delta$, $\|\rho_{t} (\widehat{P}^t_\pi - P_\pi) \|_1 \leq \langle \hat{\mu}^{\pi, \rho_t}_t, b_{t} \rangle$, and from Ass.~\ref{ass:contraction}. Hence, for $\bar{\alpha} = \alpha$, we showed the existence of a state-action distribution sequence in $\mathcal{M}_t$ satisfying the constraints in Eq.~\eqref{eq:aditional_constraints}.
\end{proof}

\subsection{Proof of Lemma~\ref{lemma:almost_equal_dist}}

\begin{proof}
 Recall that we define $\mu_t := \hat{\mu}^{ \pi_t, \rho_t}_t$ as the solution to Eq.~\eqref{iteration_md_solver} at episode $t$. 
 
 Since $\mu_t$ is the solution to the optimization problem in Eq.~\eqref{iteration_md_solver}, it satisfies the following constraint:
 \begin{equation*}
 \begin{split}
     \|\mu_{t,N} - \rho \|_1 \leq \langle \mu_t, b_{t} \rangle + \alpha \|\rho_t - \rho \|_1.
 \end{split}
 \end{equation*}

Using that $\mu_{t,N} = \rho_t \widehat{P}^t_{\pi_t}$, and that the final state-action distribution we reach when playing policy $\pi_t$ starting from the initial distribution $\rho_t$ in the true MDP is $\rho_{t+1} = \mu_N^{\pi_t, \rho_t} = \rho_t P_{\pi_t}$, we have that
\begin{equation*}
    \begin{split}
        \| \mu_N^{\pi_t, \rho_t} - \rho \|_1 &= \|  \rho_t P_{\pi_t} - \rho \|_1  \\
        &\leq \| \rho_t P_{\pi_t} - \rho_t \widehat{P}_{\pi_t}^t \|_1 + \| \rho_t \widehat{P}_{\pi_t}^t  - \rho \|_1  \\
        &\leq 2 \langle \hat{\mu}^{\pi_t, \rho_t}_t, b_{t} \rangle +  \alpha \|\rho_t - \rho \|_1,
    \end{split}
\end{equation*}
where the last inequality follows from Lemma~\ref{lemma:mu_bonus} stating that for any $\delta \in (0,1)$, with probability $1-\delta$, $\| \rho_t P_{\pi_t} - \rho_t \widehat{P}_{\pi_t}^t \|_1  \leq \langle \hat{\mu}^{ \pi, \rho_t}_t, b_{t} \rangle$, and that the policy $\pi_t$ induces a state-action distribution sequence in $\mathcal{M}_t$ satisfying the inequality constraint.

 Recall that $\rho_t := \rho_{t-1} P_{\pi_{t-1}} = \mu_N^{\pi_{t-1}, \rho_{t-1}}$. Substituting the inequalities iteratively for all $t$, until the first episode where $\rho_1 = \rho$, we derive that:
\begin{equation*}
    \begin{split}
        \| \mu_N^{\pi_t, \rho_t} - \rho \|_1 &\leq  2 \langle \mu_t, b_{t} \rangle +  \alpha \|\rho_t - \rho \|_1 \\
        &=  2 \langle \mu_t, b_t \rangle  +  \alpha \|\mu_N^{\pi_{t-1}, \rho_{t-1}} -\rho \|_1 \\
        &\leq 2 \langle \mu_t, b_t\rangle +   \alpha [ 2 \langle \mu_{t-1}, b_{t-1} \rangle + \alpha \| \rho_{t-1} - \rho \|]  \\
        &\leq  2 \sum_{s=0}^{t-1} \alpha^{s}  \langle \mu_{t-s}, b_{t-s}\rangle +   \alpha^t \underbrace{\|\rho_1 - \rho\|_1}_{=0}.
    \end{split}
\end{equation*}

Therefore, summing over $t \in [T]$, with probability $1-T \delta$,
\begin{equation*}
    \begin{split}
        \sum_{t=1}^T \| \mu_N^{\pi_t, \rho_t} - \rho \|_1 &\leq  2 \sum_{t=1}^T \sum_{s=0}^{t-1} \alpha^s \langle \mu_{t-s}, b_{t-s}\rangle. \\
    \end{split}
\end{equation*}
By rearranging the sum and using Corollary~\ref{cor:bonus_analysis} we have that with probability $1-(2+ T) \delta$, 
\begin{equation*}
    \begin{split}
         \sum_{t=0}^T\alpha^t \sum_{s=1}^{T-t} \langle \mu_s, b_s \rangle &\leq \sum_{t=0}^T  \alpha^t 15  N^2  |\mathcal{X}|^{3/2} \sqrt{ |\mathcal{A}| T } \log\bigg(\frac{|\mathcal{X} |\mathcal{A}| N T }{\delta} \bigg) \sqrt{\log \bigg(\frac{N}{\delta} \bigg)}  \\
         &\leq \frac{ 15 N^2}{1- \alpha } |\mathcal{X}|^{3/2} \sqrt{ |\mathcal{A}| T} \log\bigg(\frac{|\mathcal{X} |\mathcal{A}| N T }{\delta} \bigg) \sqrt{\log \bigg(\frac{N}{\delta} \bigg)}, \\
    \end{split}
\end{equation*}
where we use that $\sum_{t=0}^T  \alpha^t \leq \frac{1}{1 -  \alpha }$ as $0 \leq  \alpha  < 1$ by definition. 

Therefore, we obtain that for any $\delta \in  (0,1)$, with probability at least $1-(2+T)\delta$,
\begin{equation*}
    \begin{split}
    \sum_{t=1}^T \| \rho_{t+1} - \rho \|_1 &\leq \frac{ 30 N^2}{1- \alpha } |\mathcal{X}|^{3/2} \sqrt{ |\mathcal{A}| T} \log\bigg(\frac{|\mathcal{X} |\mathcal{A}| N T }{\delta} \bigg) \sqrt{\log \bigg(\frac{N}{\delta} \bigg)}, \\
    &=\tilde{O}\bigg(\frac{N^2}{1-\alpha} |\mathcal{X}|^{3/2} \sqrt{|\mathcal{A}| T} \bigg).   
    \end{split}
\end{equation*}
  
\end{proof}

\section{Upper bound on Mirror Descent term}\label{app:md_proof}

Recall that for any initial state-action distribution $\rho \in \Delta_{\mathcal{X} \times \mathcal{A}}$, and for any probability transition kernel $p$, we define in Eq.~\eqref{eq:bellman_flow} the set of state-action distribution sequences satisfying the MDP dynamics:
\[
\mathcal{M}_\rho^p := \Big\{\mu := (\mu_n)_{n \in [N]} \; \Big| \; \mu_0 = \rho \text{ and } \sum_{a} \mu_{n+1}(x,a) := \sum_{x',a'} \mu_n(x',a') p_{n+1}(x|x', a') \;   \Big\}.
\]
Let $\mathcal{M}_\rho^{p,*}$ be the subset of $\mathcal{M}^p_\rho$ where the corresponding policies (in the sense of Eq.~\eqref{mu_induced_pi}) satisfy $\pi_n(a|x) \neq 0$. We denote by $\psi : \mathbb{R}^{N \times |\mathcal{X}| \times |\mathcal{A}|} \rightarrow \mathbb{R}$ the function inducing the Bregman divergence $D_\psi$ in the sense that for all $\mu, \mu' \in \mathbb{R}^{N \times |\mathcal{X}| \times |\mathcal{A}|} $,
\begin{equation}\label{def:bregman_induction}
    D_\psi(\mu, \mu') := \psi(\mu) - \psi(\mu') - \langle \nabla \psi(\mu'), \mu - \mu' \rangle.
\end{equation}

\begin{assumption}\label{ass:gamma_inequality}
    For any $\mu$ and $\mu'$,
\begin{equation}\label{gamma_inequality}
    D_\psi(\mu, \mu') \geq \frac{1}{2} \|\mu - \mu'\|_{\infty,1}^2.
\end{equation}
\end{assumption}

From now on, we assume $D_\psi$ is one of the two following commonly used Bregman divergences satisfying Ass.~\ref{ass:gamma_inequality}:
\begin{itemize}
    \item Using the convention that \( 0 \log(0) = 0 \) and assuming \( \mu'_n(x,a) > 0 \) for all \( (n,x,a) \), let $D_\psi$ be the Kullback-Leibler divergence:
\begin{equation}\label{gamma:kl}
D_\psi(\mu, \mu') = \text{KL}(\mu, \mu') := \sum_{n=0}^N \sum_{x,a} \mu_n(x,a) \log\bigg(\frac{\mu_n(x,a)}{\mu'_n(x,a)} \bigg).
\end{equation}
By Pinsker's inequality, the result in Eq.~\eqref{gamma_inequality} holds. Note that, in this case, Pinsker's inequality remains valid even if \(\mu\) and \(\mu'\) are induced by different probability transitions.

\item For any two probability transition kernels $p, q$, for any two initial state-action distributions $\rho, \rho'$, for all $\mu \in \mathcal{M}_\rho^p$, and $\mu'\in \mathcal{M}_{\rho'}^{q,*}$ with respective policies $\pi, \pi'$, let
\begin{equation}\label{gamma:non_standard}
D_\psi(\mu, \mu') = \Gamma(\mu, \mu'):= \sum_{n=1}^N \sum_{x,a} \mu_n(x,a) \log\bigg(\frac{\pi_n(a|x)}{\pi'_n(a|x)} \bigg) + \rho(x,a) \log \bigg(\frac{\rho(x,a)}{\rho'(x,a)} \bigg).
\end{equation}
If $p = q$, Lemma B.1 from \cite{pmlr-v238-moreno24a} shows that $D_\psi$ satisfies the inequality from Eq.~\eqref{gamma_inequality}.
\end{itemize}

Lemma~\ref{lemma:aux_md} presents an auxiliary result for Online Mirror Descent, where, in each episode $t$, the iteration is performed over a space of measures \(\mathcal{M}_{\rho_t}^{q_t}\) determined by a probability transition kernel \(q_t\) and an initial state-action distribution \(\rho_t\), both of which may vary across episodes. Additionally, each iteration is performed over a set of convex constraints $\mathcal{W}_t$ that may also vary across episodes. We compare the sequence of state-action distributions \((\mu_t)_{t \in [T]}\), generated by solving each OMD iteration, with any sequence \(\nu_t := \nu^{q_t, \pi, \rho_t}\) that satisfies the constraints in \(\mathcal{W}_t\) and is induced by a fixed policy \(\pi\).

We then apply Lemma~\ref{lemma:aux_md} for the case when \( q_t := \hat{p}_t \) as defined in Eq.~\eqref{eq:proba_est}, \( \rho_t := \rho_{t-1} P_{\pi_{t-1}} = \mu_N^{\pi_{t-1}, p, \rho_{t-1}} \), where the initial distribution in the first episode is set as \( \rho_1 := \rho \), and $\mathcal{W}_t := \{ \mu := (\mu_n)_{n \in [N]} \; | \; \| \mu_n - \rho \|_1 \leq \langle \mu, b_t \rangle + \bar{\alpha} \|\rho_t - \rho\|_1 \}$. Joining this result with the feasibility results from App.~\ref{app:feasibility}, and some auxiliary Lemmas in App.~\ref{app:auxiliary}, we show in Prop.~\ref{prop:bound_md_term} a final upper bound on the term \( R_T^{\text{MD}} \) from the periodic regret decomposition presented in Sec.~\ref{sec:regret_analysis}.

\begin{lemma}\label{lemma:aux_md}
    Consider a sequence of vectors $(z_t)_{t \in [T]}$ where $z_t \in \mathbb{R}^{N \times |\mathcal{X}| \times |\mathcal{A}|}$ such that $\max_{t \in [T]} \|z_t \|_{1,\infty} \leq \zeta$. Let $(\mathcal{W}_t)_{t \in [T]}$ define a sequence of convex constraint sets for the state-action distributions sequences. Let $(\rho_t)_{t \in [T]}$ define a sequence of state-action distributions, $(q_t)_{t \in [T]}$ be a sequence of probability transitions, and $\mathcal{M}_t := \mathcal{M}_{\rho_t}^{q_t}$. Assume that $\mathcal{M}_t \cap \mathcal{W}_t$ is non-empty for all $t \in [T]$. Let
    \begin{equation*}
   \mu_{t+1} \in \argmin_{\mu \in \mathcal{M}_{t+1} \cap \mathcal{W}_{t+1}} \quad \Big\{ \eta \langle z_t, \mu \rangle +  D_\psi(\mu, \mu_{t}) \Big\}
\end{equation*}
for \(D_\psi\) as given in either Eq.~\eqref{gamma:kl} or Eq.~\eqref{gamma:non_standard}, $\mu_1$ initialized such that $\nabla \psi(\mu_1) = 0$, and $\eta >0$. For any sequence of distributions $(\nu_t)_{t \in [T]}$ with $\nu_t := \nu^{q_t, \pi, \rho_t}$ for a common policy $\pi$,
\begin{equation*}
    \begin{split}
        \eta \sum_{t=1}^T \langle z_t, \mu_t - \nu_t \rangle &\leq \eta^2 \zeta^2 T + N V_T \mathds{1}_{\{ D_\psi = \Gamma \}} - \psi(\mu_1) \\
        &+ \sum_{t=1}^T \langle \nabla \psi(\mu_t), \nu_t - \nu_{t+1} \rangle + \eta \sum_{t=1}^T \langle z_t, \nu_{t+1} - \nu_t \rangle,
    \end{split}
\end{equation*}
where $V_T \geq 1 + \max_{(n,x,a)} \sum_{t=1}^{T-1} \|q_{t,n}(\cdot|x,a) - q_{t+1,n}(\cdot|x,a) \|_1$.
\end{lemma}
\begin{proof}
    Since this is a convex problem and the constraint set is non-empty, using the optimality conditions and the three points Bregman inequality, for all $\nu_{t+1} \in \mathcal{M}_{t+1} \cap \mathcal{W}_{t+1}$,
    \[
    \eta \langle z_t, \mu_{t+1} - \nu_{t+1} \rangle \leq D_\psi(\nu_{t+1}, \mu_t) - D_\psi(\nu_{t+1}, \mu_{t+1}) - D_\psi(\mu_{t+1}, \mu_t).
    \]
    Re-arranging the terms, 
    \begin{equation}\label{eq:decomp_md}
        \begin{split}
            \eta \sum_{t=1}^T \langle z_t, \mu_t - \nu_t \rangle &\leq  \underbrace{\eta \sum_{t=1}^T \langle z_t, \mu_t - \mu_{t+1} \rangle - D_\psi(\mu_{t+1}, \mu_t)}_{ (1)} + \underbrace{\sum_{t=1}^T D_\psi(\nu_{t+1}, \mu_t) - D_\psi(\nu_{t+1}, \mu_{t+1} ))}_{(2)} \\
            & + \eta \sum_{t=1}^T \langle z_t, \nu_{t+1} - \nu_t \rangle.
        \end{split}
    \end{equation}

We analyse each term separetly. 
\paragraph{Term $1$ analysis:} Applying Young's inequality for some $\sigma > 0$ we have that for each $t \in [T]$,
\[
\eta\langle z_t, \mu_t - \mu_{t+1} \rangle - D_\psi(\mu_{t+1}, \mu_t) \leq \frac{\eta^2 \|z_t\|^2_{1, \infty}}{2 \sigma} + \frac{\sigma}{2} \|\mu_t - \mu_{t+1} \|_{\infty, 1}^2 - D_\psi(\mu_{t+1}, \mu_t).
\]
\begin{itemize}
    \item $D_\psi = \text{KL}$: If $D_\psi$ is the Kullback-Leibler divergence as in Eq.~\eqref{gamma:kl}, then for $\sigma =1$, using the inequality from Eq.~\eqref{gamma_inequality},
    \[\frac{\sigma}{2} \|\mu_t - \mu_{t+1} \|_{\infty, 1}^2 - \text{KL}(\mu_{t+1}, \mu_t) \leq 0.\]
    Hence, $(1) \leq \frac{\eta^2 \zeta^2 T}{2}.$

\item $D_\psi = \Gamma:$
If \(D_\psi\) is the divergence defined in Eq.~\eqref{gamma:non_standard}, then it does not satisfy Inequality~\eqref{gamma_inequality}, as the probability transition kernel \(q_t\), which induces \(\mu_t\), may differ from \(q_{t+1}\) inducing $\mu_{t+1}$. However, by the definition of \(\Gamma\), the probability transition kernel inducing the term in the second entry is irrelevant; only the policy and initial distribution matter. Therefore, we have
\[
\Gamma(\mu_{t+1}, \mu_t) = \Gamma(\mu_{t+1}, \mu^{\pi_t, \rho_t}_{t+1}) \geq \frac{1}{2} \|\mu_{t+1} - \mu^{ \pi_t, \rho_t}_{t+1} \|_{\infty,1}^2.
\]

Thus, for $\sigma =1/2$ we have that
\begin{equation*}
    \begin{split}
      & \sum_{t=1}^T \frac{1}{4} \| \mu_t - \mu_{t+1} \|^2_{\infty,1} - \Gamma(\mu_{t+1}, \mu_t) \leq \sum_{t=1}^T \frac{1}{4} \| \mu_t - \mu_{t+1} \|^2_{\infty,1} - \frac{1}{2} \|\mu_{t+1} - \mu^{ \pi_t, \rho_t}_{t+1} \|_{\infty,1}^2  \\
        &\quad \quad \leq \sum_{t=1}^T\frac{1}{2} \|\mu_t - \mu^{ \pi_t, \rho_t}_{t+1}\|_{\infty,1}^2 \\
        &\quad \quad \leq \sum_{t=1}^T \sup_{n \in \{0, \ldots, N\} } \sum_{i=0}^{n-1} \sum_{x,a} \mu_{t,i}(x,a) \|q_{t,i+1}(\cdot|x,a) - q_{t+1,i+1}(\cdot|x,a) \|_1 \\
        &\quad \quad \leq  N V_T,
    \end{split}
\end{equation*}
where the second to last inequality comes from Lemma~\ref{lemma:bound_norm_mu_diff_mu0}. Hence,
    \[
    (1) \leq \eta^2 \zeta^2 T +  N V_T.
    \]
\end{itemize}

\paragraph{Term $2$ analysis:} In the analysis of classic online mirror descent \cite{shalev_oco}, the sum in term \((2)\) is telescopic. However, because in our case the constraint set changes at each episode due to varying probability transitions and changing initial state-action distributions, this term requires further analysis. To address this, we sum and subtract \(D_\psi(\nu_t, \mu_t)\):
\begin{equation*}
    \begin{split}
        (2) &= \sum_{t=1}^T  D_\psi(\nu_{t+1}, \mu_t) - D_\psi(\nu_{t+1}, \mu_{t+1}) \\
        &= \underbrace{\sum_{t=1}^T D_\psi(\nu_{t+1}, \mu_t) -D_\psi(\nu_t, \mu_t)}_{(i)} + \underbrace{\sum_{t=1}^T D_\psi(\nu_t, \mu_t) - D_\psi(\nu_{t+1}, \mu_{t+1})}_{(ii)}.
    \end{split}
\end{equation*}

For $(i)$, using that $\psi$ is the function inducing the Bregman divergence $D_\psi$ in the sense of Eq.~\eqref{def:bregman_induction}, we have that
\begin{equation*}
    \begin{split}
        (i) &= \sum_{t=1}^T \psi(\nu_{t+1}) - \psi(\nu_t) - \langle \nabla \psi(\mu_t), \nu_{t+1} - \nu_t \rangle \\
        &= \psi(\nu_{T+1}) - \psi(\nu_1) + \sum_{t=1}^T \langle \nabla \psi(\mu_t), \nu_t - \nu_{t+1} \rangle.
    \end{split}
\end{equation*}

Term $(ii)$ is a telescopic sum, therefore, as a Bregman divergence is always positive, 
\[
(ii) = D_\psi(\nu_1, \mu_1) - D_\psi(\nu_{T+1}, \mu_{t+1}) \leq D_\psi(\nu_1, \mu_1).
\]

Joining terms $(i)$ and $(ii)$, using that $\mu_1$ is initialized such that $\nabla \psi (\mu_1) = 0$, and that $\psi(\mu) \leq 0$ for all state-action distribution sequence $\mu$ for both Bregman divergences considered, we get that
\[
(2) \leq \psi(\nu_{T+1}) - \psi(\nu_1) + \sum_{t=1}^T \langle \nabla \psi(\mu_t), \nu_t - \nu_{t+1} \rangle + D_\psi(\nu_1, \mu_1) \leq - \psi(\mu_1) + \sum_{t=1}^T \langle \nabla \psi(\mu_t), \nu_t - \nu_{t+1} \rangle.
\]

\paragraph{Joining all terms:} replacing the upper bounds for each term in the inequality of Eq.~\eqref{eq:decomp_md}, we obtain that, 
\begin{equation*}
    \begin{split}
        \eta \sum_{t=1}^T \langle z_t, \mu_t - \nu_t \rangle &\leq \eta^2 \zeta^2 T + N V_T \mathds{1}_{\{ D_\psi = \Gamma \}} - \psi(\mu_1) \\
        &+ \sum_{t=1}^T \langle \nabla \psi(\mu_t), \nu_t - \nu_{t+1} \rangle + \eta \sum_{t=1}^T \langle z_t, \nu_{t+1} - \nu_t \rangle.
    \end{split}
\end{equation*}
\end{proof}

\subsection{Proof of Prop.~\ref{prop:bound_md_term}}
We prove Prop.~\ref{prop:bound_md_term} which bounds the $R_T^{\text{MD}}$ term of the periodic regret decomposition.
\begin{proof}
    Let $\nu_t := \hat{\mu}^{ \pi, \rho_t}_t$ for all $t \in [T]$, where $\pi$ is any periodic policy in the true MDP for the initial distribution $\rho$, \emph{i.e.}, $\rho P_\pi = \rho$. From the feasibility result on Lemma~\ref{lemma:feasibility}, for all $t \in [T]$, we have that with high probability,
    \begin{equation}\label{eq:set_convex_constraints}
        \nu_t \in \mathcal{W}_t := \{ \mu := (\mu_n)_{n \in [N]} \; | \; \| \mu_n - \rho \|_1 \leq \langle \mu, b_t \rangle + \bar{\alpha} \|\rho_t - \rho\|_1 \},
    \end{equation}
    and $\mathcal{W}_t$ is a convex set on the sequences of state-action distributions. Hence, the set $\mathcal{M}_t \cap \mathcal{W}_t$ is convex and non-empty. 

    Since $f_{t,n}$ is $\ell$-Lipschitz with respect to the $\|\cdot\|_1$ norm, we have $\|\ell_t\|_{1, \infty} \leq N\ell$. From the definition of the bonus vector $\bar{b}_t$ in Eq.~\eqref{eq:bonus}, it follows that $\|\bar{b}_t\|_{1, \infty} \leq \ell C_\delta N^2$, where $C_\delta$ is given in Eq.~\eqref{eq:c_delta}. By applying Lemma~\ref{lemma:aux_md} with $z_t := \ell_t - \bar{b}_t$ such that $\zeta \leq 2 \ell N^2 C_\delta$, $q_t := \hat{p}_t$ as defined in Eq.~\eqref{eq:proba_est}, $\rho_t := \rho_{t-1} P_{\pi_{t-1}} = \mu_N^{\pi_{t-1}, \rho_{t-1}}$ for all $1 < t \leq T$, $\rho_1 = \rho$, and $\mathcal{W}_t$ as defined in Eq.~\eqref{eq:set_convex_constraints}, and noting that $\mu_t$ is the solution to the optimization problem in Eq.~\eqref{iteration_md_solver}, we obtain:
    \begin{equation}\label{eq:inequality_md}
    \begin{split}
          R_T^{\text{MD}} := \sum_{t=1}^T \langle \ell_t - \bar{b}_t, \mu_t - \hat{\mu}^{\pi, \rho_t}_t \rangle &\leq \eta (2 \ell N^2 C_\delta)^2 T + \frac{2}{\eta}  N e \log(T) \mathds{1}_{\{ D_\psi = \Gamma \}}  - \frac{\psi(\mu_1)}{\eta} \\
          &+ \frac{1}{\eta} \underbrace{\sum_{t=1}^T \langle \nabla \psi(\mu_t), \nu_t - \nu_{t+1} \rangle}_{(i)} +  \underbrace{\sum_{t=1}^T \langle \ell^t - \bar{b}_t, \nu_{t+1} - \nu_t \rangle}_{(ii)},
    \end{split}
    \end{equation}
    where we also use that for all $(n,x,a)$, $V_T := 1 + \sum_{t=1}^T \|\hat{p}_{t,n}(\cdot|x,a) - \hat{p}_{t+1,n}(\cdot|x,a) \|_1 \leq e \log(T)$ from Lemmas~\ref{lemma:difference_consecutive_p} and~\ref{lemma:kernel_diff_two_episodes}. We now analyze terms $(i)$ and $(ii)$ considering that $\nu_t = \hat{\mu}^{\pi, \rho_t}_t$.

\paragraph{Term $(i)$:} We assume that $\|\nabla \psi(\mu_{t})\|_{1,\infty} \leq \Psi$. In practice, to ensure this upper bound, one can define a smoothed version of the policy by mixing it with a uniform policy, as proposed in \cite{pmlr-v238-moreno24a}. This modification results in only minor changes to the regret, affecting the final regret bound by an additional logarithmic term. Therefore, for clarity of presentation, we assume directly that the gradient of $\psi$ is bounded.

From Holder's inequality, we have that, for any $\delta \in (0,1)$, with probability at least $1-(2+T)\delta$, 
\begin{equation}\label{eq:bound_i_md}
\begin{split}
   & \sum_{t=1}^T \langle \nabla \psi(\mu_{t}), \nu_{t} - \nu_{t+1} \rangle \leq \sum_{t=1}^T \|\nabla \psi(\mu_{t})\|_{1,\infty} \|\nu_{t} - \nu_{t+1} \|_{\infty,1} \\
    &\quad \quad \leq  \Psi \sum_{t=1}^T \sup_{n \in \{0,\ldots,N\}}  \|\nu_{t,n} - \nu_{t+1,n} \|_{\infty,1} \\
    &\quad \quad =  \Psi \sum_{t=1}^T \sup_{n \in \{0,\ldots,N\}} \|\hat{\mu}^{\pi, \rho_t}_{t,n} - \hat{\mu}^{\pi, \rho_{t+1}}_{t+1,n} \|_1 \\
    &\quad \quad \underbrace{\leq}_{\text{Lemma~\ref{lemma:bound_norm_mu_diff_mu0}}} \Psi \sum_{t=1}^T \bigg[ \sup_{\{0,\ldots,N\}} \sum_{i=1}^n \sum_{x,a} \hat{\mu}_{t,i-1}^{\pi, \rho_t}(x,a) \|\hat{p}_{t,i}(\cdot|x,a) - \hat{p}_{t+1,i}(\cdot|x,a) \|_1 + \|\rho_t - \rho_{t+1} \|_1 \bigg] \\
    &\quad \quad \underbrace{\leq}_{\text{Lemmas~\ref{lemma:difference_consecutive_p} and~\ref{lemma:kernel_diff_two_episodes}}} \Psi N e \log(T) + \Psi \sum_{t=1}^T \|\rho_t - \rho_{t+1} \|_1 \\
    &\quad \quad \underbrace{\leq}_{\text{Lemma~\ref{lemma:almost_equal_dist}}} \Psi \bigg[ N e \log(T) +  \frac{ 60 N^2}{1- \alpha } |\mathcal{X}|^{3/2} \sqrt{ |\mathcal{A}| T} \log\bigg(\frac{|\mathcal{X} |\mathcal{A}| N T }{\delta} \bigg) \sqrt{\log \bigg(\frac{N}{\delta} \bigg)}
    \bigg].
\end{split}
\end{equation}

\paragraph{Term $(ii)$:} The analysis of the second term is similar to the first, but it depends on $|\ell_t - \bar{b}_t|_{1, \infty}$ instead of $\Psi$. Thus,
\[
\sum_{t=1}^T \langle \ell_t - \bar{b}_t, \nu_{t+1} - \nu_t \rangle \leq 2 \ell N^2 C_\delta \bigg[ N e \log(T) + \frac{ 60 N^2}{1- \alpha } |\mathcal{X}|^{3/2} \sqrt{ |\mathcal{A}| T} \log\bigg(\frac{|\mathcal{X} |\mathcal{A}| N T }{\delta} \bigg) \sqrt{\log \bigg(\frac{N}{\delta} \bigg)} \bigg].
\]

\paragraph{Final step on $R_T^{\text{MD}}$ analysis:} let 
\[
\beta = N e \log(T) \mathds{1}_{\{ D_\psi = \Gamma \}}  - \psi(\mu_1) + \Psi \bigg[ N e \log(T) +  \frac{ 60 N^2}{1- \alpha } |\mathcal{X}|^{3/2} \sqrt{ |\mathcal{A}| T} \log\bigg(\frac{|\mathcal{X} |\mathcal{A}| N T }{\delta} \bigg) \sqrt{\log \bigg(\frac{N}{\delta} \bigg)}
    \bigg].
\]
Joining the bounds of terms $(i)$ and $(ii)$ in Eq.~\eqref{eq:inequality_md}, we obtain that with probability at least $1 - 2(2 + T)\delta$,
\begin{equation*}
\begin{split}
    R_T^{\text{MD}} &\leq \eta (2 \ell N^2 C_\delta)^2 T + \frac{1}{\eta} \beta + 2 \ell N^2 C_\delta \bigg[ N e \log(T) + \frac{ 60 N^2}{1- \alpha } |\mathcal{X}|^{3/2} \sqrt{ |\mathcal{A}| T} \log\bigg(\frac{|\mathcal{X} |\mathcal{A}| N T }{\delta} \bigg) \sqrt{\log \bigg(\frac{N}{\delta} \bigg)} \bigg].
\end{split}
\end{equation*}
For $\eta = \sqrt{\frac{\beta}{(2 \ell N^2 C_\delta)^2 T}}$, we then obtain that with high probability
\[
R_T^{\text{MD}} \leq \tilde{O} \bigg( \ell N^3 |\mathcal{X}|^{5/4} |\mathcal{A}|^{1/4} \sqrt{\frac{\Psi}{1 - \alpha}} T^{3/4} \bigg).
\]

\end{proof}

\section{Main result known $\rho_t$}\label{app:main_result}
Before presenting the main periodic regret bound for our algorithm with known \( \rho_t \), we analyze the terms \( R_T^{\text{MDP}} \) and \( R_T^{\text{bonus}} \) from the regret decomposition in Sec.~\ref{sec:algorithm}. These results build upon and generalize existing results in the literature (see, for example, \cite{moreno_icml}). We also prove below Prop.~\ref{prop:diff_rho_rhot} upper bounding the regret term $R_T^{\text{diff. $\rho_t$}}$.

\subsection{Upper bound on $R_T^{\text{MDP}}$}
\begin{proposition}\label{prop:R_T_mdp}
For any $\delta \in (0,1)$, with probability at least $1-2\delta$,
    \[
    R_T^{\text{MDP}} \leq \ell N^2 \bigg( 3 \sqrt{2} |\mathcal{X}| \sqrt{ |\mathcal{A}| T \log\bigg(\frac{|\mathcal{X}| |\mathcal{A}| N T }{\delta}}\bigg) + 2 |\mathcal{X}| \sqrt{2 T \log\bigg(\frac{N}{\delta}\bigg)} \bigg).
    \]
\end{proposition}
\begin{proof}
    Using that $f_{t,n}$ is convex and $\ell$-Lipschitz with respect to the norm $\|\cdot\|_{1}$ for all $t \in [T]$ and $n \in [N]$, recalling that \( (x_{t,n}, a_{t,n})_{n \in [N]} \) denotes the agent trajectory observed at episode \( t \), and that the trajectory observed in episode $t$ is independent of the previous trajectories by the design of MDPP-K (Alg.~\ref{alg:main1}), we obtain that
    \begin{equation*}
        \begin{split}
            R_T^{\text{MDP}} &= \sum_{t=1}^T F_t(\mu^{\pi_t, \rho_t}) - F_t(\hat{\mu}^{\pi_t, \rho_t}_t) \\
            &\leq \sum_{t=1}^T \langle \nabla F_t(\mu^{\pi_t, \rho_t}), \mu^{\pi_t, \rho_t} - \hat{\mu}^{\pi_t, \rho_t}_t \rangle \\
            &\leq \ell \sum_{t=1}^T \sum_{n=1}^N \| \mu^{\pi_t, \rho_t}_n - \hat{\mu}^{\pi_t, \rho_t}_{t,n} \|_1 \\
            &\underbrace{\leq}_{\text{Lemma~\ref{lemma:bound_norm_mu_diff_mu0}}} \ell \sum_{t=1}^T \sum_{n=1}^N \sum_{i=0}^{n-1} \sum_{x,a} \mu_i^{\pi_t, \rho_t}(x,a) \|p_{i+1}(\cdot|x,a) - \hat{p}_{t,i+1}(\cdot|x,a)\|_1. \\
        \end{split}
    \end{equation*}

Applying Prop.~\ref{prop:mdp_martingale} with $\nu_t = \rho_t$, with $\pi_t$ the policy used to observe the trajectories in the true MDP, and with $\xi_{t,n}(x,a) =  \|p_{n+1}(\cdot|x,a) - \hat{p}_{t,n+1}(\cdot|x,a)\|_1 \leq \frac{C_\delta}{\sqrt{N_{t,n}(x,a)}}$, such that $C = C_\delta$, and $c = 2$, and where the inequality happens with high probability and follows from Lemma~\ref{lemma:proba_difference}, we have that
\begin{equation*}
    \begin{split}
        R_T^{\text{MDP}} &\leq \ell \sum_{n=1}^N \sum_{i=0}^{n-1} \bigg( 3 C_\delta \sqrt{|\mathcal{X}| |\mathcal{A}| T} + 2 |\mathcal{X}| \sqrt{2 T \log\bigg(\frac{N}{\delta}\bigg) \bigg)} \\
        &\leq \ell N^2 \bigg( 3 \sqrt{2} |\mathcal{X}| \sqrt{ |\mathcal{A}| T \log\bigg(\frac{|\mathcal{X}| |\mathcal{A}| N T }{\delta}}\bigg) + 2 |\mathcal{X}| \sqrt{2 T \log\bigg(\frac{N}{\delta}\bigg)} \bigg).
    \end{split}
\end{equation*}

\end{proof}

\subsection{Upper bound on $R_T^{\text{bonus}}$}
\begin{proposition}\label{prop:R_T_bonus}
For any $\delta \in (0,1)$, with probability at least $1-\delta$, 
    \[
R_T^{\text{bonus}} \leq \frac{1}{1-\alpha} \tilde{O} \big(\ell N^3 |\mathcal{X}|^{2} \sqrt{|\mathcal{A}| T} \big).
\]
\end{proposition}
\begin{proof}
    Using Holder's inequality and that $\|\ell_{t,n}\|_{\infty} \leq \ell$ for all $t \in [T]$ and $n \in [N]$, as $\ell$ is the Lipschitz constant of $f_{t,n}$ with respect to the norm $\|\cdot\|_{1}$, it follows that
    \begin{equation*}
        \begin{split}
            \sum_{t=1}^T \langle \ell_t, \hat{\mu}^{\pi, \rho}_t - \mu^{\pi,\rho} \rangle &\leq \ell \sum_{t=1}^T \sum_{n=1}^N \|\hat{\mu}_{t,n}^{\pi,\rho} - \mu_n^{\pi,\rho} \|_1 \\ &\underbrace{\leq}_{\text{Lemma~\ref{lemma:bound_norm_mu_diff_mu0}}} \ell \sum_{t=1}^T \sum_{n=1}^N \sum_{i=0}^{n-1} \sum_{x,a} \hat{\mu}^{\pi,\rho}_{t,i}(x,a) \|\hat{p}_{t,i+1}(\cdot|x,a) - p_{i+1}(\cdot|x,a) \|_1 \\
            &\underbrace{\leq}_{\text{Lemma~\ref{lemma:proba_difference}}} \ell \sum_{t=1}^T \sum_{n=1}^N \sum_{i=0}^{n-1} \sum_{x,a} \hat{\mu}^{\pi,\rho}_{t,i}(x,a) \frac{C_\delta}{\sqrt{N_{t,i}(x,a)}} \\
            &= \ell \sum_{t=1}^T \sum_{n=0}^N (N-n) \sum_{x,a} \hat{\mu}^{\pi,\rho}_{t,n}(x,a) \frac{C_\delta}{\sqrt{N_{t,n}(x,a)}} \\
            &= \sum_{t=1}^T \langle \bar{b}_t,  \hat{\mu}^{\pi,\rho}_t \rangle + \sum_{t=1}^T \langle \bar{b}_{t,0},  \rho  \rangle,
        \end{split}
    \end{equation*}
    where $\bar{b}_t := (\bar{b}_{t,n})_{n \in [N]}$ is the sequence of bonus vectors defined for all $n \in \{0, \ldots, N\}$ and $(x,a)$ as in Eq.~\eqref{eq:bonus} by
    \[
    \bar{b}_{t,n}(x,a) := \ell (N-n) \frac{C_\delta}{\sqrt{N_{t,n}(x,a)}},
    \]
    with $C_\delta := \sqrt{2 |\mathcal{X}| \log \bigg(\frac{|\mathcal{X} | |\mathcal{A}| N T }{\delta}} \bigg)$.

Hence, replacing it in the bonus regret term we have that
\begin{equation*}
    \begin{split}
        R_T^{\text{bonus}} &= \sum_{t=1}^T \langle \bar{b}_t, \mu_t - \hat{\mu}^{\pi, \rho}_t \rangle + \langle \ell_t,  \hat{\mu}^{\pi, \rho}_t - \mu^{\pi, \rho} \rangle \\
        &\leq  \sum_{t=1}^T \langle \bar{b}_t, \mu_t \rangle + \sum_{t=1}^T \langle \bar{b}_{t,0},  \rho  \rangle \\
        &= \underbrace{\sum_{t=1}^T \langle \bar{b}_t, \mu_t \rangle + \langle \bar{b}_{t,0}, \rho_t \rangle}_{(i)} + \underbrace{\sum_{t=1}^T \langle \bar{b}_{t,0}, \rho - \rho_t \rangle}_{(ii)}.
    \end{split}
\end{equation*}

Note that term $(ii)$ arises because we cannot guarantee that the initial distribution at each episode is equal to $\rho$. This issue does not occur when using bonuses in episodic MDPs, as done in \cite{moreno_icml}. We analyze this additional term in more detail below.

As $\mu_0^t = \rho_t$, we have from Corollary~\ref{cor:bonus_analysis} that the first term satisfies for any $\delta \in (0,1)$, with probability at least $1-\delta$, 
\[
(i) = \sum_{t=1}^T \langle \bar{b}_t, \mu_t \rangle + \langle \bar{b}_{t,0}, \rho_t \rangle \leq 15 \ell N^3 |\mathcal{X}|^{3/2} \sqrt{|\mathcal{A}| T} \log\bigg(\frac{|\mathcal{X}| |\mathcal{A}| N T}{\delta} \bigg) \sqrt{\log\bigg(\frac{N}{\delta}\bigg)}.
\]

We now proceed with the second term. Using Holder's inequality, that $\|\bar{b}_{t,0}\|_\infty \leq \ell N C_\delta$, and Lemma~\ref{lemma:almost_equal_dist}, we have that with high probability
\begin{equation*}
    \begin{split}
        (ii) = \sum_{t=1}^T \langle \bar{b}_{t,0}, \rho - \rho_t \rangle &\leq \sum_{t=1}^T \|\bar{b}_{t,0} \|_\infty \|\rho - \rho_t \|_1 \\
        &\leq \ell N C_\delta \sum_{t=1}^T \|\rho - \rho_t \|_1 \\
        &\leq \ell N C_\delta  \frac{30 N^2}{1-\alpha} |\mathcal{X}|^{3/2} \sqrt{|\mathcal{A}| T}  \log\bigg(\frac{|\mathcal{X}| |\mathcal{A}| N T}{\delta} \bigg) \sqrt{\log\bigg(\frac{N}{\delta}\bigg)} \\
        &= \frac{\ell 30 \sqrt{2} N^3}{1-\alpha} |\mathcal{X}|^2 \sqrt{|\mathcal{A}| T}  \log^{3/2}\bigg(\frac{|\mathcal{X}| |\mathcal{A}| N T}{\delta} \bigg) \sqrt{\log\bigg(\frac{N}{\delta}\bigg)},
    \end{split}
\end{equation*}
where for the last inequality we use the definition of $C_\delta$.

Joinining both terms, we obtain that with high probability, 
\begin{equation*}
    \begin{split}
        R_T^{\text{bonus}} &\leq  15 \ell N^3 |\mathcal{X}|^{3/2} \sqrt{|\mathcal{A}| T} \log\bigg(\frac{|\mathcal{X}| |\mathcal{A}| N T}{\delta} \bigg) \sqrt{\log\bigg(\frac{N}{\delta}\bigg)} \\
        &\quad \quad+ \frac{\ell 30 \sqrt{2} N^3}{1-\alpha} |\mathcal{X}|^2 \sqrt{|\mathcal{A}| T}  \log^{3/2}\bigg(\frac{|\mathcal{X}| |\mathcal{A}| N T}{\delta} \bigg) \sqrt{\log\bigg(\frac{N}{\delta}\bigg)} \\
        &=\tilde{O} \bigg(\frac{\ell N^3}{1-\alpha} |\mathcal{X}|^{2} \sqrt{|\mathcal{A}| T} \bigg).
    \end{split}
\end{equation*}
\end{proof}

\subsection{Upper bound on $R_T^{\text{diff. $\rho_t$}}$: proof of Prop.~\ref{prop:diff_rho_rhot}}\label{app:proof_rho_rhot}
\begin{proof}
    Recall that $\|\ell_{t,n}\|_\infty \leq \ell$ as $\ell$ is the Lipschitz constant of $f_{t,n}$ with respect to the norm $\|\cdot \|_1$ for all $n \in [N]$ and $t \in [T]$. Given the definition of $\bar{b}_t$ in Eq.~\eqref{eq:bonus}, we have that for all $t \in [T]$ and $n \in [N]$, $\|\bar{b}_{t,n} \|_\infty \leq N C_\delta \ell$. Hence, from Holder's inequality we have that with high probability
    \begin{equation*}
        \begin{split}
            R_T^{\text{diff, $\rho_t$}} &= \sum_{t=1}^T \langle \ell_t - \bar{b}_t, \hat{\mu}^{ \pi, \rho_t}_t - \hat{\mu}^{\pi, \rho}_t \rangle \\
            &\leq 2 \ell N C_\delta \sum_{t=1}^T \sum_{n=1}^N \|\hat{\mu}_{t,n}^{ \pi, \rho_t} - \hat{\mu}_{t,n}^{ \pi, \rho} \|_1 \\
            &\underbrace{\leq}_{\text{Lemma~\ref{lemma:bound_norm_mu_diff_mu0}}} 2 \ell N C_\delta \sum_{t=1}^T N \|\rho_t - \rho \|_1 \\
            &\underbrace{\leq}_{\text{Lemma~\ref{lemma:almost_equal_dist}}} 2 \ell N^2 C_\delta \frac{ 30 N^2}{1-\alpha} |\mathcal{X}|^{3/2} \sqrt{|\mathcal{A}| T}  \log\bigg(\frac{|\mathcal{X}| |\mathcal{A}| N T}{\delta} \bigg) \sqrt{\log\bigg(\frac{N}{\delta}\bigg)} \\
            &= \tilde{O}\bigg(\frac{1}{1-\alpha} \ell N^4 |\mathcal{X}|^2 \sqrt{|\mathcal{A}| T} \bigg).
        \end{split}
    \end{equation*}
\end{proof}

\subsection{Final result: proof of Thm.~\ref{thm:main_periodic_regret_known_rhot}}
\begin{proof}
    Following the decomposition of the periodic regret in Sec.~\ref{sec:algorithm}, and using the convexity of $F_t$ with $\ell_t := \nabla F_t(\mu_t)$, we have that for all periodic policy $\pi$, 
    \begin{equation*}
        \begin{split}
            R_T(\pi) &\leq \underbrace{\sum_{t=1}^T F_t(\mu^{\pi_t, \rho_t}) - F_t(\hat{\mu}^{\pi_t, \rho_t}_t)}_{R_T^{\text{MDP}}} + \underbrace{\sum_{t=1}^T \langle \ell_t - \bar{b}_t, \mu_t - \hat{\mu}^{\pi, \rho_t}_t \rangle}_{R_T^{\text{MD}}} 
            + \underbrace{\sum_{t=1}^T \langle \ell_t - \bar{b}_t, \hat{\mu}^{\pi, \rho_t}_t - \hat{\mu}^{\pi, \rho}_t \rangle}_{R_T^{\text{diff. }\rho_t}} \\
            &+ \underbrace{\sum_{t=1}^T \langle \bar{b}_t, \mu_t - \hat{\mu}^{\pi, \rho}_t \rangle + \langle \ell_t, \hat{\mu}^{\pi, \rho}_t - \mu^{\pi, \rho} \rangle}_{R_T^{\text{bonus}}} + \underbrace{\gamma \sum_{t=1}^T  \|\rho_t - \rho \|_1}_{R_T^{\text{reg}}}.
        \end{split}
    \end{equation*}

We analyze each term separately:
\begin{itemize}
    \item From Prop.~\ref{prop:R_T_mdp}, we have that with high probability
    \begin{equation*}
        \begin{split}
             R_T^{\text{MDP}} &\leq \ell N^2 \bigg( 3 \sqrt{2} |\mathcal{X}| \sqrt{ |\mathcal{A}| T \log\bigg(\frac{|\mathcal{X}| |\mathcal{A}| N T }{\delta}}\bigg) + 2 |\mathcal{X}| \sqrt{2 T \log\bigg(\frac{N}{\delta}\bigg)} \bigg) \\
             &= \tilde{O}\big(\ell N^2 |\mathcal{X}| \sqrt{|\mathcal{A}| T} \big).
        \end{split}
    \end{equation*}

    \item From Prop.~\ref{prop:bound_md_term}, we have that with high probability
    \[
    R_T^{\text{MD}} \leq  \tilde{O} \bigg( \ell N^3 |\mathcal{X}|^{5/4} |\mathcal{A}|^{1/4} \sqrt{\frac{\Psi}{1 - \alpha}} T^{3/4} \bigg).
    \]

    \item From Prop.~\ref{prop:diff_rho_rhot}, we obtain that with high probability
    \[
   R_T^{\text{diff. $\rho_t$}} \leq \tilde{O}\bigg(\frac{1}{1-\alpha} \ell N^4 |\mathcal{X}|^2 \sqrt{|\mathcal{A}| T} \bigg).
    \]

    \item From Prop.~\ref{prop:R_T_bonus} we have that with high probability,
    \[
    R_T^{\text{bonus}} \leq  \tilde{O} \big(\frac{1}{1-\alpha} \ell N^3 |\mathcal{X}|^{2} \sqrt{|\mathcal{A}| T} \big).
    \]

    \item Finally, from Lemma~\ref{lemma:almost_equal_dist}, 
    \[
    R_T^{\text{reg}} \leq \tilde{O}\bigg(\frac{\gamma}{1-\alpha} N^2 |\mathcal{X}|^{3/2} \sqrt{|\mathcal{A}| T } \bigg).
    \]
\end{itemize}

By replacing all upper bounds in the decomposition of the periodic regret, we conclude that, with high probability, for all periodic policies $\pi$,
\[
R_T(\pi) \leq \tilde{O} \bigg(  \ell N^2 |\mathcal{X}|^{5/4} |\mathcal{A}|^{1/4} \sqrt{\frac{\Psi}{1 - \alpha}} T^{3/4} +  \frac{1}{1-\alpha} \ell N^4 |\mathcal{X}|^{2} \sqrt{|\mathcal{A}| T}   \bigg).
\]
\end{proof}

\section{Practical Solution Using Dynamic Programming}\label{app:practical_sol}

We show in this section how to solve the optimization problem in Eq.~\eqref{iteration_md_solver} of the main paper that we use to build the main algorithm in practice:
\begin{equation*}
\begin{aligned}
   \mu_{t+1} \in \argmin_{\mu \in \mathcal{M}_{t+1}} \quad &\Big\{ \eta \langle \ell_{t} - \bar{b}_t, \mu \rangle +  D_\psi(\mu, \mu_{t}) \Big\}\\
\textrm{s.t.} \quad  &\| \mu_N - \rho \|_1 \leq \langle \mu, b_t \rangle +  \bar{\alpha} \|\rho_t - \rho\|_1.
\end{aligned}
\end{equation*}

Let \(\lambda \geq 0\) be the Lagrange multiplier corresponding to the inequality constraint. The Lagrangian is defined as
\begin{equation*}
    \begin{split}
        \mathcal{L}(\mu, \lambda) &:= \langle \ell_{t} - \bar{b}_t, \mu \rangle + \frac{1}{\eta} D_\psi(\mu, \mu_{t}) + \lambda [\| \mu_N - \rho \|_1 - \langle \mu, b_t \rangle -  \bar{\alpha} \|\rho_t - \rho\|_1].
    \end{split}
\end{equation*}
We also represent the constraint function as
\[
G(\mu) := \| \mu_N - \rho \|_1 - \langle \mu, b_t \rangle - \bar{\alpha} \|\rho_t - \rho\|_1.
\]

Assume there exists \(\nu\in \mathcal{M}_{t+1}\) that strictly satisfies the constraint, \emph{i.e.}, \( G(\nu) < 0 \) (to ensure that, we can slightly increase the value of \(\alpha\) in practice). Since the problem is convex, strong duality holds, and thus
\begin{equation}\label{eq:consequence_strong_duality}
\begin{aligned}
   \mu_{t+1} \in \argmin_{\mu \in \mathcal{M}_{t+1}} &\quad \Big\{ \eta \langle \ell_{t} - \bar{b}_t, \mu \rangle +  D_\psi(\mu, \mu_{t}) \Big\} \quad := \quad  \max_{\lambda \geq 0} \min_{\mu \in \mathcal{M}_{t+1}} \mathcal{L}(\mu, \lambda).\\
\textrm{s.t.} \quad  & G(\mu) \leq 0
\end{aligned}
\end{equation}

\paragraph{Minimizing the Lagrangian for a given \(\lambda\):} For a fixed \(\lambda \geq 0\), we provide a solution for \(\min_{\mu \in \mathcal{M}_{t+1}} \mathcal{L}(\mu, \lambda)\) when the Bregman divergence is defined as in Eq.~\eqref{gamma:non_standard}. We derive a closed-form solution for the policies using dynamic programming.

Since the following procedure applies to all \( t \in [T] \) and any probability transitions, we omit the episode index and the fact that the optimization is performed on the set of distribution sequences induced by the estimated MDP. Thus, instead of using \( \ell_t, \bar{b}_t, b_t, \mathcal{M}_{t+1}, \mu_t, \pi_t, \hat{\mu}^{ \pi, \rho_{t+1}}_{t+1}, \hat{p}_{t+1}, \rho_{t+1} \), we simplify the notation to \( \ell, \bar{b}, b, \mathcal{M}, \bar{\mu}, \bar{\pi}, \mu^\pi, \bar{p}, \bar{\rho} \).

Let $\ell^{\lambda} := \ell- \bar{b} - \lambda b$. Hence,
\begin{equation}\label{eq:lagrangian_problem_decomposed}
    \begin{split}
        &\min_{\mu \in \mathcal{M}} \mathcal{L}(\mu, \lambda) = \min_{\mu \in \mathcal{M}} \Big\{ \langle \ell^{ \lambda}, \mu \rangle + \frac{1}{\eta} D_\psi(\mu, \bar{\mu}) + \lambda \|\mu_N - \rho \|_1 \Big\} \\
        &=\min_{\mu \in \mathcal{M}} \bigg\{ \sum_{n=1}^N \sum_{x,a} \ell^{\lambda}_n(x,a) \mu_n(x,a) + \frac{1}{\eta} \sum_{n=1}^N \sum_{x,a} \mu_n(x,a) \log\bigg( \frac{\pi_n(a|x)}{\bar{\pi}_n(x,a)} \bigg) + \lambda \sum_{x,a} | \mu_N(x,a) - \rho(x,a) | \bigg\} \\
        &=\min_{\mu \in \mathcal{M}} \bigg\{ \sum_{n=1}^N \sum_{x,a}  \mu_n(x,a) \bigg[\ell^{\lambda}_n(x,a)+ \frac{1}{\eta} \log\bigg( \frac{\pi_n(a|x)}{\bar{\pi}_n(x,a)} \bigg) \bigg] + \lambda \sum_{x,a} | \mu_N(x,a) - \rho(x,a) | \bigg\} \\
        &=\min_{\pi \in \Pi} \bigg\{ \sum_{n=1}^N \sum_{x,a}  \mu^{\pi}_n(x,a) \bigg[\ell^{\lambda}_n(x,a)+ \frac{1}{\eta} \log\bigg( \frac{\pi_n(a|x)}{\bar{\pi}_n(x,a)} \bigg) \bigg] + \lambda \sum_{x,a} | \mu^{\pi}_N(x,a) - \rho(x,a) | \bigg\} \\
        &= \min_{\pi \in \Pi} \bigg\{ \mathbb{E} \bigg[ \sum_{n=1}^N \ell^\lambda_n(x_n,a_n) + \log\bigg(\frac{\pi_n(a_n|x_n)}{\bar{\pi}_n(a_n|x_n)} \bigg) \bigg| \bar{\rho}, \bar{p}, \pi \bigg] + \lambda \sum_{x,a} \Big| \mathbb{E}\big[\mathds{1}_{\{(x_N,a_N) = (x,a)\} } \big| \bar{\rho}, \bar{p}, \pi\big] - \rho(x,a) \Big| \bigg\}.
    \end{split}
\end{equation}

We define the state-action value function sequence $(Q_i)_{i \in [N]}$ where, for $i = N$, for all $(x,a) \in \mathcal{X} \times \mathcal{A}$,
\begin{equation*}
    \begin{split}
        Q^\lambda_N(x,a) &= \ell_N^\lambda(x,a) +  \lambda \sum_{\bar{x}, \bar{a}} \Big| \mathbb{E} \big[\mathds{1}_{\{ (x_N, a_N) = (\bar{x}, \bar{a}) \}} \big| (x_N,a_N) = (x,a) \big] - \rho(\bar{x}, \bar{a}) \Big| \\
        &=  \ell_N^\lambda(x,a)  + \lambda \sum_{\bar{x}, \bar{a}} \Big| \mathds{1}_{\{ (x, a) = (\bar{x}, \bar{a}) \}} - \rho(\bar{x} ,\bar{a} ) \Big| \\
        &=  \ell_N^\lambda(x,a)  + \lambda 2 (1 - \rho(x,a) ),
    \end{split}
\end{equation*}
and for each \( i \in \{1, \ldots, N-1\} \), where we denote \( \pi_{i+1:N} := (\pi_j)_{j \in \{i+1, \ldots, N\}} \),
\begin{equation*}
\begin{split}
    Q^\lambda_i(x,a) &:= \min_{\pi_{i+1:N}}  \bigg\{ \ell^\lambda_i(x,a) + \mathbb{E} \bigg[ \sum_{n=i+1}^N \ell^\lambda_n(x_n,a_n) + \log\bigg(\frac{\pi_n(a_n|x_n)}{\bar{\pi}_n(a_n|x_n)} \bigg) \bigg| (x_i,a_i) = (x,a), \bar{p}, \pi \bigg] \\
      &\quad \quad \quad \quad + \lambda \sum_{\bar{x},\bar{a}} \Big| \mathbb{E}\big[\mathds{1}_{\{(x_N,a_N) = (\bar{x},\bar{a})\} } \big| (x_i,a_i) = (\bar{x},\bar{a}),  \bar{p}, \pi\big] - \rho(\bar{x},\bar{a}) \Big| \bigg\}. \\
\end{split}
\end{equation*}
Note that this state-action value function satisfies
\begin{equation}\label{eq:state_action_func}
\begin{split}
     Q^\lambda_i(x,a) &= \ell^\lambda_i(x,a) + \min_{\pi_{i+1:N}}  \Bigg\{ \sum_{x'} \bar{p}_{i+1}(x'|x,a) \sum_{a'} \pi_{i+1}(a'|x') \Bigg[ \ell^{\lambda}_{i+1}(x',a') + \log\bigg(\frac{\pi_{i+1}(a'|x')}{\bar{\pi}_{i+1}(a'|x')} \bigg) \\
    &\quad \quad \quad \quad + \mathbb{E}\bigg[ \sum_{n=i+2}^N \ell^\lambda_n(x_n,a_n) + \log\bigg(\frac{\pi_n(a_n|x_n)}{\bar{\pi}_n(a_n|x_n)} \bigg) \bigg| (x_{i+1}, a_{i+1}) = (x',a'), \bar{p}, \pi \bigg]\\
    &\quad \quad \quad \quad + \lambda \sum_{\bar{x},\bar{a}}  \Big| \mathbb{E}\big[\mathds{1}_{\{(x_N,a_N) = (\bar{x},\bar{a})\} } \big| (x_{i+1},a_{i+1}) = (x',a'),  \bar{p}, \pi\big] - \rho(\bar{x},\bar{a}) \Big| \Bigg] \Bigg\} \\
    &=\ell^\lambda_i(x,a) +  \min_{\pi_{i+1}} \bigg\{ \sum_{x'} \bar{p}_{i+1}(x'|x,a) \sum_{a'} \pi_{i+1}(a'|x') \bigg[\log\bigg(\frac{\pi_{i+1}(a'|x')}{\bar{\pi}_{i+1}(a'|x')} \bigg) + Q^\lambda_{i+1}(x',a') \bigg] \bigg\}.
\end{split}
\end{equation}
Observe that computing \( \mathbb{E}_{(x,a) \sim \bar{\rho}}[Q^\lambda_0(x,a)] \) is equivalent to solving the Lagrangian minimization for a fixed $\lambda$ given the decomposition in Eq.~\eqref{eq:lagrangian_problem_decomposed}. Hence, our goal is to find the sequence of policies that is used to compute \( \mathbb{E}_{(x,a) \sim \bar{\rho}}[Q_0(x,a)] \). Using the relation satisfied by the state-action value function shown in Eq.~\eqref{eq:state_action_func}, we can compute each element in the sequence of optimal policies, that we denote by $(\pi_n^\lambda)_{n \in [N]}$, separately backwards in time by taking for all $x \in \mathcal{X}$
\[
\pi^\lambda_{i+1}(\cdot|x) = \argmin_{\pi(\cdot|x) \in \Delta_\mathcal{A}} \big\{ \langle \pi(\cdot|x), Q^\lambda_{i+1}(x,\cdot) \rangle + \frac{1}{\eta} \text{KL}\big( \pi(\cdot|x), \bar{\pi}_{i+1}(\cdot|x) \big) \big\},
\]
which is well known to have a closed-form solution (see for example \cite{pmlr-v238-moreno24a}) with for all $(x,a) \in \mathcal{X} \times \mathcal{A}$,
\begin{equation}\label{eq:closed_form}
\pi^\lambda_{i+1}(a|x) = \frac{\bar{\pi}_{i+1}(a|x) \exp(-\eta Q^\lambda_{i+1}(x,a) ) }{\sum_{a'} \bar{\pi}_{i+1}(a'|x) \exp(-\eta Q^\lambda_{i+1}(x,a') )  }.
\end{equation}

Finally, using the recursive relation from Eq.~\eqref{mu_induced_pi}, we can recover the state-action distribution sequence \( \mu^\lambda \) corresponding to \( \pi^\lambda \), which is the solution to the Lagrangian minimization $\min_{\mu \in \mathcal{M}} \mathcal{L}(\mu,\lambda)$ for a fixed Lagrange multiplier $\lambda$.

\paragraph{Maximizing the Lagrangian:} For a given Lagrange multiplier $\lambda > 0$, we previously described how to compute $\mu^\lambda \in \argmin_{\mu \in \mathcal{M}} \mathcal{L}(\mu, \lambda)$. To find the pair $(\mu^{\lambda_*}, \lambda_*) \in \argmin_{\lambda \geq 0} \min_{\mu \in \mathcal{M}} \mathcal{L}(\mu^\lambda, \lambda)$, we adopt the min-max iterative procedure outlined in Algorithm~\ref{alg:min_max}. We begin by initializing the multiplier $\lambda_1 \geq 0$ and proceed iteratively. At each iteration $s$, we compute the exact state-action distribution sequence $\mu^{\lambda_s}$ associated with the current multiplier $\lambda_s$, following the method described above. We then check whether the inequality constraint is satisfied up to a small constant $\varepsilon$, \emph{i.e.}, $G(\mu^{\lambda_s}) \leq \varepsilon$. If it is, we return $\mu^{\lambda_s}$ as the solution to the original optimization problem. Otherwise, we update the multiplier by performing a gradient ascent step with a tunable learning rate $\eta_\lambda$. We omit the analysis for Algorithm~\ref{alg:min_max} as it is not the main focus of this work and it follows from the classic gradient ascent analysis (see \cite{francis}).

\begin{algorithm}
    \caption{Lagrangian min-max approach}\label{alg:min_max}
    \begin{algorithmic}
        \STATE {\bfseries Input:} initial multiplier $\lambda_1$, learning rate $\eta_\lambda$, $\varepsilon > 0$
        \WHILE{$G(\mu^{\lambda_s}) > \varepsilon$}
        \STATE Update Lagrangian multiplier: $\lambda_{s+1} = \lambda_s + \eta_\lambda G(\mu^{\lambda_s})$
        \STATE Compute the associated state-action distribution $\mu^{\lambda_{s+1}}$
        \ENDWHILE
    \end{algorithmic}
\end{algorithm}

\paragraph{Computing $\bar{\alpha}$:} In practice, the algorithm only requires a value $\bar{\alpha}$ that ensures feasibility. From Lemma~\ref{lemma:almost_equal_dist}, we know that Eq.~\eqref{iteration_md_solver} is feasible for any $\bar{\alpha} \geq \alpha$. Therefore, we can construct a grid over the interval $(0,1)$ and search for the smallest value of $\bar{\alpha}$ in this grid for which the problem remains feasible. To test feasibility, the $L_1$ constraint in Eq.~\eqref{iteration_md_solver} can be reformulated as a set of linear inequalities, allowing the use of a solver to efficiently verify feasibility.

\section{Main results unknown $\rho_t$}\label{app:setting2}

In this section, we present the main results for the second framework, where $\rho_t$ is not observed by the learner. The algorithmic scheme is detailed in Algorithm~\ref{alg:main2} in App.~\ref{app:algo_scheme}, with the key differences with respect to the method for the first framework highlighted in blue. 

To establish the regret bounds for this framework, we also introduce Ass.~\ref{ass:ergodicity} below.

\begin{assumption}\label{ass:ergodicity}
    There exists some $0 \leq \alpha < 1$ such that for all policies $\pi \in \Pi$ and distributions $\nu, \nu' \in \Delta_{\mathcal{X} \times \mathcal{A}}$,
    \[
    \|\nu P_\pi - \nu' P_\pi \|_1 \leq \alpha \|\nu - \nu'\|_1.
    \]
\end{assumption}

\subsection{Estimating $\rho_t$}

We build an estimate of the initial state-action distribution at episode $t$ by designating one of the $M$ agents as an agent that can be restarted at the beginning of each episode. We refer to this agent as the \emph{special} agent. Importantly, only this single agent is restarted, unlike in standard episodic RL, where all agents would be reset, making this a weaker assumption. For further discussion, see Sec.~\ref{sec:algorithm_unknown_rhot} in the main paper. Let $\tilde{\rho}_t$ denote the estimate of $\rho_t$ for all $t \in [T]$.

Let $\tilde{\ell}_{t-1} := \nabla F_{t-1}(\hat{\mu}^{ \pi_{t-1}, \tilde{\rho}_{t-1}}_{t-1})$, and let $\widetilde{\mathcal{M}}_{t} := \mathcal{M}_{\tilde{\rho}_{t}}^{\hat{p}_{t}}$, the set of state-action distributions sequences that start at the initial distribution $\tilde{\rho}_{t}$, and that satisfies the dynamics of the MDP induced by the estimated probability kernel $\hat{p}_{t}$. Algorithm~\ref{alg:main2} computes the policy $\pi_{t}$ by running each online mirror descent iteration over the distributions initialized in $\tilde{\rho}_{t}$. Hence, at episode $t$, for some $\eta >0 $ to be tuned later, the learner computes
\begin{equation}\label{eq:iterative_scheme_setting2}
\begin{aligned}
\mu_{t} \in \argmin_{\mu \in \widetilde{\mathcal{M}}_{t}} &\left\{ \eta \langle \tilde{\ell}_{t-1} - \bar{b}_{t-1}, \mu \rangle + D_\psi(\mu, \mu_{t-1}) \right\} \\
\textrm{s.t.} \quad  &\| \mu_N - \rho \| \leq \langle \mu, b_{t-1} \rangle +  \bar{\alpha} \|\tilde{\rho}_{t-1} - \rho\|,
\end{aligned}
\end{equation}
and takes $\pi_t$ as the policy associated to the state-action distribution $\mu_t$ as discussed in Sec.~\ref{sec:learning_problem}.

We denote the trajectory of the restarted agent under policy $\pi_{t}$ by $(\tilde{x}_{t,n}, \tilde{a}_{t,n})_{n \in [N]}$. At the beginning of episode $t$, the restarted agent is restarted at an initial state-action pair sampled from $\tilde{\rho}_t$, \emph{i.e.}, $(\tilde{x}_{t,0}, \tilde{a}_{t,0}) \sim \tilde{\rho}_t$. For each tuple $(n, x, a, x')$, we define the corresponding counts as follows
\begin{align*}
    \widetilde{N}_{t,n}(x,a) = \sum_{s=1}^{t-1} \mathds{1}_{\{\tilde{x}_{s,n} = x, \tilde{a}_{s,n} = a\}}, \quad \quad \widetilde{M}_{t,n}(x'|x,a) = \sum_{s=1}^{t-1} \mathds{1}_{\{\tilde{x}_{s,n+1} = x' ,\tilde{x}_{s,n} = x, \tilde{a}_{s,n} = a\}}.
\end{align*}
We also define another estimate of the transition probability kernel at the end of episode $t-1$ as
\[
\tilde{p}_{t,n+1}(x'|x,a) := \frac{\widetilde{M}_{t,n}(x'|x,a)}{\max\{1, \widetilde{N}_{t,n}(x,a) \}}.
\]

We define the stochastic matrix corresponding to $\tilde{p}_t := (\tilde{p}_{t,n})_{n \in [N]}$ under a policy $\pi$, as in Eq.~\eqref{eq:markov_chain_proba}, and denote it by $\widetilde{P}_{\pi}^t$. We then set $\tilde{\rho}_1 := \rho$, and for each episode $1 < t \leq T$, we recursively define $\tilde{\rho}_t := \tilde{\rho}_{t-1} \widetilde{P}_{\pi_t}^t$.

We state and prove below Prop.~\ref{prop:estimate_rhot}, giving a result on the accuracy of the estimated initial state-action distribution $\tilde{\rho}_t$ computed and used in Algorithm~\ref{alg:main2}.
\begin{proposition*}
Assuming Ass.~\ref{ass:ergodicity} holds for \(0 \leq \alpha < 1\), we obtain that for any $\delta \in (0,1)$, with probability at least $1-2\delta$,
    \[
    \sum_{t=1}^T \|\rho_t - \tilde{\rho}_t \|_1 \leq \frac{1}{1- \alpha} \bigg[3 N |\mathcal{X}| \sqrt{2 |\mathcal{A}| T \log\bigg(\frac{|\mathcal{X}| |\mathcal{A}| N T}{\delta} \bigg)} + 2 N |\mathcal{X}| \sqrt{2 T \log\bigg(\frac{N}{\delta} \bigg)}  \bigg].
    \]
\end{proposition*}
\begin{proof}

 For $i < j$, for all $\nu \in \Delta_{\mathcal{X} \times \mathcal{A}}$, let $\nu P_{\pi_{i:j}} := \nu P_{\pi_i} P_{\pi_{i+1}} \ldots P_{\pi_j}$, and $\nu \widetilde{P}_{\pi_{i:j}}^{i:j} := \nu \widetilde{P}^i_{\pi_i} \widetilde{P}^{i+1}_{\pi_{i+1}} \ldots \widetilde{P}^{j}_{\pi_j}$. Hence, note that $\tilde{\rho}_t := \rho \widetilde{P}^{1:t-1}_{\pi_{1:t-1}}$, and $\rho_t = \rho P^{1:t-1}_{\pi_{1:t-1}}$. Thus, we have that for each $t \in [T]$,
\begin{equation*}
    \begin{split}
        \rho_t  -\tilde{\rho}_t &= \rho P_{\pi_{1:t-1}} - \rho \widetilde{P}^{1:t-1}_{\pi_{1:t-1}} \\
        &= \rho ( P_{\pi_1} - \widetilde{P}^1_{\pi_1} ) P_{\pi_{2:t-1}} + \rho \widetilde{P}_{\pi_1}^1 ( P_{\pi_2} - \widetilde{P}^2_{\pi_2} ) P_{\pi_{3:t-1}} + \ldots + \rho \widetilde{P}^{1:t-2}_{\pi_1:t-2} (P_{\pi_{t-1}} - \widetilde{P}^{t-1}_{\pi_{t-1}} ) \\
        &= \sum_{i=1}^{t-1} \rho \widetilde{P}^{1:i-1}_{\pi_{1:i-1}} (P_{\pi_i} - \widetilde{P}^i_{\pi_i} ) P_{\pi_{i+1:t-1}},
    \end{split}
\end{equation*}
where $\widetilde{P}^{1:0}_{\pi_{1:0}} = 1$ and $P_{\pi_{t:t-1}} = 1$. Thus, by taking the $L_1$ norm and summing over $t \in [T]$ we get that
\begin{equation}\label{eq:current_inequality}
\begin{split}
    \sum_{t=1}^T \|  \rho_t  -\tilde{\rho}_t \|_1 &\leq \sum_{t=1}^T \sum_{i=1}^{t-1} \| \rho \widetilde{P}^{1:i-1}_{\pi_{1:i-1}} (P_{\pi_i} - \widetilde{P}^i_{\pi_i} ) P_{\pi_{i+1:t-1}} \|_1 \\
    &\leq \sum_{t=1}^T \sum_{i=1}^{t-1} \alpha^{t-i-1} \|\rho \widetilde{P}^{1:i-1}_{\pi_{1:i-1}} (P_{\pi_i} - \widetilde{P}^i_{\pi_i} )  \|_1 \\
    &= \sum_{t=1}^T \sum_{i=1}^{t-1} \alpha^{t-i-1} \|\tilde{\rho}_i (P_{\pi_i} - \widetilde{P}^i_{\pi_i} )  \|_1 \\
\end{split}
\end{equation}
where for the before to last inequality we use Ass.~\ref{ass:ergodicity}, and for the last equality we use that $\rho \widetilde{P}^{1:i-1}_{\pi_{1:i-1}} = \tilde{\rho}_i$. 

We now express the final $L_1$ norm using the state-action distribution notation. Let $\mu^{\pi, \nu}$ denote the sequence of state-action distributions induced by executing policy $\pi$ in an episode starting from initial distribution $\nu$ under the true MDP, as defined in Eq.~\eqref{mu_induced_pi}. Similarly, let $\tilde{\mu}^{\pi, \nu}_t$ represent the corresponding distribution in the MDP defined by the estimated transition model $\tilde{p}_t$. In particular, we have that $\tilde{\rho}_i P_{\pi_i} = \mu_N^{\pi_i, \tilde{\rho}_i}$ and $\tilde{\rho}_i \widetilde{P}^i_{\pi_i} = \tilde{\mu}_{i,N}^{ \pi_i, \tilde{\rho}_i}$. Therefore, from the inequality above we have that
\begin{equation}\label{eq:decomp_rho_rhotilde}
    \begin{split}
        \sum_{t=1}^T \|  \rho_t  -\tilde{\rho}_t \|_1 &\leq \sum_{t=1}^T \sum_{i=1}^{t-1} \alpha^{t-i-1} \| \mu_N^{\pi_i, \tilde{\rho}_i} - \tilde{\mu}_{i,N}^{ \pi_i, \tilde{\rho}_i} \|_1 \\
        &\underbrace{\leq}_{\text{Lemma~\ref{lemma:bound_norm_mu_diff_mu0}}} \sum_{t=1}^T \sum_{i=1}^{t-1} \alpha^{t-i-1} \sum_{j=0}^{N-1} \sum_{x,a} \mu_j^{\pi_i, \tilde{\rho}_i}(x,a) \|p_{j+1}(\cdot|x,a) - \tilde{p}_{i, j+1}(\cdot|x,a) \|_1 \\
        &= \sum_{t=1}^{T-1} \alpha^{t-1} \sum_{i=1}^{T-t} \sum_{j=0}^{N-1} \sum_{x,a} \mu_j^{\pi_i, \tilde{\rho}_i}(x,a) \|p_{j+1}(\cdot|x,a) - \tilde{p}_{i, j+1}(\cdot|x,a) \|_1 .
    \end{split}
\end{equation}

We apply Prop.~\ref{prop:mdp_martingale} to the trajectory of the restarted agent $(\tilde{x}_{t,n}, \tilde{a}_{t,n})_{n \in [N]}$ when following the policy $\pi_t$ (solution to Eq.~\eqref{eq:iterative_scheme_setting2}), with initial distribution $\nu_t = \tilde{\rho}_t$, and where $\xi_{t,n}(x,a) = \|p_{n+1}(\cdot \mid x,a) - \tilde{p}_{t,n}(\cdot \mid x,a)\|_1 \leq C_\delta/\sqrt{\max\{1, \widetilde{N}_{t,n}(x,a)\}}$. This inequality follows from Lemma~\ref{lemma:proba_difference}, allowing us to set $C = C_\delta$ and $c = 2$. Since the restarted agent is restarted at the beginning of each episode from $\tilde{\rho}_t$, the resulting trajectories are independent across episodes. Therefore, Prop.~\ref{prop:mdp_martingale} yields that
\begin{equation*}
    \begin{split}
        &\sum_{i=1}^{T-t} \sum_{j=0}^{N-1} \sum_{x,a} \mu_j^{\pi_i, \tilde{\rho}_i}(x,a) \|p_{j+1}(\cdot|x,a) - \tilde{p}_{i, j+1}(\cdot|x,a) \|_1\\
        &\quad \quad \leq 3 N |\mathcal{X}| \sqrt{2 |\mathcal{A}| (T-t) \log\bigg(\frac{|\mathcal{X}| |\mathcal{A}| N T}{\delta} \bigg)} + 2 N |\mathcal{X}| \sqrt{2 (T-t) \log\bigg(\frac{N}{\delta} \bigg)} .
    \end{split}
\end{equation*}

Replacing it in the decomposition on Eq.~\eqref{eq:decomp_rho_rhotilde}, we have that
\begin{equation*}
    \begin{split}
        \sum_{t=1}^T \|\rho_t - \tilde{\rho}_t \|_1 &\leq \sum_{t=1}^{T-1} \alpha^{t-1} \bigg[3 N |\mathcal{X}| \sqrt{2 |\mathcal{A}| (T-t) \log\bigg(\frac{|\mathcal{X}| |\mathcal{A}| N T}{\delta} \bigg)} + 2 N |\mathcal{X}| \sqrt{2 (T-t) \log\bigg(\frac{N}{\delta} \bigg)}  \bigg] \\
        &\leq \frac{1}{1- \alpha} \bigg[3 N |\mathcal{X}| \sqrt{2 |\mathcal{A}| T \log\bigg(\frac{|\mathcal{X}| |\mathcal{A}| N T}{\delta} \bigg)} + 2 N |\mathcal{X}| \sqrt{2 T \log\bigg(\frac{N}{\delta} \bigg)}  \bigg],
    \end{split}
\end{equation*}
where for the last inequality we use that $0 \leq \alpha < 1$.

\end{proof}

\subsection{Auxiliary results}

We now present some auxiliary results that are used in establishing the main periodic regret bound for Algorithm~\ref{alg:main2} in Thm.~\ref{thm:main_periodic_regret_unknown_rhot}. These are mainly adaptations of the results from the previous sections to take into account the use of the estimated initial state-action distribution $\tilde{\rho}_t$ rather than the actual distribution $\rho_t$ when calculating the policy at episode $t$. 

\paragraph{New feasibility result:} We begin by presenting a new feasibility result. In Lemma~\ref{lemma:feasibility}, we showed that the problem in Eq.~\eqref{iteration_md_solver} is feasible when the sequence of distributions starts from the true state-action distribution $\rho_t$. However, since the learner now uses the estimated distribution $\tilde{\rho}_t$ to compute the policy, a corresponding feasibility result is required.

\begin{lemma}\label{lemma:new_feasibility}
    There exists $0 \leq \bar{\alpha} < 1$ such that for every episode $t \in [T]$, the optimization problem in Eq.~\eqref{eq:iterative_scheme_setting2} is feasible with high-probability.
\end{lemma}
\begin{proof}
Let $\pi$ be any periodic policy. Hence,
\begin{equation*}
\begin{split}
    \| \tilde{\rho}_t \hat{P}^t_\pi - \rho \|_1 &\leq \|\tilde{\rho}_t (\hat{P}^t_\pi - P_\pi) \|_1 + \| (\tilde{\rho}_t - \rho) P_\pi \|_1 \\
    &\leq \langle \hat{\mu}^{\pi, \tilde{\rho}_t}_t, b_t \rangle + \alpha \|\tilde{\rho}_t - \rho \|_1,
    \end{split}
\end{equation*}
where we use the high-probability result from Lemma~\ref{lemma:mu_bonus} and Ass.~\ref{ass:contraction}.
\end{proof}

\begin{corollary}\label{cor:bonus_estimate_rhot}
    Let $(b_t)_{t \in [T]}$ and $(\bar{b}_t)_{t \in [T]}$ be the two sequences of bonus vectors defined in Eq.~\eqref{eq:bonus}. Recall that $\hat{\mu}^{\pi, \nu}_t$ is the state-action distribution sequence induced in the estimated MDP $\hat{p}_t$, with policy $\pi$, with initial distribution $\nu$. Thus, for any $\delta \in (0,1)$, with probability at least $1-4\delta$,
    \[
    \sum_{t=1}^T \langle b_t, \hat{\mu}^{\pi_t, \tilde{\rho}_t}_t \rangle \leq \tilde{O}\bigg(\frac{ N^2}{1-\alpha} |\mathcal{X}|^{3/2} \sqrt{|\mathcal{A}| T}\bigg) \quad \quad \text{and} \quad \quad  \sum_{t=1}^T \langle \bar{b}_t, \hat{\mu}^{\pi_t, \tilde{\rho}_t}_t \rangle \leq   \tilde{O}\bigg(\frac{ \ell N^3}{1-\alpha} |\mathcal{X}|^{3/2} \sqrt{|\mathcal{A}| T}\bigg).
    \]
\end{corollary}
\begin{proof}
    We start by decomposing the product as
    \begin{equation*}
        \begin{split}
            \sum_{t=1}^T \langle b_t, \hat{\mu}^{\pi_t, \tilde{\rho}_t}_t \rangle = \underbrace{\sum_{t=1}^T \langle b_t, \hat{\mu}^{\pi_t, \tilde{\rho}_t}_t - \hat{\mu}^{\pi_t, \rho_t}_t \rangle}_{(i)} + \underbrace{\sum_{t=1}^T \langle b_t, \hat{\mu}^{\pi_t, \rho_t}_t \rangle}_{(ii)}. 
        \end{split}
    \end{equation*}
From Corollary~\ref{cor:bonus_analysis}, we have that for any $\delta \in (0,1)$, with probability $1-2 \delta$, $(ii) = \tilde{O}(N^2 |\mathcal{X}|^{3/2}\sqrt{|\mathcal{A}|T}).$ As for the first term, using Holder's inequality and that for each $n \in [N]$, $\|b_{t,n}\|_\infty \leq  C_\delta$, we get that
\begin{equation*}
    \begin{split}
        (i) &\leq C_\delta \sum_{t=1}^T \sum_{n=1}^N \| \hat{\mu}^{\pi_t, \tilde{\rho}_t}_{t,n} - \hat{\mu}^{\pi_t, \rho_t}_{t,n}  \|_1 \\
        &\underbrace{\leq}_{\text{Lemma~\ref{lemma:bound_norm_mu_diff_mu0}}} C_\delta N \sum_{t=1}^T \|\tilde{\rho}_t - \rho_t \|_1 \\
        &\underbrace{\leq}_{\text{Prop.~\ref{prop:estimate_rhot}}}  \tilde{O}\bigg(\frac{ N^2}{1-\alpha} |\mathcal{X}|^{3/2} \sqrt{|\mathcal{A}| T}\bigg).
    \end{split}
\end{equation*}

By summing both terms, we conclude the proof for the bonus vector sequence $(b_t)_{t \in [T]}$.

The result for the bonus vector sequence $(\bar{b}_t)_{t \in [T]}$ follows from a similar decomposition, the application of Corollary~\ref{cor:bonus_analysis}, and the observation that for each $n \in [N]$, $\|\bar{b}_{t,n}\|_\infty \leq \ell N C_\delta$.
\end{proof}

\paragraph{Distance between $\tilde{\rho}_t$ and $\rho$:} We now present a result analogous to Lemma~\ref{lemma:almost_equal_dist}, which shows that the estimated initial state-action distribution remains, on average, close to the target distribution $\rho$. This result is crucial for establishing a low periodic regret bound.
\begin{lemma}\label{lemma:diff_estimate_rhot_true_rho}
With high probability we have that
    \[
    \sum_{t=1}^T \|\tilde{\rho}_{t+1} - \rho \|_1 \leq \tilde{O}\bigg(\frac{ N^2}{(1- \alpha)^2} |\mathcal{X}|^{3/2} \sqrt{|\mathcal{A}| T} \bigg).
    \]
\end{lemma}
\begin{proof}
We denote by $\tilde{\mu}^{\pi, \nu}_t$ the sequence of state-action distributions induced by policy $\pi$ in the estimated MDP starting from $\nu$, according to the recursive relation in Eq.~\eqref{mu_induced_pi}, with initial state-action distribution $\nu$.

For short, we let $\tilde{\mu}_t := \tilde{\mu}^{ \pi_t, \tilde{\rho}_t}_t$, and $\mu_t := \hat{\mu}^{\pi_t, \tilde{\rho}_t}_t$. Additionally, we define a bonus vector based on the trajectory of the restarted agent; that is, for all $(n, x, a)$, we set:
\[
\tilde{b}_{t,n}(x,a) := \frac{C_\delta}{\sqrt{\max\{\widetilde{N}_{t,n}(x,a),1}\}}.
\]

Note that by definition, $\tilde{\rho}_{t+1} = \tilde{\rho}_t \widetilde{P}^t_{\pi_t}$, hence
    \begin{equation*}
        \begin{split}
            \|\tilde{\rho}_{t+1} - \rho \|_1 &= \| \tilde{\rho}_t \widetilde{P}^t_{\pi_t} - \rho \|_1 \\
            &\leq \|\tilde{\rho}_t ( \widetilde{P}^t_{\pi_t} - P_{\pi_t} ) \|_1 +  \|\tilde{\rho}_t ( \hat{P}_{\pi_t}^t - P_{\pi_t} ) \|_1 +  \|\tilde{\rho}_t \hat{P}^t_{\pi_t} - \rho \|_1 \\
            &\leq \underbrace{\langle \tilde{\mu}_{t}, \tilde{b}_t \rangle + 2 \langle \mu_t, b_t \rangle}_{:= \beta_t} + \alpha \|\tilde{\rho}_t - \rho \|_1 \\
            &\leq \beta_t + \alpha \big[ \beta_{t-1} + \alpha \|\tilde{\rho}_{t-1} - \rho \|_1 ] \\ 
            &\leq \sum_{s=1}^{t} \alpha^{t-s} \beta_s,
        \end{split}
    \end{equation*}
where for the second inequality we use that $\pi_t$ is a solution of the OMD iteration at episode $t$ in Eq.~\eqref{eq:iterative_scheme_setting2}, therefore satisfying the constraint $\|\tilde{\rho}_t \hat{P}^t_{\pi_t} - \rho \|_1 \leq \langle \mu_t, b_t \rangle + \alpha \|\tilde{\rho}_t - \rho \|_1$, and we also use Lemma~\ref{lemma:mu_bonus} for both the stochastic matrices $\hat{P}^t_{\pi_t}$ and $\widetilde{P}_{\pi_t}^t$. 

     For any $s \in [T]$, 
    \[
     \sum_{t=1}^s \beta_s = \underbrace{\sum_{t=1}^s \langle \tilde{\mu}_{t}, \tilde{b}_t \rangle}_{(i)}  + \underbrace{\sum_{t=1}^s 2 \langle \mu_t, b_t \rangle}_{(ii)}.
    \]
    We analyze each term individually.

    Term $(ii)$ can be upper bounded via a direct application of Corollary~\ref{cor:bonus_estimate_rhot}, which shows that
    \[
    (ii) = \sum_{t=1}^T \langle \mu_t, b_t \rangle \leq \tilde{O}\bigg(\frac{ N^2}{1- \alpha} |\mathcal{X}|^{3/2} \sqrt{|\mathcal{A}| s} \bigg).
    \]

    For term $(i)$, we apply Prop.~\ref{prop:mdp_martingale} to the trajectory of the restarted agent under policy $\pi_t$, \emph{i.e.}, $(\tilde{x}_{t,n}, \tilde{a}_{t,n})_{n \in [N]}$, using the sequence of initial state-action distributions $(\tilde{\rho}_t)_{t \in [T]}$ from which the restarted agent is reset at the beginning of each episode. We set $\xi_{t,n}(x,a) = \tilde{b}_{t,n}(x,a)$ for all $(n, x, a)$, and therefore take $C = c = C_\delta$. As a result, for any $\delta \in (0,1)$, we obtain that with probability at least $1 - \delta$,
    \begin{equation*}
        \begin{split}
            (i) = \sum_{t=1}^s \langle \tilde{\mu}_{t}, \tilde{b}_t \rangle &\leq 3 N C_\delta \sqrt{|\mathcal{X}| |\mathcal{A}| s } + N C_\delta |\mathcal{X}| \sqrt{2 s \log\bigg(\frac{N}{\delta}\bigg)} \\        
            &= 3 N |\mathcal{X}| \sqrt{ |\mathcal{A}| s \log\bigg( \frac{|\mathcal{X}| |\mathcal{A}| N s }{\delta} \bigg)} + N |\mathcal{X}|^{3/2}  \sqrt{2 s \log\bigg( \frac{|\mathcal{X}| |\mathcal{A}| N s }{\delta} \log\bigg(\frac{N}{\delta}\bigg)} \\
            &= \tilde{O}\big(N |\mathcal{X}|^{3/2} \sqrt{|\mathcal{A}| s} \big),
        \end{split}
    \end{equation*}
    where we use the definition of $C_\delta$ in Eq.~\eqref{eq:c_delta}.

    Joining the upper bounds on terms $(i)$ and $(ii)$ we obtain that
    \[
    \sum_{t=1}^s \beta_s \leq \tilde{O}\bigg(\frac{ N^2}{1- \alpha} |\mathcal{X}|^{3/2} \sqrt{|\mathcal{A}| s} \bigg).
    \]

    Summing over $t \in [T]$ we then get that
    \begin{equation*}
        \begin{split}
            \sum_{t=1}^T \|\tilde{\rho}_{t+1} - \rho \|_1 &\leq \sum_{t=1}^{T} \sum_{s=1}^t \alpha^{t-s} \beta_s \\
            &= \sum_{t=1}^{T-1} \alpha^{t-1} \sum_{s=1}^{T-t} \beta_s \\
            &\leq \sum_{t=1}^{T-1} \alpha^{t-1} \tilde{O}\bigg(\frac{ N^2}{1- \alpha} |\mathcal{X}|^{3/2} \sqrt{|\mathcal{A}| (T-t)} \bigg) \\
            &\leq \tilde{O}\bigg(\frac{N^2}{(1- \alpha)^2} |\mathcal{X}|^{3/2} \sqrt{|\mathcal{A}| T} \bigg),
        \end{split}
    \end{equation*}
    concluding the proof.
\end{proof}

As a consequence of Lemma~\ref{lemma:diff_estimate_rhot_true_rho}, Corollary~\ref{cor:diff_rho_rhot} establishes that the true state-action distribution $\rho_t$ remains, on average, close to the target distribution $\rho$, despite the fact that in Framework 2, we do not observe $\rho_t$ directly and instead rely on the estimate $\tilde{\rho}_t$ to compute the policies.

\begin{corollary}\label{cor:diff_rho_rhot}
With high probability, we have that
    \[
    \sum_{t=1}^T \|\rho_{t} - \rho \|_1 \leq  \tilde{O}\bigg(\frac{N^2}{(1-\alpha)^2} |\mathcal{X}|^{3/2} \sqrt{|\mathcal{A}| T} \bigg).
    \]
\end{corollary}
\begin{proof}
    The proof follows from the triangle inequality, then from Prop.~\ref{prop:estimate_rhot} and Lemma~\ref{lemma:diff_estimate_rhot_true_rho}:
    \begin{equation*}
        \begin{split}
            \sum_{t=1}^T \|\rho_{t} - \rho \|_1 &\leq \sum_{t=1}^T \|\rho_{t} - \tilde{\rho}_{t} \|_1  + \sum_{t=1}^T \|\tilde{\rho}_{t} - \rho \|_1 \\
            &\leq  \tilde{O}\bigg(\frac{N^2}{(1-\alpha)^2} |\mathcal{X}|^{3/2} \sqrt{|\mathcal{A}| T} \bigg).
        \end{split}
    \end{equation*}
\end{proof}

\paragraph{New online mirror descent proof:} We present here a result analogous to Prop.~\ref{prop:bound_md_term}, concerning the policy computed via online mirror descent (OMD), but now accounting for the fact that OMD is performed as in Eq.~\eqref{eq:iterative_scheme_setting2}, over the set of state-action distributions initialized at $\tilde{\rho}_t$ at the start of each episode.

\begin{lemma}\label{lemma:md_estimate_rho}
For each episode $t$, we let $\hat{\mu}^{\pi_t, \tilde{\rho}_t}_t$ be the solution of Eq.~\eqref{eq:iterative_scheme_setting2}. Suppose further that $\|\nabla \psi(\hat{\mu}^{\pi_t, \tilde{\rho}_t}_t) \|_{1, \infty} \leq \psi$. Then, with high probability,
            \[\sum_{t=1}^T \langle \tilde{\ell}_t - \bar{b}_t, \hat{\mu}^{\pi_t, \tilde{\rho}_t}_t - \hat{\mu}^{\pi, \tilde{\rho}_t}_t \rangle  = \tilde{O} \bigg( \frac{\ell N^2}{1-\alpha} |\mathcal{X}|^{5/4} |\mathcal{A}|^{1/4} \Psi^{1/2} T^{3/4} \bigg).
            \]
\end{lemma}

For $\pi$ a periodic policy, we apply Lemma~\ref{lemma:aux_md} with $\nu_t = \hat{\mu}^{\pi, \tilde{\rho}_t}_t$. We take $q_t = \hat{p}_t$, and $\rho_t = \tilde{\rho}_t$, with $\rho_1 = \rho$, such that $\widetilde{\mathcal{M}}_t = \mathcal{M}_{\tilde{\rho}_t}^{\hat{p}_t}$ following the definition in Eq.~\eqref{eq:bellman_flow}. Note that for
\[
\widetilde{\mathcal{W}}_t := \{ \mu := (\mu_n)_{n \in [N]} \; | \; \| \mu_N - \rho \|_1 \leq \langle \mu, b_t \rangle + \alpha \|\tilde{\rho}^t - \rho\|_1 \},
\]
we have $\nu_t \in \widetilde{\mathcal{W}}_t \cap \widetilde{\mathcal{M}}_t$ from the new feasibility result in Lemma~\ref{lemma:new_feasibility}. Recall that $f_{t,n}$ is $\ell$-Lipschitz with respect to the norm $\|\cdot\|_1$, hence $\|\tilde{\ell}_{t,n}\|_\infty \leq \ell$ and $\|\bar{b}_{t,n}\|_\infty \leq b := \ell N C_\delta$ for all $(n,t)$. As $\hat{\mu}^{\pi_t, \tilde{\rho}_t}_t$ is solution to Eq.~\eqref{eq:iterative_scheme_setting2}, we have that 
\begin{equation*}
    \begin{split}
        \sum_{t=1}^T \langle \tilde{\ell}_t - \bar{b}_t, \hat{\mu}^{\pi_t, \tilde{\rho}_t}_t - \hat{\mu}^{\pi, \tilde{\rho}_t}_t \rangle &\leq \eta (2 \ell N C_\delta)^2 T + \frac{2}{\eta} N e \log(T) \mathds{1}_{\{ D_\psi = \Gamma \}}  - \frac{\psi(\mu^1)}{\eta} \\
          &+ \frac{1}{\eta} \underbrace{\sum_{t=1}^T \langle \nabla \psi(\mu_t), \nu_t - \nu_{t+1} \rangle}_{(i)} +  \underbrace{\sum_{t=1}^T \langle \tilde{\ell}^t - \bar{b}_t, \nu_{t+1} - \nu_t \rangle}_{(ii)},
    \end{split}
\end{equation*}
  where we also use that for all $(n,x,a)$, $V_T := 1 + \sum_{t=1}^T \|\hat{p}^t_n(\cdot|x,a) - \hat{p}^{t+1}_n(\cdot|x,a) \|_1 \leq e \log(T)$ from Lemmas~\ref{lemma:difference_consecutive_p} and~\ref{lemma:kernel_diff_two_episodes}. The analysis of terms $(i)$ and $(ii)$ is similar as in the proof of Prop.~\ref{prop:bound_md_term}, except that we now need to account for the difference between two consecutive estimates of the initial state-action distribution:
  \begin{equation}\label{eq:tilde_rho_consecutive}
      \begin{split}
          \sum_{t=1}^T \|\tilde{\rho}_t - \tilde{\rho}_{t+1} \|_1 &\leq \sum_{t=1}^T \|\tilde{\rho}_t - \rho \|_1 + \sum_{t=1}^T \|\rho - \tilde{\rho}_{t+1} \|_1 \\
          &\underbrace{\leq}_{\text{Lemma~\ref{lemma:diff_estimate_rhot_true_rho}}} \tilde{O} \bigg( \frac{N^2 }{(1-\alpha)^2} |\mathcal{X}|^{3/2} \sqrt{|\mathcal{A}| T} \bigg).
      \end{split}
  \end{equation}

Thus, by repeating the analysis of terms $(i)$ and $(ii)$ from the proof of Prop.~\ref{prop:bound_md_term}, but replacing occurrences of $\|\rho_t - \rho_{t+1}\|_1$ with $\|\tilde{\rho}_t - \tilde{\rho}_{t+1}\|_1$, and applying the upper bound from Eq.~\eqref{eq:tilde_rho_consecutive}, we obtain that
\[
(i) \leq \Psi \bigg[ N e \log(T) + \tilde{O} \bigg( \frac{N^2 }{(1-\alpha)^2} |\mathcal{X}|^{3/2} \sqrt{|\mathcal{A}| T} \bigg)\bigg],
\]
and 
\[
(ii) \leq 2 \ell N C_\delta \bigg[N e \log(T) + \tilde{O} \bigg( \frac{N^2 }{(1-\alpha)^2} |\mathcal{X}|^{3/2} \sqrt{|\mathcal{A}| T} \bigg)\bigg].
\]

Therefore, joining the bounds of terms $(i)$ and $(ii)$, and optimizing over $\eta$, we obtain that
\begin{equation*}
    \begin{split}
        \sum_{t=1}^T \langle \tilde{\ell}_t - \bar{b}_t, \hat{\mu}^{\pi_t, \tilde{\rho}_t}_t - \hat{\mu}^{\pi, \tilde{\rho}_t}_t \rangle  = \tilde{O} \bigg( \frac{\ell N^3}{1-\alpha} |\mathcal{X}|^{5/4} |\mathcal{A}|^{1/4} \Psi^{1/2} T^{3/4} \bigg). \\
    \end{split}
\end{equation*}

\subsection{Final periodic regret bound: proof of Theorem~\ref{thm:main_periodic_regret_unknown_rhot}}\label{proof:frame2}

\begin{proof}
We begin by decomposing the periodic regret as follows: for any periodic policy $\pi$, 

\begin{equation*}
    \begin{split}
        R_T(\pi) &= \sum_{t=1}^T \big[ F_t(\mu^{\pi_t, \rho_t}) - F_t(\mu^{\pi, \rho}) \big] + \sum_{t=1}^T \|\rho_t - \rho\|_1 \\
        &= \underbrace{\sum_{t=1}^T \big[ F_t(\mu^{\pi_t, \rho_t}) - F_t(\hat{\mu}^{\pi_t, \rho_t}_t) \big]}_{R_T^{\text{MDP}}}  
        + \underbrace{\sum_{t=1}^T \big[ F_t(\hat{\mu}^{\pi_t, \rho_t}_t) - F_t(\hat{\mu}^{\pi_t, \tilde{\rho}_t}_t) \big]}_{R_T^{\text{diff. } \tilde{\rho}_t - \rho_t}} \\
        &+ \underbrace{\sum_{t=1}^T \big[ F_t(\hat{\mu}^{\pi_t, \tilde{\rho}_t}_t) - F_t(\mu^{\pi, \rho}) \big]}_{R_T^{\text{policy}}} + \gamma \sum_{t=1}^T \|\rho_t - \rho\|_1 .
    \end{split}
\end{equation*}

We analyze each term individually
\paragraph{$R_T^{\text{MDP}}$:} The results of Prop.~\ref{prop:R_T_mdp} still apply to this term, hence for any $\delta \in (0,1)$, with probability at least $1-2\delta$,
    \[
    R_T^{\text{MDP}} \leq \ell N^2 \bigg( 3 \sqrt{2} |\mathcal{X}| \sqrt{ |\mathcal{A}| T \log\bigg(\frac{|\mathcal{X}| |\mathcal{A}| N T }{\delta}}\bigg) + 2 |\mathcal{X}| \sqrt{2 T \log\bigg(\frac{N}{\delta}\bigg)} \bigg).
    \]

\paragraph{$R_T^{\text{diff. } \tilde{\rho}_t - \rho_t}$:} Using the convexity of $F_t$, Holder's inequality, and that $f_{t,n}$ is $\ell$-Lipschitz with respect to the norm $\|\cdot\|_1$, we have that
\begin{equation*}
    \begin{split}
        R_T^{\text{diff. } \tilde{\rho}_t - \rho_t} &= \sum_{t=1}^T  F_t(\hat{\mu}^{\pi_t, \rho_t}_t) - F_t(\hat{\mu}^{\pi_t, \tilde{\rho}_t}_t) \\
        &\leq \sum_{t=1}^T \sum_{n=1}^N \langle \nabla f_{t,n}(\hat{\mu}^{\pi_t, \rho_t}_{t,n}), \hat{\mu}^{\pi_t, \rho_t}_{t,n} - \hat{\mu}^{\pi_t, \tilde{\rho}_t}_{t,n} \rangle \\
        &\leq \ell \sum_{t=1}^T \sum_{n=1}^N \|\hat{\mu}^{\pi_t, \tilde{\rho}_t}_{t,n} - \hat{\mu}^{\pi_t, \rho_t}_{t,n} \|_1 \\
        &\underbrace{\leq}_{\text{Lemma~\ref{lemma:bound_norm_mu_diff_mu0}}} \ell N \sum_{t=1}^T \|\rho_t - \tilde{\rho}_t \|_1 \\
        &\underbrace{\leq}_{\text{Prop.~\ref{prop:estimate_rhot}}} \frac{\ell N}{1- \alpha} \bigg[3 N |\mathcal{X}| \sqrt{2 |\mathcal{A}| T \log\bigg(\frac{|\mathcal{X}| |\mathcal{A}| N T}{\delta} \bigg)} + 2 N |\mathcal{X}| \sqrt{2 T \log\bigg(\frac{N}{\delta} \bigg)}  \bigg] \\
        &= \tilde{O}\bigg( \frac{\ell N^2}{1-\alpha} |\mathcal{X}| \sqrt{|\mathcal{A}| T} \bigg).
    \end{split}
\end{equation*}

\paragraph{Regularization term}: By applying Corollary~\ref{cor:diff_rho_rhot}
\[
\sum_{t=1}^T \|\rho_{t} - \rho \|_1 \leq  \tilde{O}\bigg(\frac{\gamma N^2}{(1-\alpha)^2} |\mathcal{X}|^{3/2} \sqrt{|\mathcal{A}| T} \bigg).
\]

\paragraph{$R_T^{\text{policy}}$:} Using the convexity of $F_t$, and defining $\tilde{\ell}_t := \nabla F_t(\hat{\mu}^{\pi_t, \tilde{\rho}_t}_t)$, we decompose this term in three parts:
\begin{equation*}
    \begin{split}
        R_T^{\text{policy}} &\leq \sum_{t=1}^T \langle \tilde{\ell}_t , \hat{\mu}^{\pi_t, \tilde{\rho}_t}_t - \mu^{\pi, \rho} \rangle \\
        &\leq \underbrace{\sum_{t=1}^T \langle \tilde{\ell}_t - \bar{b}_t, \hat{\mu}^{\pi_t, \tilde{\rho}_t}_t - \hat{\mu}^{\pi, \tilde{\rho}_t}_t \rangle}_{R_T^{\text{MD}}} + \underbrace{\sum_{t=1}^T \langle \tilde{\ell}_t - \bar{b}_t, \hat{\mu}^{\pi, \tilde{\rho}_t}_t - \hat{\mu}^{\pi, \rho}_t \rangle}_{R_T^{\text{diff. } \tilde{\rho}_t - \rho}} \\
        &+ \underbrace{\sum_{t=1}^T \langle \bar{b}_t, \hat{\mu}^{\pi_t, \tilde{\rho}_t}_t - \hat{\mu}^{\pi, \rho}_t \rangle + \sum_{t=1}^T \langle \tilde{\ell}_t, \hat{\mu}^{\pi, \rho}_t - \mu^{\pi, \rho} \rangle}_{R_T^{\text{bonus}}}.
    \end{split}
\end{equation*}

We again break the analysis of each term:
\begin{itemize}
    \item $R_T^{\text{MD}}$: The bound for this term follows directly from Lemma~\ref{lemma:md_estimate_rho}, which gives us that with high probability
    \[
    R_T^{\text{MD}} \leq \tilde{O} \bigg( \frac{\ell N^3}{1-\alpha} |\mathcal{X}|^{5/4} |\mathcal{A}|^{1/4} \Psi^{1/2} T^{3/4} \bigg).
    \]

    \item $R_T^{\text{diff. } \tilde{\rho}_t - \rho}$:  Using Holder's inequality, and that $\|\tilde{\ell}_t - \bar{b}_t\|_\infty \leq 2 \ell N C_\delta$, we have that
\begin{equation*}
    \begin{split}
        R_T^{\text{diff. } \tilde{\rho}_t - \rho} &=\sum_{t=1}^T \langle \tilde{\ell}_t - \bar{b}_t, \hat{\mu}^{\pi, \tilde{\rho}_t}_t - \hat{\mu}^{\pi, \rho}_t \rangle\\
        &\leq 2 \ell N C_\delta \sum_{t=1}^T \sum_{n=1}^N \| \hat{\mu}^{\pi, \tilde{\rho}_t}_t - \hat{\mu}^{\pi, \rho}_t \|_1 \\
        &\underbrace{\leq}_{\text{Lemma~\ref{lemma:bound_norm_mu_diff_mu0}}} 2 \ell N^2 C_\delta \sum_{t=1}^T \| \tilde{\rho}_t - \rho \|_1 \\
        &\underbrace{\leq}_{\text{Lemma.~\ref{lemma:diff_estimate_rhot_true_rho}}} \tilde{O} \bigg( \frac{\ell N^4}{(1- \alpha)^2} |\mathcal{X}|^2 \sqrt{|\mathcal{A}| T} \bigg).
    \end{split}
\end{equation*}

\item{$R_T^{\text{bonus}}:$} Following the proof of Prop.~\ref{prop:R_T_bonus}, we obtain the following decomposition:
\begin{equation*}
    \begin{split}
        R_T^{\text{bonus}} &= \sum_{t=1}^T \langle \bar{b}_t, \hat{\mu}^{\pi_t, \tilde{\rho}_t}_t - \hat{\mu}^{\pi, \rho}_t \rangle + \sum_{t=1}^T \langle \tilde{\ell}_t, \hat{\mu}^{\pi, \rho}_t - \mu^{\pi, \rho} \rangle \\
        &\leq \sum_{t=1}^T \langle \bar{b}_t, \hat{\mu}^{\pi_t, \tilde{\rho}_t}_t - \hat{\mu}^{\pi, \rho}_t \rangle + \sum_{t=1}^T \langle \bar{b}_t, \hat{\mu}^{\pi, \rho}_t \rangle + \sum_{t=1}^T \langle \bar{b}_{t,0}, \rho \rangle \\
        &= \underbrace{\sum_{t=1}^T \langle \bar{b}_t, \hat{\mu}^{\pi_t \tilde{\rho}_t}_t \rangle + \langle \bar{b}_{t,0}, \tilde{\rho}_t \rangle}_{(i)}+ \underbrace{\sum_{t=1}^T \langle \bar{b}_{t,0}, \rho - \tilde{\rho}_t \rangle}_{(ii)}.
    \end{split}
\end{equation*}

We begin by upper bounding term $(i)$. Since, by definition, $\hat{\mu}^{\pi_t, \tilde{\rho}_t}_{t,0} = \tilde{\rho}_t$, Corollary~\ref{cor:bonus_estimate_rhot} implies that
\[
(i) \leq \tilde{O} \bigg(\frac{\ell N^3}{1-\alpha} |\mathcal{X}|^{3/2} \sqrt{|\mathcal{A}| T} \bigg).
\]

As for term $(ii)$, using Holder's inequality, that $\|\bar{b}_{t,0} \|_1 \leq \ell N C_\delta$, and Lemma~\ref{lemma:diff_estimate_rhot_true_rho}, we have that
\begin{equation*}
    \begin{split}
        (ii) &= \sum_{t=1}^T \langle \bar{b}_{t,0}, \rho - \tilde{\rho}_t \rangle \\
        &\leq \sum_{t=1}^T \|\bar{b}_{t,0} \|_\infty \|\rho - \tilde{\rho}_t \|_1 \\
        &\leq \ell N C_\delta \sum_{t=1}^T \|\rho - \tilde{\rho}_t \|_1 \\
        &\leq \ell N C_\delta \tilde{O} \bigg( \frac{N^2}{(1-\alpha)^2} |\mathcal{X}|^{3/2} \sqrt{|\mathcal{A}| T} \bigg) \\
        &= \tilde{O} \bigg( \frac{\ell N^3}{(1-\alpha)^2} |\mathcal{X}|^2 \sqrt{|\mathcal{A}| T} \bigg).
    \end{split}
\end{equation*}

Joining the upper bounds on terms $(i)$ and $(ii)$ we obtain that
\[
R_T^{\text{bonus}} = \tilde{O} \bigg( \frac{\ell N^3}{(1-\alpha)^2} |\mathcal{X}|^2 \sqrt{|\mathcal{A}| T} \bigg).
\]

\end{itemize}

\paragraph{Conclusion:} By summing the upper bounds from the terms of the periodic regret decomposition above, we obtain the final upper bound on the periodic regret of Algorithm~\ref{alg:main2}.
 
 \end{proof}


\end{document}